\makeatletter
\def\thanks#1{\protected@xdef\@thanks{\@thanks
        \protect\footnotetext{#1}}}
\makeatother

\documentclass{article}

\usepackage{authblk}
\usepackage{amsmath}
\usepackage{amsfonts,amssymb}
\usepackage{mathrsfs}
\usepackage{ntheorem}
\usepackage{indentfirst}
\usepackage{cite}
\usepackage{geometry}
\usepackage{graphicx}
\usepackage{subfigure}
\usepackage{epstopdf}
\usepackage{authblk}
\usepackage{float}
\usepackage{color} 
\newtheorem{theorem}{Theorem}[section]
\newtheorem{lemma}{Lemma}[section]
\newtheorem{proposition}{Proposition}[section]

\theoremstyle{definition}
\newtheorem{definition}{Definition}[section]
\newtheorem{example}{Example}[section]

\theoremstyle{remark}
\newtheorem{remark}{Remark}[section]

\newenvironment{proof}[1][Proof]{\noindent\textbf{#1.} }

\begin{document}
\title{\textbf{Linear-Quadratic Large-Population Problem with Partial Information: Hamiltonian Approach and Riccati Approach}
\thanks{This work was funded by the National Natural Science Foundation of China (12022108,11971267,11831010,61961160732), Shandong Provincial Natural Science Foundation (ZR2019ZD42), the Taishan Scholars Climbing Program of Shandong (TSPD20210302), and by the Distinguished Young Scholars Program of Shandong University. (\emph{Corresponding author: Tianyang Nie}.)}
}

\author{Min Li, ~Tianyang Nie ~and~ Zhen Wu \thanks{M.~Li, T.~Nie and Z. Wu are with the School of Mathematics, Shandong University, Jinan 250100, China. Email: liminmath@outlook.com (M. Li); nietianyang@sdu.edu.cn (T. Nie); wuzhen@sdu.edu.cn (Z. Wu).}}

\date{}

\maketitle

\begin{abstract}
This paper studies a class of partial information linear-quadratic  mean-field game problems. A general stochastic large-population system is considered, where the diffusion term of the dynamic of each agent can depend on the state and control. We study both the control constrained case and unconstrained case. In control constrained case, by using Hamiltonian approach and convex analysis, the explicit decentralized strategies can be obtained through projection operator.  The corresponding Hamiltonian type consistency condition  system is derived, which turns out to be a nonlinear mean-field forward-backward stochastic differential equation with projection operator. The well-posedness of such kind of equations is proved by using discounting method. Moreover, the corresponding $\varepsilon$-Nash equilibrium property is verified. In control unconstrained case, the decentralized strategies can be further represented explicitly as the feedback of filtered state through Riccati approach. The existence and uniqueness of a solution to a new Riccati type consistency condition system is also discussed. As an application, a general inter-bank borrowing and lending problem is studied to illustrate that the effect of partial information cannot be ignored.
 \end{abstract}



 \section{Introduction}
Over the decades, the research on dynamic optimization for stochastic large-population systems has already received extensive attentions. Distinguished from the single one, large-population system has considerable agents, which is broadly applied in the fields of engineering, finance, social science and so on.  In this context, the single individual agent effect is insignificant and negligible while the collective behaviors of the whole population cannot be ignored.  All the agents are weakly coupled via the state average or empirical distribution in dynamics and/or cost functionals. For these reasons, the centralized strategies of the given agent, which are based on the information of all peers, are infeasible. An efficient approach is to discuss the associated mean-filed game (MFG) to determine an approximate equilibrium by analyzing the related limiting behavior. Roughly speaking, the main idea is to study the related control problem by considering its own individual state and some off-line quantities. The introduction of some frozen quantities enables the highly complicated coupling problem to be a decoupled one, which is in essence a stochastic optimal control problem.  As a consequence, by the well-known results such as stochastic maximum principle and dynamic programming principle, the decentralized strategies are obtained, which only rely on the local information. To complete the procedure, the off-line quantities will be computed by some consistency condition (CC) systems.

MFG problems were originally formulated by Lasry and Lions \cite{LarsyLion20061,LarsyLion20062,LarsyLion2007} and independently by  Huang, Malham\'e and Caines \cite{Huang2006,Huang2007}. Indeed, motivated by phenomena in statistical mechanics and physics, Lasry and Lions \cite{LarsyLion20061,LarsyLion20062,LarsyLion2007} (see also Cardaliaguet \cite{Cardaliaguet2010}) were concerned with situations that involved a large number of rational players and obtained distributed closed-loop strategies, which were represented by the forward-backward systems of a Hamilton-Jacobi-Bellman (HJB) equation and a Fokker-Planck equation. Huang, Malham\'e and Caines \cite{Huang2006,Huang2007} introduced Nash Certainty Equivalence (NCE) to handle MFG problems from the perspective of engineering applications, which led to decentralized control synthesis. For further research on MFG problems and related topics, the readers are referred to the papers e.g. \cite{Buckdahn2009,Huang2012,Cardaliaguet2013,Carmona2013,
Bardi2014,Moon2017,Hu2018,Nie2018,Li2020,Wang2020} and the monographs by Bensoussan \cite{Bensoussan2013}, and by Carmona and Delarue \cite{Carmona2018}.

We mention that almost above listed literatures  are built on the fact that agents have access to full information. However, in many real applications, only partial information can be acquired by agents. The research of partial information has far-fetching theoretical significance and extensive application value, thus large-population problems with partial information have attracted researchers' intensive attentions. For example, Huang, Caines and Malham\'e \cite{Caines2006} considered dynamic games in a large population of stochastic agents, where agents had local noisy
measurements of its own state. Huang and Wang \cite{Huang2016} studied a class of dynamic optimization problems for large-population system with partial information. \c{S}en and Caines \cite{Caines2016} investigated MFG theory with partially observed major agent in both linear and nonlinear formulations. Huang, Wang and Wu \cite{Wang2016} considered the MFG problem for backward stochastic systems with partial information. \c{S}en and Caines \cite{Caines2019} studied partially observed stochastic large-population problem with nonlinear dynamics and nonlinear cost functionals. Bensoussan, Feng and Huang \cite{Bensoussan2019} focused on a class of linear-quadratic-gaussian (LQG)
MFGs with partial observation and common noise. To make some comparisons, another relevant but essentially different topic is the mean-field type control problem. For the literature review, some associated works on partial information mean-field control problem are described. For example, Wang, Zhang and Zhang \cite{Wang2013} proposed the stochastic maximum principle of mean-field type optimal control problem with partial information. Buckdahn, Li and Ma \cite{Buckdahn2017} explored a new type of mean-field, non-Markovian stochastic control problems with partial observations. %Wang, Xiao and Xing \cite{Wang2017} studied an optimal control problem derived by mean-field FBSDE with noisy observation.
For further studies, the readers can see \cite{Hafayed2015,Ma2016,Fu2019,NieYan2022,Wang2017} and the reference therein.%  Hafayed, Abbas and Abba \cite{Hafayed2015} examined a class of mean-field-type partial information stochastic optimal control problems driven by L\'evy process. Ma and Liu \cite{Ma2016} concerned with linear-quadratic (LQ) optimal control problem for partially observed forward-backward stochastic differential equations (FBSDEs) of mean-field type.  Li and Fu \cite{Fu2019} discussed a partially observed optimal control problem described by nonlinear mean-field FBSDEs.

It is widely acknowledged that the control constrained problem appears frequently in the fields of finance and economics. For example, in many financial models, control variables have some restrictions, such as taking only nonnegative values. A typical example is the mean-variance portfolio selection problem with no-shorting, see. e.g. Li, Zhou and Lim \cite{Li2002}. Moreover, control constrained problems have been applied to many other fields, such as aeronautics, artificial intelligence, network communication. As for the control constrained stochastic LQ control problem with random coefficients, Hu and Zhou \cite{Hu2005} derived the explicit optimal control and optimal cost through two extended stochastic Riccati equations. Pu and Zhang \cite{Pu2019} generalized it to the infinite time horizon case. Recently, Hu, Shi and Xu \cite{Hu2020} explored a control constrained stochastic LQ control problem with regime-switching, which can better reflect the random environment. Concerning the stochastic large-population problem with control constraints, to our best knowledge, it is really a new topic and there are few literatures to give the explicit decentralized strategies. Indeed,  Hu, Huang and Li \cite{Hu2018} studied a class of control constrained stochastic large-population problems with uniform minor agents where the individual control was constrained in a closed convex set. Hu, Huang and Nie \cite{Nie2018} investigated a class of LQG mixed MFGs with heterogenous input constraints, where a major agent and numerous heterogenous minor agents were embraced. In \cite{Hu2018,Nie2018}, the decentralized strategies are given explicitly through projection operators and CC system, which is a kind of nonlinear mean-field forward-backward stochastic differential equation (MF-FBSDE) with projection operator. See also \cite{ZhangLi2019} for extended works.

The current paper focuses on a class of general stochastic large-population problems with partial information, where the diffusion term of the dynamics of each agent can depend both on the state and control. To illustrate our motivations, we present the following example, which generalizes the model of Carmona, Fouque and Sun \cite{Carmona2015}.

\begin{example}\label{example}$($General Inter-Bank Borrowing and Lending Problem$)$ Considering a model consists of $N$ banks, which can lend to and borrow from each other. Meanwhile, there exist cash flows between each bank and the central bank. We denote $x_i(\cdot)$ the log-monetary reserve of bank $i$ lending to and borrowing from each other  and $u_i(\cdot)$ the corresponding control rate of borrowing from or lending to the central bank. Suppose the evolution of  $x_i(\cdot)$ is described by $($see \cite{Carmona2015} for some simple cases$)$
\begin{equation*}
\left\{
\begin{aligned}
dx_i(t)=&[Ax_i(t)+\frac{a}{N}\underset{j=1}{\overset{N}{\sum}}(x_j(t)-x_i(t))+Bu_i(t)+b]dt\\
&+[Cx_i(t)+Du_i(t)+\sigma]dW_i(t)+[\widetilde{D}u_i(t)+\widetilde{\sigma}]d\widetilde{W}_i(t),\\
x_i(0)=&~x,
\end{aligned}
\right.
\end{equation*}
where  $\{W_i(\cdot), 1\leq i \leq N\}$ and $\{\widetilde{W}_i(\cdot), 1\leq i \leq N\}$ are  two $N$-dimensional mutually independent standard Brownian motions. Here, $W_i$ and $\widetilde{W}_i$, $1\leq i \leq N$ are introduced to represent the noises of bank $i$. Due to some practical phenomena like the physical inaccessibility of some parameters, inaccuracies in measurement, discreteness of account information, bank $i$ can only observe $W_i$.   The average term $\frac{a}{N}\sum_{j=1}^N(x_j(\cdot)-x_i(\cdot))$ with $a\geq 0$ characterizes the interaction, which represents the rate of bank $i$ borrowing from or lending to other banks.
Let $x^{(N)}(\cdot)=\frac{1}{N}\sum_{j=1}^N x_{j}(\cdot)$, then we have
\begin{equation}
\left\{
\begin{aligned}\label{exstate}
dx_i(t)=&[(A-a)x_i(t)+Bu_i(t)+ax^{(N)}(t)+b]dt\\
&+[Cx_i(t)+Du_i(t)+\sigma]dW_i(t)+[\widetilde{D}u_i(t)+\widetilde{\sigma}]d\widetilde{W}_i(t),\\
x_i(0)=&~x.
\end{aligned}
\right.
\end{equation}
For $1\leq i \leq N$, bank $i$ controls its rate of borrowing from or lending to the central bank by choosing $u_i(\cdot)$ to minimize
\begin{equation}\label{excost}
\mathcal{J}_i(u(\cdot))=\frac{1}{2}\mathbb{E}\Big\{\int_0^T[\epsilon(x_i(t)-x^{(N)}(t))^2+ru_i(t)^2]dt
+c(x_i(T)-x^{(N)}(T))^2\Big\},
\end{equation}
where $u(\cdot)=(u_1(\cdot),\ldots,u_N(\cdot))$ and $\epsilon$, $r$ and $c$ are all constants with $\epsilon,c\geq 0$, $r>0$. The quadratic terms $(x_i(\cdot)-x^{(N)}(\cdot))^2$ in the running cost  and in terminal cost penalize the departure from the average. The quadratic terms $u_i(\cdot)^2$ denotes the cost of control process. Here, $u_i(\cdot)>0$ $($resp. $u_i(\cdot)<0$$)$ means that bank $i$ borrows from $($resp. lends to$)$ the central bank. Generally, the central bank does not need to absorb the funding from local ones, which suggests that $u_i(\cdot)$ often takes positive values.  Consequently, the general inter-bank borrowing and lending problem formulates a class of control constrained LQ large-population problems with partial information.
\end{example}

Inspired by above discussions, we consider a class of LQ large-population problems with partial information, where both the control constrained case and control unconstrained case are studied.
It should be noted that Carmona, Fouque and Sun \cite{Carmona2015} studied an inter-bank borrowing and lending model, which was by nature a LQ large-population problem without control constraints, and they obtained the decentralized strategies in form of state feedback via Riccati equation. As its promotion, we consider a general problem as in Example \ref{example}, where $x_i(\cdot)$ and $u_i(\cdot)$ both affect the drift and diffusion terms. Moreover, partial information structure and control constrained are also taken into consideration. Using Hamitonian approach, we can give the explicit decentralized strategies in control constrained case. When the control constraint is removed as in \cite{Carmona2015,Huang2016}, by using Riccati approach, we can represent the decentralized strategies as the feedback of filtered state.

To make the readers quickly grasp the spirit of ideas in
large-population problem, we begin with comparisons between our large-population problem and the classical stochastic differential games considered in many literatures, see e.g.
\cite{Moon2020, NieWangYu2021,Yu2015} and the reference therein. From the framework setting,  all players in the classical differential game control a common state system while in our large-population problem, each player possesses its own individual state.  Furthermore, one usually looks for the Nash equilibrium in classical differential game while
the large-population problem will result some approximate ones, since it is not implementable and not efficient to solve the large-population problem directly due to the highly complex interactions among individuals. Alternatively, the MFG provides one effective approach to determine an approximate equilibrium. In fact, the corresponding limiting problem turns out to be a class of LQ stochastic optimal control problems at bottom. 

We mention that even for the case without control constraint, our results are not trivial generalization of  \cite{Carmona2015, Huang2016}. In fact, compared with \cite{Carmona2015}, our stochastic large population problems are in the framework of partial information, although \cite{Huang2016} gave some results for stochastic large population problems with partial information, it cannot be applied to our case, since our  diffusion term can depend both on the state and control. The structure of partial information and the dependence of our diffusion term on state and control need subtle Riccati approach (as explained in the listed main contribution of the paper below).

The main contributions of this paper can be summarized as follows.
\begin{itemize}
  \item A general stochastic large-population problem with partial information is introduced. In control constrained case, by using Hamiltonian approach and convex analysis, the decentralized strategies can be given explicitly. The Hamiltonian type CC system turns out to be a new kind of  MF-FBSDEs with projection operator, where both the expectation term and conditional expectation term appear.  The well-posedness of such equation is proved by discounting method which is used to construct a contractive mapping.
  \item In control unconstrained case, by using Riccati approach, the decentralized strategies can be further represented explicitly as the feedback of filtered state. The optimal filtering equation, meanwhile, is established.
  \item A new Riccati type CC system is given (see \eqref{RCC}), which is not in the classical form. The well-posedness of this Riccati type CC system is obtained through some subtle analysis. In fact, the existence of $\widetilde{D}(\cdot)u_i(\cdot)$ in unobservable diffusion term brings essential difficulties to the solvability of Riccati type CC system. The first difficulty is that the additional term $\widetilde{D}^{\top}(\cdot)P(\cdot)\widetilde{D}(\cdot)$ cannot be combined with the original term $D^{\top}(\cdot)P(\cdot)D(\cdot)$ into a quadratic form (see \eqref{r1}), thus the first equation (equation for $P(\cdot)$) in system \eqref{RCC} is not a standard Riccati equation and then the existing results cannot be implemented. We can use modified iterative method (see \eqref{P0} and  \eqref{iterative}) and mathematical induction to prove the existence of a solution. The second difficulty is that the equation for $\Lambda(\cdot)$ in system \eqref{RCC} is not a standard Riccati equation too, since usually the inequality $\delta P(\cdot)-Q(\cdot)\geq 0$ fails. To overcome this difficulty, inspired by Yong \cite{Yong2013}, we transform the solvability of $\Lambda(\cdot)$ to the solvability of another Riccati equation which looks quite different to the transformed Riccati equation in \cite{Yong2013} due to the additional term  $\widetilde{D}^{\top}(\cdot)P(\cdot)\widetilde{D}(\cdot)$. It is interesting that we can use some algebraic inequalities (see e.g. inequality \eqref{positive}) to obtain also the well-posedness of our new transformed Riccati equation.
\end{itemize}

The rest of this paper is organized as follows. We formulate the LQ large-population problems with partial information in section \ref{sec:pre}. In section \ref{sec:cc}, we obtain the explicit decentralized strategies. Subsection \ref{subsec:1} devotes to the study of control constrained case. By using Hamiltonian approach, we get the explicit decentralized strategies through Hamiltonian type CC system, which is a new kind of MF-FBSDEs. The well-posedness results of MF-FBSDE is given and the corresponding $\varepsilon$-Nash equilibrium property is also verified. In subsection \ref{subsec:2}, we study the case without control constraint, by Riccati approach, the decentralized strategies can be further represented explicitly as the feedback of filtered state.  Subsection \ref{subsec:3} aims to showing the existence and uniqueness of a solution to Riccati type CC system.  In section \ref{sec:app}, the motivating example (Example \ref{example}) is solved. Section \ref{sec:con} concludes the paper. The well-poseness of a general kind of MF-FBSDEs is presented in Appendix A.

\section{Problem Formulation}\label{sec:pre}
Consider a large-population system which is composed of  $N$ agent $\mathcal{A}_i$, $1\leq i \leq N$. For a fixed $T>0$, let $(\Omega, \mathcal{F},\{\mathcal{F}_t\}_{0\leq t \leq T},\mathbb{P})$ be a completed filtered probability space satisfying usual conditions, on which we define two $N$-dimensional mutually independent standard Brownian motions $\{W_i(t), 1\leq i \leq N\}_{0\leq t\leq T}$ and $\{\widetilde{W}_i(t), 1\leq i \leq N\}_{0\leq t\leq T}$. Assume that $\mathcal{F}_t$ is the natural filtration generated by $\{W_i(s),\widetilde{W}_i(s),0\leq s \leq t, 1\leq i \leq N\}$ augmented by all $\mathbb{P}$-null set $\mathcal{N}$. Moreover, let $\mathcal{F}_t^i=\sigma\{W_i(s),\widetilde{W}_i(s),0\leq s \leq t\}\vee \mathcal{N}$, $\mathcal{G}_t=\{W_i(s),0\leq s \leq t, 1\leq i \leq N\}\vee \mathcal{N}$ and $\mathcal{G}_t^i=\{W_i(s),0\leq s \leq t\}\vee \mathcal{N}$.

Throughout the paper, we denote $\mathbb{R}^n$ the $n$-dimensional Euclidean space, with the usual norm and the usual inner product given by $|\cdot|$ and $\langle\cdot,\cdot\rangle$, respectively. For any vector or matrix, the superscript $\top$ denotes its transpose. $\mathcal{S}^d$ represents the set of $d\times d$-dimensional symmetric matrices, $\mathcal{S}_{+}^d$ represents the set of $d\times d$-dimensional symmetric matrices which are semi-positive,  and $I$ denotes the identity matrix. For any matrix $M\in \mathbb{R}^{n\times d}$, let $|M|=\sqrt{\textup{tr}(M^{\top}M)}$ denotes its norm, where $\textup{tr}(M^{\top}M)$ stands for the trace of $M^{\top}M$. If $M\in\mathcal{S}^d$ and $M\geq(>) 0$, we say that $M$ is semi-positive (positive) definite. For positive constant $k$, if $M\in\mathcal{S}^d$ and $M>kI$, we denote $M\gg0$. For any Euclidean space $\mathbb{M}$ and any filtration $\mathcal{V}$, if $h(\cdot):[0,T]\rightarrow \mathbb{M}$ is continuous, we denote $h(\cdot)\in C([0,T];\mathbb{M})$; if $h(\cdot):[0,T]\rightarrow \mathbb{M}$ is uniformly bounded, we denote $h(\cdot)\in L^\infty(0,T;\mathbb{M})$; if $h(\cdot):\Omega\times[0,T]\rightarrow \mathbb{M}$ is $\mathcal{V}_t$-adapted process s.t. $\mathbb{E}\int_0^T|h(t)|^2dt<\infty$, we denote $h(\cdot)\in L_{\mathcal{V}_t}^2(0,T;\mathbb{M})$.

Suppose the state of $i$-th agent $\mathcal{A}_i$ satisfies the following linear stochastic differential equation (SDE)
\begin{equation}\label{state}
\left\{
\begin{aligned}
dx_{i}(t)=&~[A(t)x_{i}(t)+B(t)u_{i}(t)+F(t)x^{(N)}(t)+b(t)]dt \\
&+[C(t)x_{i}(t)+D(t)u_{i}(t)+H(t)x^{(N)}(t)+\sigma (t)]dW_{i}(t) \\
&+[\widetilde{C}(t)x_{i}(t)+\widetilde{D}(t)u_{i}(t)+\widetilde{H}%
(t)x^{(N)}(t)+\widetilde{\sigma}(t)]d\widetilde{W}_{i}(t), \\
x_{i}(0)=&~x,%
\end{aligned}%
\right.
\end{equation}
where $x\in\mathbb{R}^n$ is the initial value, $x_i(\cdot)$ and $u_i(\cdot)$ denote the state process and control process, respectively. Moreover, we denote $x^{(N)}(\cdot):=\frac{1}{N}\sum_{i=1}^N x_{i}(\cdot)$, which characterizes the state average of all agents. The corresponding coefficients $A(\cdot),B(\cdot),F(\cdot),b(\cdot)$,\\$C(\cdot),D(\cdot),H(\cdot),\sigma(\cdot), \widetilde{C}(\cdot),\widetilde{D}(\cdot),\widetilde{H}(\cdot),\widetilde{\sigma}(\cdot)$ meet appropriate assumptions given later.

Let $\Gamma\subseteq \mathbb{R}^m$ be a nonempty closed convex set. We define centralized strategy set as $\mathcal{U}_{ad}^{c}=\{u_{i}(\cdot )~|~u_{i}(\cdot )\in L_{\mathcal{G}%
_{t}}^{2}(0,T;\Gamma )\}$ and decentralized strategy set as $\mathcal{U}_{ad}^{d,i}=\{u_{i}(\cdot )~|~u_{i}(\cdot )\in L_{\mathcal{G}%
_{t}^{i}}^{2}(0,T;\Gamma )\},~1\leq i \leq N$. Obviously, for $1\leq i \leq N$, $\mathcal{G}_t\subseteq\mathcal{F}_t$ and $\mathcal{G}_t^i\subseteq\mathcal{F}_t^i$, which means that our problem is in the setting of partial information. For simplicity, let $u(\cdot)=(u_1(\cdot),\ldots,u_N(\cdot))$ be the set of strategies of all agents and $u_{-i}(\cdot)=(u_1(\cdot),\ldots,u_{i-1}(\cdot),u_{i+1}(\cdot),\ldots,u_N(\cdot))$ be the set of strategies except for $i$-th agent. The cost functional of $\mathcal{A}_i$ is given by
\begin{equation}
\begin{aligned}
&\mathcal{J}_{i}(u_{i}(\cdot ),u_{-i}(\cdot ))\\
=&\frac{1}{2}\mathbb{E}%
\Big\{\int_{0}^{T}\big[\langle Q(t)(x_{i}(t)-x^{(N)}(t)),x_{i}(t)-x^{(N)}(t)\rangle
+\langle R(t)u_{i}(t),u_{i}(t)\rangle \big]dt \\
&\qquad\qquad+\langle G(x_{i}(T)-x^{(N)}(T)),x_{i}(T)-x^{(N)}(T)\rangle \Big\}\label{cost}.
\end{aligned}
\end{equation}%
Now, we aim to formulate the large-population problem with partial information.

\textbf{Problem (LP)} To choose strategy profile $\bar{u}(\cdot)=(\bar{u}_{1}(\cdot),\ldots ,\bar{u}_{N}(\cdot))$, where $\bar{u}_i(\cdot)\in \mathcal{U}_{ad}^{c} $, such that
\begin{equation*}
\mathcal{J}_{i}(\bar{u}_{i}(\cdot ),\bar{u}_{-i}(\cdot ))=\underset{%
u_{i}\left( \cdot \right) \in \mathcal{U}_{ad}^{c}}{\inf }\mathcal{J}%
_{i}(u_{i}(\cdot ),\bar{u}_{-i}(\cdot )),~~~1\leq i\leq N, \text{~~subjects to \eqref{state} and \eqref{cost}.}
\end{equation*}%
For $1\leq i\leq N$, if there exists such $\bar{u}_i(\cdot)$ satisfying above relationship, it's called the Nash equilibrium of Problem (LP). We denote the corresponding optimal state as $\bar{x}_i(\cdot)$. For further study, we provide the definition of $\varepsilon$-Nash equilibrium.
\begin{definition}
The strategy profile $\bar{u}(\cdot)=(\bar{u}_1(\cdot),\ldots,\bar{u}_N(\cdot))$, where $\bar{u}_i(\cdot)\in \mathcal{U}_{ad}^c$, $1\leq i \leq N$, is called an $\varepsilon$-Nash equilibrium with respect to the cost $\mathcal{J}_i$,  if there exists an $\varepsilon\geq 0$, such that for any $1\leq i \leq N$
\begin{equation*}
\mathcal{J}_{i}(\bar{u}_{i}(\cdot ),\bar{u}_{-i}(\cdot ))\leq \mathcal{J}_{i}(u_{i}(\cdot ),\bar{u}_{-i}(\cdot ))+\varepsilon,
\end{equation*}
where $u_{i}(\cdot )\in \mathcal{U}_{ad}^c$ is any alternative strategy applied by $\mathcal{A}_i$.
\end{definition}

\begin{remark}
Due to the coupling structure, the centralized strategies $(i.e.~u_i(\cdot)\in \mathcal{U}_{ad}^c )$  which need the information of other agents are difficult to obtain in the noncooperative game. In this paper, we use the MFG method to derive the decentralized strategies $(i.e.~u_i(\cdot)\in \mathcal{U}_{ad}^{d,i})$, which turn out to satisfy $\varepsilon$-Nash equilibrium property.
\end{remark}

We introduce the following assumptions for the coefficients. If no confusion happens, we will omit the dependence on time $t$.

\textup{(H1)} The coefficients of state equation satisfy
\begin{equation*}
A,C,\widetilde{C},F,H,\widetilde{H}, b,\sigma ,%
\widetilde{\sigma }\in L^{\infty }(0,T;\mathbb{R}^{n}),\quad
B,D,\widetilde{D}\in L^{\infty }(0,T;\mathbb{R}^{n\times m}),
\end{equation*}

\textup{(H2)} The coefficients of cost functional satisfy
\begin{equation*}
Q\in L^{\infty }(0,T;\mathcal{S}^{n}),\quad R\in L^{\infty }(0,T;\mathcal{S%
}^{m}),\quad G\in \mathcal{S}^{n},\quad
Q\geq 0,\quad R\gg0,\quad G\geq 0.
\end{equation*}

\begin{remark}
We mention that it is not enough to solve a LQ stochastic classical differential game problems under (H1) and (H2).  For example,   \cite{NieWangYu2021} used the stochastic maximum principle to study the corresponding Nash equilibrium, where a complex fully coupled FBSDE including one SDE and two BSDEs was involved. To guarantee the existence and uniqueness of a solution for this FBSDE, in addition to (H1)-(H2), some other assumptions (see Assumption 3 in \cite{NieWangYu2021}) were needed. However, the large-population problem is essentially different from the classical differential game as explained in the introduction. The MFG provides one effective
approach to determine an approximate equilibrium for large-population problem and the corresponding limiting problem is a LQ stochastic optimal control problem (see Section \ref{sec:cc}). Based on these observations, assumptions (H1) and (H2) are enough to solve our problem in some sense.
\end{remark}

\section{$\varepsilon$-Nash Equilibrium for Problem (LP)}\label{sec:cc}

As analyzed earlier, since the highly complicated coupling structure, it is tricky to solve the large-population problem directly. In other words, it is intractable and inefficient to obtain the centralized strategies of Problem (LP). One possible approach is to adopt MFG method to discuss corresponding decentralized strategies. In subsection \ref{subsec:1}, we use Hamiltonian approach to obtain the explicit decentralized strategies in control constrained case. The Hamiltonian type CC system is derived and its well-posedness is given. Moreover, we show that the decentralized strategies satisfy the $\varepsilon$-Nash equilibrium property. In subsection \ref{subsec:2}, we show that in control unconstrained case,  the decentralized strategies can be further represented explicitly as the feedback of filtered state through Riccati approach. The well-posedness of associated Riccati type CC system is explored.
\subsection{Control Constrained Case: Hamiltonian Approach}\label{subsec:1}
Let $\Gamma\subset\mathbb{R}^m$ be a nonempty closed convex set. When $\Gamma\neq\mathbb{R}^m$, Problem (LP) turns out to be a class of control constrained large-population problems with partial information. We will analyze the asymptotic behavior of large-population system when the number of agents tends to infinity. To start with, we suppose that the coupling term $\bar{x}^{N}(\cdot)=\frac{1}{N}\sum_{i=1}^N \bar{x}_{i}(\cdot)$ is approximated by $l(\cdot)$, which is a frozen limiting term and will be determined later. We introduce an auxiliary limiting system defined as
\begin{equation}
\left\{
\begin{aligned}
dz_{i}(t)=&~[A(t)z_{i}(t)+B(t)u_{i}(t)+F(t)l(t)+b(t)]dt \\
&+[C(t)z_{i}(t)+D(t)u_{i}(t)+H(t)l(t)+\sigma (t)]dW_{i}(t) \\
&+[\widetilde{C}(t)z_{i}(t)+\widetilde{D}(t)u_{i}(t)+\widetilde{H}(t)l(t)+%
\widetilde{\sigma }(t)]d\widetilde{W}_{i}(t), \\
z_{i}(0)=&~x,\label{lstate}%
\end{aligned}%
\right.
\end{equation}
and the limiting cost functional is given by
\begin{equation}
\begin{aligned}
J_i(u_{i}(\cdot )) =&\frac{1}{2}\mathbb{E}\Big\{\int_{0}^{T}\big[\langle
Q(t)(z_{i}(t)-l(t)),z_{i}(t)-l(t)\rangle +\langle
R(t)u_{i}(t),u_{i}(t)\rangle \big]dt \\
&\qquad\qquad+\langle G(z_{i}(T)-l(T)),z_{i}(T)-l(T)\rangle \Big\}\label{lcost}.
\end{aligned}
\end{equation}
Then, we formulate the limiting large-population problem with partial information as follows.

\textbf{Problem(LLP)} For each agent $\mathcal{A}_i$, $1\leq i \leq N$, to find $\bar{u}_{i}(\cdot )\in \mathcal{U}_{ad}^{d,i}$ such that
\begin{equation*}
J_{i}(\bar{u}_{i}(\cdot ))=\underset{u_{i}(\cdot )\in \mathcal{U}_{ad}^{d,i}}{\inf }J_{i}(u_{i}(\cdot )), \text{\qquad subjects to \eqref{lstate} and \eqref{lcost}.}
\end{equation*}%
If there exists  $\bar{u}_{i}(\cdot )\in \mathcal{U}_{ad}^{d,i}$ satisfying above relationship, it's the so-called decentralized strategy, and the corresponding $\bar{z}_i(\cdot)$ denotes the optimal decentralized trajectory. With the help of the frozen limiting term, the original large-population problem (LP) is transformed into a decoupled stochastic LQ control problem with partial information, which can be addressed with some classical methods of optimal control and filtering technique.\\
Indeed, we define the Hamiltonian function $\mathcal{H}_i:\Omega\times\mathbb{R}^n\times\mathbb{R}^n\times\mathbb{R}^n\times\mathbb{R}^n\times\Gamma\rightarrow\mathbb{R}$ by
\begin{equation*}
\begin{aligned}
&\mathcal{H}_{i}(t,p_{i},k_{i},\widetilde{k}_{i},z_{i},u_{i}) =\langle
p_{i},Az_{i}+Bu_{i}+Fl+b\rangle +\langle k_{i},Cz_{i}+Du_{i}+Hl+\sigma
\rangle \\
&\qquad\qquad \qquad+\langle \widetilde{k}_{i},\widetilde{C}z_{i}+\widetilde{D}u_{i}+%
\widetilde{H}l+\widetilde{\sigma }\rangle -\frac{1}{2}\langle
Q(z_{i}-l),z_{i}-l\rangle -\frac{1}{2}\langle Ru_{i},u_{i}\rangle,
\end{aligned}
\end{equation*}
and we introduce the following adjoint equation %for $(p_i(\cdot),k_i(\cdot),\widetilde{k}_i(\cdot))$
\begin{equation}
\left\{
\begin{aligned}
dp_{i}(t)=&-[A^{\top }(t)p_{i}(t)+C^{\top }(t)k_{i}\left( t\right) +%
\widetilde{C}^{\top }(t)\widetilde{k}_{i}\left( t\right) -Q(t)(\bar{z}_i\left(
t\right) -l(t))]dt\\
& +k_{i}\left( t\right) dW_{i}(t)+\widetilde{k}_{i}\left( t\right) d\widetilde{%
W}_{i}(t), \\
p_{i}(T)=&-G(\bar{z}_{i}\left( T\right) -l(T)).\label{adjoint}%
\end{aligned}
\right.
\end{equation}%
Then, we have the following maximum principle of Problem (LLP), whose proof is trivial and thus is omitted.
\begin{theorem}
Let \textup{(H1)} and \textup{(H2)} hold. For $1\leq i \leq N$ and any fixed $l(\cdot)$, suppose $\bar{u}_{i}(\cdot )$ is the decentralized strategy of Problem (LLP), $\bar{z}_{i}(\cdot )$ is the optimal decentralized trajectory, and $(\bar{p}_i(\cdot),\bar{k}_i(\cdot),\bar{\widetilde{k}}_i(\cdot))$ solves \eqref{adjoint}, then
\begin{equation*}
\mathbb{E}[\langle \frac{\partial \mathcal{H}_{i}}{\partial u_{i}}(t,\bar{p}_{i},\bar{k}_{i},%
\bar{\widetilde{k}}_{i},\bar{z}_{i},\bar{u}_{i}),u_{i}-\bar{u}_{i}(t)\rangle |\mathcal{G}%
_{t}^{i}]\leq 0,~~\text{for any}~u_{i}\in \Gamma,~\text{a.e.}
~t\in[0,T], ~\mathbb{P}\text{-a.s.}
\end{equation*}%

\end{theorem}

According to the adaption of strategies, we have
\begin{equation*}
\langle \mathbb{E}[B^{\top }(t)\bar{p}_{i}(t)+D^{\top }(t)\bar{k}_{i}(t)+\widetilde{D}%
^{\top }(t)\bar{\widetilde{k}}_{i}(t)-R(t)\bar{u}_{i}(t)|\mathcal{G}%
_{t}^{i}],u_{i}-\bar{u}_{i}(t)\rangle \leq 0,
\end{equation*}%
%which yields that
%\begin{equation*}\langle R^{\frac{1}{2}}(t)[R(t)^{-1}\mathbb{E}[B^{\top }(t)\bar{p}_{i}(t)+D^{\top}(t)\bar{k}_{i}(t)+\widetilde{D}^{\top }(t)\bar{\widetilde{k}}_{i}(t)|\mathcal{G}_{t}^{i}]
%-\bar{u}_{i}(t)],R^{\frac{1}{2}}(t)(u_{i}-\bar{u}_{i}(t))\rangle
%\leq 0.\end{equation*}%
and by using convex analysis, the decentralized strategies read as
\begin{equation}\label{gcontrol}
\bar{u}_{i}(t)=\mathbf{P}_{\Gamma }[R(t)^{-1}(B^{\top }(t)\mathbb{E}[\bar{p}%
_{i}(t)|\mathcal{G}_{t}^{i}]+D^{\top }(t)\mathbb{E}[\bar{k}_{i}(t)|\mathcal{G%
}_{t}^{i}]+\widetilde{D}^{\top }(t)\mathbb{E}[\bar{\widetilde{k}}_{i}(t)|%
\mathcal{G}_{t}^{i}])],
\end{equation}%
where \textbf{P}$[\cdot]$ is a projection operator from $\mathbb{R}^m$ to $\Gamma$ under the norm $||x||_{R}^{2}=\langle \langle x,x\rangle \rangle =\langle R^{\frac{1}{2}
}x,R^{\frac{1}{2}}x\rangle$ . Moreover, the stochastic Hamiltonian system can be rewritten as the following nonlinear MF-FBSDE with projection operator:
\begin{equation}
\left\{
\begin{aligned}
d\bar{z}_{i}(t)=&\{A(t)\bar{z}_{i}(t)+B(t)\mathbf{P}_{\Gamma
}[R(t)^{-1}(B^{\top }(t)\mathbb{E}[\bar{p}_{i}(t)|\mathcal{G}%
_{t}^{i}]+D^{\top }(t)\mathbb{E}[\bar{k}_{i}(t)|\mathcal{G}_{t}^{i}] \\
&+\widetilde{D}^{\top }(t)\mathbb{E}[\bar{\widetilde{k}}_{i}(t)|\mathcal{G}%
_{t}^{i}])]+F(t)l(t)+b(t)\}dt\\
&+\{C(t)\bar{z}_{i}(t)+D(t)\mathbf{P}_{\Gamma }[R(t)^{-1}(B^{\top }(t)\mathbb{E}[\bar{p}_{i}(t)|%
\mathcal{G}_{t}^{i}]+D^{\top }(t)\mathbb{E}[\bar{k}_{i}(t)|\mathcal{G}%
_{t}^{i}] \\
&+\widetilde{D}^{\top }(t)\mathbb{E}[\bar{\widetilde{k}}_{i}(t)|\mathcal{G}%
_{t}^{i}])]+H(t)l(t)+\sigma (t)\}dW_{i}(t)\\
&+\{\widetilde{C}(t)\bar{z}_{i}(t)+\widetilde{D}(t)\mathbf{P}_{\Gamma }[R(t)^{-1}(B^{\top }(t)\mathbb{E}[\bar{p%
}_{i}(t)|\mathcal{G}_{t}^{i}]+D^{\top }(t)\mathbb{E}[\bar{k}_{i} (t)|%
\mathcal{G}_{t}^{i}] \\
&+\widetilde{D}^{\top }(t)\mathbb{E}[\bar{\widetilde{k}}_{i}(t)|\mathcal{G}%
_{t}^{i}])]+\widetilde{H}(t)l(t)+\widetilde{\sigma }(t)\}d\widetilde{W}%
_{i}(t), \\
d\bar{p}_{i}(t)=&-[A^{\top }(t)\bar{p}_{i}(t)+C^{\top }(t)\bar{k}_{i} \left(
t\right) +\widetilde{C}^{\top }(t)\bar{\widetilde{k}}_{i}\left( t\right)
-Q(t)(\bar{z}_{i}\left( t\right) -l(t))]dt\\
&+\bar{k}_{i}\left( t\right) dW_{i}(t)+\bar{\widetilde{k}}_{i}
\left( t\right) d\widetilde{W}_{i}(t), \\
\bar{z}_{i}(0)=&x,~\bar{p}_{i}(T)=-G(\bar{z}_{i}\left( T\right) -l(T)).\label{Hamiltion}
\end{aligned}
\right.
\end{equation}%

When $N\rightarrow\infty$, we would like to approximate $x_i(\cdot)$ by $z_i(\cdot)$, thus $\frac{1}{N}\sum_{i=1}^N x_{i}(\cdot)$ is approximated by $\frac{1}{N}\sum_{i=1}^N z_{i}(\cdot)$. Recalling that $z_i(\cdot)$ and $z_j(\cdot)$ are independent identically distributed (i.i.d), for $1\leq i,j\leq N$, $i\neq j$. Consequently, by the strong law of large number, it should follows that
\begin{equation}\label{limit}
l(\cdot)=\underset{N\rightarrow \infty }{\lim }\frac{1}{N}\overset{N}{\underset{%
i=1}{\sum }}\bar{z}_{i}(\cdot )=\mathbb{E}[\bar{z}_{i}(\cdot )].
\end{equation}
%In fact, we can show that in Lemma \ref{averageerror}.
Thus, by substituting \eqref{limit} into \eqref{Hamiltion}, we derive the following Hamiltonian type CC system, which is also a nonlinear MF-FBSDE with projection operator
\begin{equation}
\left\{
\begin{aligned}\label{CC}
d\bar{z}_{i}(t)=&\{A(t)\bar{z}_{i}(t)+B(t)\mathbf{P}_{\Gamma
}[R(t)^{-1}(B^{\top }(t)\mathbb{E}[\bar{p}_{i}(t)|\mathcal{G}%
_{t}^{i}]+D^{\top }(t)\mathbb{E}[\bar{k}_{i}(t)|\mathcal{G}_{t}^{i}] \\
&+\widetilde{D}^{\top }(t)\mathbb{E}[\bar{\widetilde{k}}_{i}(t)|\mathcal{G}%
_{t}^{i}])]+F(t)\mathbb{E}[\bar{z}_{i}(t)]+b(t)\}dt\\
&+\{C(t)\bar{z}_{i}(t)+D(t)\mathbf{P}_{\Gamma }[R(t)^{-1}(B^{\top }(t)\mathbb{E}[\bar{p}_{i}(t)|%
\mathcal{G}_{t}^{i}]+D^{\top }(t)\mathbb{E}[\bar{k}_{i}(t)|\mathcal{G}%
_{t}^{i}] \\
&+\widetilde{D}^{\top }(t)\mathbb{E}[\bar{\widetilde{k}}_{i}(t)|\mathcal{G}%
_{t}^{i}])]+H(t)\mathbb{E}[\bar{z}_{i}(t)]+\sigma (t)\}dW_{i}(t)\\
&+\{\widetilde{C}(t)\bar{z}_{i}(t)
+\widetilde{D}(t)\mathbf{P}_{\Gamma }[R(t)^{-1}(B^{\top }(t)\mathbb{E}[\bar{p%
}_{i}(t)|\mathcal{G}_{t}^{i}]+D^{\top }(t)\mathbb{E}[\bar{k}_{i} (t)|%
\mathcal{G}_{t}^{i}] \\
&+\widetilde{D}^{\top }(t)\mathbb{E}[\bar{\widetilde{k}}_{i}(t)|\mathcal{G}%
_{t}^{i}])]+\widetilde{H}(t)\mathbb{E}[\bar{z}_{i}(t)]+\widetilde{\sigma }(t)\}d\widetilde{W}%
_{i}(t), \\
d\bar{p}_{i}(t)=&-\{A^{\top }(t)\bar{p}_{i}(t)+C^{\top }(t)\bar{k}_{i} \left(
t\right) +\widetilde{C}^{\top }(t)\bar{\widetilde{k}}_{i}\left( t\right)
-Q(t)(\bar{z}_{i}\left( t\right)-\mathbb{E}[\bar{z}_{i}(t)]\}dt\\
&+\bar{k}_{i}\left( t\right) dW_{i}(t)+\bar{\widetilde{k}}_{i}
\left( t\right) d\widetilde{W}_{i}(t), \\
\bar{z}_{i}(0)=&~x,~\bar{p}_{i}(T)=-G(\bar{z}_{i}\left( T\right) -\mathbb{E}[\bar{z}_{i}(T)]),
\end{aligned}
\right.
\end{equation}%

\begin{remark}
The well-posedness of FBSDEs is of great significance in stochastic optimal control problems. The readers are referred to \cite{Antonelli1993} for the local case and to  \cite{Ma1994,Hu1995,Peng1999,Pardoux1999,Ma1999,Ma2015}, etc. for global cases.  There also exist many literatures on the well-posedness of MF-FBSDEs, for example, see \cite{Bensoussan2015,Ahuja2019,Carmona2018,Hu2018,Nie2018}.  We emphasize that our MF-FBSDE \eqref{CC} cannot be covered by above listed literatures since it includes the expectation term $\mathbb{E}[\cdot]$, the conditional expectation term $\mathbb{E}[\cdot|\mathcal{G}_t^i]$ and projection operator, thus its well-posedness in global case is not obvious.

Let us compare our MF-FBSDE with the counterpart in most relevant work \cite{Hu2018,Nie2018}. In \cite{Hu2018}, the authors studied a class of control constrained LQ MFGs where the diffusion term of each agent contains only control variable and nonhomogeneous part, thus for the study of the well-posedness of their MF-FBSDE with projection operator, the monotonicity condition holds and the continuity method works.  As a contrast, in our work, the diffusion term of individual agent depends both on the state and control, it causes that the monotonicity condition fails and the continuity method does not work. Moreover, we mention that \cite{Hu2018} and \cite{Nie2018} are all concerned with the  control constrained large-population problems in the full information case. By comparison, we place ourselves in the framework of partial information, which cause some essential differences between our MF-FBSDE for the CC condition and the previous ones in  \cite{Hu2018} and \cite{Nie2018}.  In fact, our MF-FBSDE will include simultaneously the expectation term, the conditional expectation term and projection operator, thus it will need some additional efforts to deal with  the simultaneous existence of expectation and conditional expectation. 
\end{remark}

However, inspired by \cite{Nie2018}, we can also use the discounting method proposed by \cite{Pardoux1999} to show the well-posedness of a kind of general MF-FBSDEs $($see \eqref{MF} below$)$ which include \eqref{CC} as a special case. To make our current paper self-contained, we will study  this kind of general MF-FBSDEs in the following.

Let $(\Omega, \mathcal{F},\mathbb{P})$ be a completed filtered probability space satisfying usual conditions, and $W$ and $\widetilde{W}$ be two $d$-dimensional mutually independent standard Brownian motions. Denote $\mathcal{F}_t^{W,\widetilde{W}}$ (resp. $\mathcal{F}_t^W$) as the natural filtration generated by $\{W(s), \widetilde{W}(s),$$ 0\leq s \leq t \}$ (resp. $\{W(s),0\leq s \leq t \}$) and augmented by all $\mathbb{P}$-null set.
Let us consider the following MF-FBSDE
\begin{equation}\label{MF}
\left\{
\begin{aligned}
dX(t)=&~b(t,\Theta(t))dt+\sigma(t,\Theta(t)) dW(t) +\widetilde{\sigma}(t,\Theta(t))d\widetilde{W}(t), \\
dY(t)=&-f(t,\Theta(t))dt+Z(t)dW(t)+\widetilde{Z}(t)d\widetilde{W}(t),\\
X(0)=&~x,~Y(T)=~g(X(T),\mathbb{E}[X(T)]),%
\end{aligned}
\right.
\end{equation}
where $\Theta(t)=(X(t),\mathbb{E}[X(t)],Y(t),\mathbb{E}[Y(t)|\mathcal{F}^W_t],Z(t),\mathbb{E}[Z(t)|\mathcal{F}^W_t],
\widetilde{Z}(t),\mathbb{E}[\widetilde{Z}(t)|\mathcal{F}^W_t])$. Here, $X(\cdot)$, $Y(\cdot)$, $Z(\cdot)$, $\widetilde{Z}(\cdot)$ take value in $\mathbb{R}^n$, $\mathbb{R}^m$, $\mathbb{R}^{m\times d}$, $\mathbb{R}^{m\times d}$, respectively. The coefficients $b,\sigma,\widetilde{\sigma}$ and $f$ are defined on $\Omega\times[0,T]\times\mathbb{R}^n\times\mathbb{R}^n\times\mathbb{R}^m\times\mathbb{R}^m\times \mathbb{R}^{m\times d}\times \mathbb{R}^{m\times d}\times \mathbb{R}^{m\times d}\times \mathbb{R}^{m\times d}$, such that $b(\cdot,\cdot,x,\alpha,y,\beta,z,\gamma,\tilde{z},\tilde{\gamma})$, $\sigma(\cdot,\cdot,x,\alpha,y,\beta,z,\gamma,\tilde{z},\tilde{\gamma})$, $\widetilde{\sigma}(\cdot,\cdot,x,\alpha,y,\beta,z,\gamma,\tilde{z},\tilde{\gamma})$ and $f(\cdot,\cdot,x,\alpha,y,\beta,z,\gamma,\tilde{z},\tilde{\gamma})$ are all $\mathcal{F}_t^{W,\widetilde{W}}$ progressively measurable processes, for any fixed $(x,\alpha,y,\beta,z,\gamma,\tilde{z},\tilde{\gamma})$. The coefficient $g$ is defined on $\Omega\times \mathbb{R}^n \times \mathbb{R}^n$ and $g(\cdot,x,\alpha)$ is $\mathcal{F}_T^{W,\widetilde{W}}$-measurable, for any fixed $(x,\alpha)$. Suppose that the coefficients $b,\sigma,\widetilde{\sigma},f$ and $g$ are all continuous with respect to $(x,\alpha,y,\beta,z,\gamma,\tilde{z},\tilde{\gamma})$. To guarantee the existence and uniqueness of a solution to \eqref{MF}, we give the following assumptions on the coefficients.

For any $t,x,x_{i},y,y_{i},\alpha,\alpha_i,\beta,\beta_i,z,z_i,\gamma,\gamma_i,\tilde{z},\tilde{z}_i,\tilde{\gamma},\tilde{\gamma}_i$, $i=1,2$, we denote $\Delta \phi=\phi_1-\phi_2$, where $\phi=x,y,\alpha,\beta,\gamma,z,\tilde{z},\tilde{\gamma}$.

\textup{(A1)} There exist $\lambda _{1},\lambda _{2}\in \mathbb{R}$ such
that
\begin{equation*}
\begin{aligned}
\langle b(t,x_{1},\alpha,y,\beta,z,\gamma,\tilde{z},\tilde{\gamma})-b(t,x_{2},y,\alpha,\beta,z,\gamma,\tilde{z},\tilde{\gamma}),\Delta x\rangle \leq& \lambda
_{1}|\Delta x|^{2},\\
\langle f(t,x,\alpha,y_1,\beta,z,\gamma,\tilde{z},\tilde{\gamma})-f(t,x,\alpha,y_2,\beta,z,\gamma,\tilde{z},\tilde{\gamma}),\Delta y\rangle \leq &\lambda
_{2}|\Delta y|^{2}.
\end{aligned}
\end{equation*}

\textup(A2) There exist positive constants $\rho$ and $\rho_i,\mu_i$, $i=1,2,\ldots,7$, such that
\begin{equation*}
\begin{aligned}
|b(t,x,\alpha,y,\beta,z,\gamma,\tilde{z},\tilde{\gamma})|\leq& |b(t,0,\alpha,y,\beta,z,\gamma,\tilde{z},\tilde{\gamma})|+\rho (1+|x|),\\
|f(t,x,\alpha,y,\beta,z,\gamma,\tilde{z},\tilde{\gamma})|\leq& |f(t,x,\alpha,0,\beta,z,\gamma,\tilde{z},\tilde{\gamma})|+\rho (1+|y|),
\end{aligned}
\end{equation*}%
\qquad\qquad and
\begin{equation*}
\begin{aligned}
&|b(t,x,\alpha_1,y_1,\beta_1,z_1,\gamma_1,\tilde{z}_1,\tilde{\gamma}_1)-b(t,x,\alpha_2,y_2,\beta_2,z_2,\gamma_2,\tilde{z}_2,\tilde{\gamma}_2)|\\
\leq &~\rho
_{1}|\Delta\alpha|+\rho _{2}|\Delta y|+\rho _{3}|\Delta\beta|+\rho _{4}|\Delta z|+\rho _{5}|\Delta \gamma|+\rho _{6}|\Delta\tilde{z}|+\rho _{7}|\Delta\tilde{\gamma}|,\\
&|f(t,x_1,\alpha_1,y,\beta_1,z_1,\gamma_1,\tilde{z}_1,\tilde{\gamma}_1)-f(t,x_2,\alpha_2,y,\beta_2,z_2,\gamma_2,\tilde{z}_2,\tilde{\gamma}_2)|\\
\leq &~\mu_1|\Delta x|+\mu_{2}|\Delta \alpha|+\mu _{3}|\Delta\beta|+\mu _{4}|\Delta z|
+\mu_{5}|\Delta\gamma|+\mu _{6}|\Delta\tilde{z}|+\mu _{7}|\Delta\tilde{\gamma}|.
\end{aligned}
\end{equation*}%

\textup(A3) There exist positive constants $w_i$ and $\kappa_i$, $i=1,2,\ldots,8$, such that
\begin{equation*}
\begin{aligned}
&|\sigma (t,x_1,\alpha_1,y_1,\beta_1,z_1,\gamma_1,\tilde{z}_1,\tilde{\gamma}_1)-\sigma
(t,x_2,\alpha_2,y_2,\beta_2,z_2,\gamma_2,\tilde{z}_2,\tilde{\gamma}_2)|^{2}\\
\leq &~ w_1^{2}|\Delta x|^{2}+w_2^{2}|\Delta\alpha|^2+w_3^2|\Delta y|^2+w_4^2|\Delta\beta|^2+w_5^2|\Delta z|^2+w_6^2|\Delta \gamma|^2+w_7^2|\Delta\tilde{z}|^2+w_8^2|\Delta\tilde{\gamma}|^2,\\
&|\widetilde{\sigma} (t,x_1,\alpha_1,y_1,\beta_1,z_1,\gamma_1,\tilde{z}_1,\tilde{\gamma}_1)-\widetilde{\sigma}
(t,x_2,\alpha_2,y_2,\beta_2,z_2,\gamma_2,\tilde{z}_2,\tilde{\gamma}_2)|^{2}\\
\leq &~ \kappa_1^{2}|\Delta x|^{2}+\kappa_2^{2}|\Delta\alpha|^2+\kappa_3^2|\Delta y|^2+\kappa_4^2|\Delta \beta|^2+\kappa_5^2|\Delta z|^2+\kappa_6^2|\Delta\gamma|^2+\kappa_7^2|\Delta\tilde{z}|^2+\kappa_8^2|\Delta\tilde{\gamma}|^2.
\end{aligned}
\end{equation*}%

\textup(A4) There exist positive constants $\rho_8$ and $\rho_9$ such that
\begin{equation*}
|g(x_1,\alpha_1)-g(x_2,\alpha_2)|^2\leq \rho_8^2 |\Delta x|^2+ \rho_9^2|\Delta\alpha|^2.
\end{equation*}

\textup{(A5)} For $\textbf{0}=(0,0,0,0,0,0,0,0)$, it holds
$
\mathbb{E}\int_{0}^{T}(|b(t,\textbf{0})|^{2}+|\sigma
(t,\textbf{0})|^{2}+|\widetilde{\sigma}
(t,\textbf{0})|^{2}+|f(t,\textbf{0})|^{2})dt+\mathbb{E}|g(0,0)|^{2}<+\infty.
$
\smallskip

Then, we can formulate the well-posedness result for \eqref{MF}, whose proof is given in Appendix \ref{appendix}.

\begin{theorem}\label{wellposeness}
Let \textup{(A1)}-\textup{(A5)} hold. If there exists a constant $\theta_0$, which depends on $\mu_i$, $\rho_1,\rho_8,\rho_9$,$w_1,w_2$,$\kappa_1,\kappa_2$ and $\lambda_1,\lambda_2$, $(i=1,\ldots,7)$ and $T$ such that when $\rho_j,w_{\tau},\kappa_{\tau}\in[0,\theta_0)$, $(j=2,\ldots,7)$ and $(\tau=3,\ldots,8)$, then there exists a unique adapted solution $(X(\cdot),Y(\cdot),Z(\cdot),\widetilde{Z}(\cdot))\in L_{\mathcal{F
}_{t}^{W,\widetilde{W}}}^{2}(0,T;\mathbb{R}^{n}\times \mathbb{R}^{m}\times \mathbb{R}^{m\times d}\times \mathbb{R}^{m\times d})$ to MF-FBSDE \eqref{MF}. Moreover, if $2(\lambda_1+\lambda_2)<-2\rho_1-2\mu_3-(\mu_4+\mu_5)^2-(\mu_6+\mu_7)^2-w_1^2-w_2^2-\kappa_1^2-\kappa_2^2$, there exists a constant $\theta_1$, which depends on $\mu_i$, $\rho_1,\rho_8,\rho_9$,$w_1,w_2$,$\kappa_1,\kappa_2$ and $\lambda_1,\lambda_2$, $(i=1,\ldots,7)$ but does not depend on $T$, such that when $\rho_j,w_{\tau},\kappa_{\tau}\in[0,\theta_1)$, $(j=2,\ldots,7)$ and $(\tau=3,\ldots,8)$, there exists a unique adapted solution $(X(\cdot),Y(\cdot),Z(\cdot),\widetilde{Z}(\cdot))\in L_{\mathcal{F
}_{t}^{W,\widetilde{W}}}^{2}(0,T;\mathbb{R}^{n}\times \mathbb{R}^{m}\times \mathbb{R}^{m\times d}\times \mathbb{R}^{m\times d})$ to MF-FBSDE \eqref{MF}.
\end{theorem}
By applying Theorem \ref{wellposeness}, we obtain the following well-posedness result for MF-FBSDE \eqref{CC}.
\begin{theorem}\label{Hwellposedness}
Let $\lambda ^{\ast }$ be the maximum eigenvalue of matrix $\frac{A+A^{\top}}{2}$, assume that $4\lambda ^{\ast }<-2|F|-6|C|^{2}-6|\widetilde{C}|^{2}-5|H|^{2}-5|\widetilde{H}|^{2}$, there exists a constant $\theta_1>0$ independent of $T$, which may depend on $\lambda^{\ast}$, $|C|$, $|\widetilde{C}|$, $|F|$, $|H|$, $|\widetilde{H}|$, $|Q|$, $|G|$, when $|B|$, $|D|$, $|\widetilde{D}|$ and $|R^{-1}|\in[0,\theta_1)$, then there exists a unique adapted solution $(\bar{z}_i(\cdot),\bar{p}_i(\cdot),\bar{k}_i(\cdot),\bar{\widetilde{k}}_i(\cdot))\in L_{\mathcal{F
}_{t}^{W,\widetilde{W}}}^{2}(0,T;\mathbb{R}^{n}\times \mathbb{R}^{m}\times \mathbb{R}^{m\times d}\times \mathbb{R}^{m\times d})$ to MF-FBSDE \eqref{CC}.
\end{theorem}

\begin{proof}
Noting that MF-FBSDE \eqref{CC} is a special case of \eqref{MF}. Indeed, by choosing
\begin{equation*}
\begin{aligned}
&\lambda _{1} =\lambda _{2}=\lambda ^{\ast },\rho =|A|,\rho
_{1}=|F|,\rho _{2}=\rho _{4}=\rho_6=0,\rho_3=|B|^2|R^{-1}|,\\
&\rho_5=|B||R^{-1}||D|,\rho_7=|B||R^{-1}||\widetilde{D}|,\rho_8^2=\rho_9^2=2|G|^2,\mu_1=\mu_2=|Q|, \\
&\mu_3=\mu_5=\mu_7=0,\mu_4=|C|,\mu_6=|\widetilde{C}|,
w_1^2=5|C|^2,w_2^2=5|H|^2,\\
&w_3=w_5=w_7=0,w_4^2=5|D|^2|R^{-1}|^2|B|^2,w_6^2=5|D|^2|R^{-1}|^2|D|^2,\\
&w_8^2=5|D|^2|R^{-1}|^2|\widetilde{D}|^2,\kappa_1^2=5|\widetilde{C}|^2,
\kappa_2^2=5|\widetilde{H}|^2,\kappa_3=\kappa_5=\kappa_7=0,\\
&\kappa_4^2=5|\widetilde{D}|^2|R^{-1}|^2|B|^2,\kappa_6^2=5|\widetilde{D}|^2|R^{-1}|^2|D|^2,\kappa_8^2=5|\widetilde{D}|^2|R^{-1}|^2|\widetilde{D}|^2,
\end{aligned}
\end{equation*}
thus assumptions \textup(A1)-\textup(A5) are satisfied naturally. By applying Theorem \ref{wellposeness}, the well-posedness of \eqref{CC} holds.   \hfill$\square$
\end{proof}
As yet, we have discussed Problem (LLP) which is the related limiting problem of Problem (LP), and we obtain the candidate decentralized strategy profile $\bar{u}(\cdot)=(\bar{u}_1(\cdot),\ldots,\bar{u}_N(\cdot))$, where $\bar{u}_i(\cdot)$ is given by \eqref{gcontrol} and $(\bar{z}_i(\cdot),\bar{p}_i(\cdot),\bar{k}_i(\cdot),\bar{\tilde{k}}_i(\cdot))$ solves \eqref{CC}. From Theorem \ref{Hwellposedness}, we know that $\bar{u}(\cdot)$ is well defined. Next, we will verify that $\bar{u}(\cdot)$ is indeed an $\varepsilon$-Nash equilibrium of Problem (LP). To do this, we suppose that $\bar{x}_i(\cdot)$ is the centralized state w.r.t. $\bar{u}_i(\cdot)$, and $\bar{z}_i(\cdot)$ is the corresponding decentralized state. Then, we have the following estimates for state (see Lemma \ref{averageerror}) and cost functional (see Lemma \ref{errorcost1}). In what follows, $K$ is a constant independent of $N$ and $i$ ($1\leq i\leq N$), which may vary line by line.
\begin{lemma}\label{averageerror}
Let \textup{(H1)} and \textup{(H2)} hold,  it follows that $($recall $\bar{x}^{(N)}(t)=\frac{1}{N}\overset{N}{\underset{i=1}{\sum }}\bar{x}_{i}(t)$$)$,
\begin{equation}\label{statee}
\mathbb{E}\underset{0\leq t\leq T}{\sup }|\bar{x}^{(N)}(t)-l(t)|^{2}+\underset{1\leq i\leq N}{\sup }\mathbb{E}\underset{0\leq t\leq T}{\sup }|%
\bar{x}_{i}(t)-\bar{z}_{i}(t)|^{2}=O(\frac{1}{N}).
\end{equation}
\end{lemma}
\begin{proof}
From \eqref{lstate} and \eqref{limit}, we have
\begin{equation}
\label{lequation}
dl(t)=\{(A(t)+F(t))l(t)+B(t)\mathbb{E}[\bar{u}_i(t)]+b(t)\}dt, \qquad
l(0)=x,
\end{equation}
and by recalling \eqref{state}, it holds that
\begin{equation*}
\left\{
\begin{aligned}
d(\bar{x}^{(N)}(t)-l(t))
&=~\big\{(A(t)+F(t))(\bar{x}^{(N)}(t)-l(t))+B(t)(\frac{1%
}{N}\overset{N}{\underset{i=1}{\sum }}\bar{u}_{i}(t)-\mathbb{E}[\bar{u}_{i}(t)])\big\}dt \\
&+\frac{1}{N}\overset{N}{\underset{i=1}{\sum }}[C(t)\bar{x}_{i}(t)+D(t)\bar{%
u}_{i}(t)+H(t)\bar{x}^{(N)}(t)+\sigma (t)]dW_{i}(t) \\
&+\frac{1}{N}\overset{N}{\underset{i=1}{\sum }}[\widetilde{C}(t)\bar{x}_{i}(t)+\widetilde{D}(t)\bar{%
u}_{i}(t)+\widetilde{H}(t)\bar{x}^{(N)}(t)+\widetilde{\sigma }(t)]d\widetilde{W}_{i}(t), \\
\bar{x}^{(N)}(0)-l(0)=&~0.
\end{aligned}%
\right.
\end{equation*}%
Since $\bar{u}_i(\cdot)$ and $\bar{u}_j(\cdot)$ are i.i.d (note $(\bar{p}_i(\cdot),\bar{k}_i(\cdot),\bar{\widetilde{k}}_i(\cdot))$ and $(\bar{p}_j(\cdot),\bar{k}_j(\cdot),\bar{\widetilde{k}}_j(\cdot))$ are i.i.d, for $i\neq j$), we have
\begin{equation}\label{firstestimate}
\mathbb{E}\int_{0}^{T}|\frac{1}{N}\overset{N}{\underset{i=1}{\sum }}\bar{u}%
_{i}(t)-\mathbb{E}[\bar{u}_i(t)]|^{2}dt=\frac{1}{N^{2}}\overset{N}{\underset{i=1}{\sum }}\mathbb{E}\int_{0}^{T}|\bar{u}%
_{i}(t)-\mathbb{E}[\bar{u}_i(t)]|^{2}dt 
\leq \frac{K}{N}=O(\frac{1}{N}).
\end{equation}%
Then by Burkholder-Davis-Gundy (BDG) inequality, it follows that
\begin{equation}
\begin{aligned}
\mathbb{E}\underset{0\leq t\leq T}{\sup }|\bar{x}^{(N)}(t)-l(t)|^{2}
&\leq K\mathbb{E}\int_{0}^{T}[|\bar{x}^{(N)}(t)-l(t)|^{2}+|\frac{1}{N}\overset{N%
}{\underset{i=1}{\sum }}\bar{u}_{i}(t)-\mathbb{E}[\bar{u}_i(t)]|^{2}]dt \\
&+\frac{K}{N^{2}}\mathbb{E}\overset{N}{\underset{i=1}{\sum }}%
\int_{0}^{T}|C(t)\bar{x}_{i}(t)+D(t)\bar{u}_{i}(t)+H(t)(\bar{x}%
^{(N)}(t)-l(t))+H(t)l(t)+\sigma (t)|^{2}dt \\
&+\frac{K}{N^{2}}\mathbb{E}\overset{N}{\underset{i=1}{\sum }}\int_{0}^{T}|\widetilde{C}(t)\bar{x}_{i}(t)+\widetilde{D}(t)\bar{%
u}_{i}(t)+\widetilde{H}(t)(\bar{x}^{(N)}(t)-l(t))+\widetilde{H}(t)l(t)+\widetilde{%
\sigma }(t)|^{2}dt.
\end{aligned}
\end{equation}
From  \eqref{state}, by applying Gronwall's inequality, one can prove $\mathbb{E}\underset{0\leq t\leq T}{\sup }\overset{N}{\underset{i=1}{\sum }} |\bar{x}_{i}(t)|^2=O(N) $, and thus $\mathbb{E}\underset{0\leq t\leq T}{\sup } |\bar{x}_{i}(t)|^2\leq K $. We can also show $\mathbb{E}\underset{0\leq t \leq T}{\sup}|l(t)|^2\leq K$ by noticing \eqref{lequation}.  Then,  we have
\begin{equation}\label{secondest}
\begin{aligned}
&\frac{K}{N^{2}}\mathbb{E}\overset{N}{\underset{i=1}{\sum }}%
\int_{0}^{T}|C(t)\bar{x}_{i}(t)+D(t)\bar{u}_{i}(t)+H(t)(\bar{x}%
^{(N)}(t)-l(t))+H(t)l(t)+\sigma (t)|^{2}dt \\
&+\frac{K}{N^{2}}\mathbb{E}\overset{N}{\underset{i=1}{\sum }}\int_{0}^{T}|\widetilde{C}(t)\bar{x}_{i}(t)+\widetilde{D}(t)\bar{%
u}_{i}(t)+\widetilde{H}(t)(\bar{x}^{(N)}(t)-l(t))+\widetilde{H}(t)l(t)+\widetilde{%
\sigma }(t)|^{2}dt\\
%\leq &\frac{K}{N^2}\overset{N}{\underset{i=1}{\sum }}\mathbb{E}\int_{0}^{T}(|%
%\bar{x}_{i}(t)|^{2}+|\bar{u}_{i}(t)|^{2}+|\bar{x}%^{(N)}(t)-l(t)|^{2}+|l(t)|^{2}+|\sigma (t)|^{2})dt \\
\leq &\frac{K}{N}\Big(1+\mathbb{E}\int_{0}^{T}|\bar{x}^{(N)}(t)-l(t)|^{2}dt\Big).
\end{aligned}
\end{equation}%
Combining \eqref{firstestimate}-\eqref{secondest}, %we have
%$\mathbb{E}\underset{0\leq t\leq T}{\sup }|\bar{x}^{(N)}(t)-l(t)|^{2}\leq K%
%\mathbb{E}\int_{0}^{T}|\bar{x}^{(N)}(t)-l(t)|^{2}dt+O(\frac{1}{N})$,
and applying Gronwall's inequality, we have
$
\mathbb{E}\underset{0\leq t\leq T}{\sup }|\bar{x}^{(N)}(t)-l(t)|^{2}=O(\frac{1}{N}).
$
Moreover, by recalling \eqref{state} and \eqref{lstate}, from standard estimates of SDE, we can obtain
$\underset{1\leq i\leq N}{\sup }\mathbb{E}\underset{0\leq t\leq T}{\sup }|%
\bar{x}_{i}(t)-\bar{z}_{i}(t)|^{2}=O(\frac{1}{N})$. \hfill$\square$
\end{proof}

\begin{lemma}\label{errorcost1}
Let \textup{(H1)} and \textup{(H2)} hold, then we have
$
|\mathcal{J}_{i}(\bar{u}_{i}(\cdot ),\bar{u}_{-i}(\cdot ))-J_{i}(\bar{u}%
_{i}(\cdot ))|=O(\frac{1}{\sqrt{N}})\label{costerror}
$, for $1\leq i \leq N$.
\end{lemma}

\begin{proof}
According to \eqref{cost} and \eqref{lcost}, we have
\begin{equation*}
\begin{aligned}
\mathcal{J}_{i}(\bar{u}_{i}(\cdot ),\bar{u}_{-i}(\cdot ))-J_{i}(\bar{u}%
_{i}(\cdot ))
 =&\frac{1}{2}\mathbb{E\{}\int_{0}^{T}[\langle Q(t)(\bar{x}%
_{i}(t)-\bar{x}^{(N)}(t)),\bar{x}_{i}(t)-\bar{x}^{(N)}(t)\rangle-\langle Q(t)(\bar{z}_{i}(t)-l(t)),\bar{z}_{i}(t)-l(t)\rangle ]dt\\
&\qquad+\langle G(\bar{x}_{i}(T)-\bar{x}^{(N)}(T)),\bar{x}_{i}(T)-\bar{x}%
^{(N)}(T)\rangle-\langle G(\bar{z}_{i}(T)-l(T)),\bar{z}_{i}(T)-l(T)\rangle\}.
\end{aligned}
\end{equation*}%
By noticing that $\langle Qa,a\rangle-\langle Qb,b\rangle=\langle Q(a-b),a-b\rangle+2\langle Q(a-b),b\rangle$, we obtain
\begin{equation*}
\begin{aligned}
&\mathbb{E}\int_{0}^{T}[\langle Q(t)(\bar{x}_{i}(t)-\bar{x}^{(N)}(t)),%
\bar{x}_{i}(t)-\bar{x}^{(N)}(t)\rangle -\langle Q(t)(\bar{z}%
_{i}(t)-l(t)),\bar{z}_{i}(t)-l(t)\rangle ]dt \\
\leq &K\int_{0}^{T}\mathbb{E}|\bar{x}_{i}(t)-\bar{z}_{i}(t)|^{2}dt+K%
\int_{0}^{T}\mathbb{E}|\bar{x}^{(N)}(t)-l(t)|^{2}dt\\&+K\int_{0}^{T}(\mathbb{E}|%
\bar{x}_{i}(t)-\bar{x}^{(N)}(t)-(\bar{z}_{i}(t)-l(t))|^{2})^{\frac{1}{2}%
}(\mathbb{E}|\bar{z}_{i}(t)-l(t)|^{2})^{\frac{1}{2}}dt \\
\leq &K\int_{0}^{T}\mathbb{E}|\bar{x}_{i}(t)-\bar{z}_{i}(t)|^{2}dt+K%
\int_{0}^{T}\mathbb{E}|\bar{x}^{(N)}(t)-l(t)|^{2}dt\\
&+K\int_{0}^{T}(\mathbb{E}|%
\bar{x}_{i}(t)-\bar{z}_{i}(t)|^{2}+\mathbb{E}|\bar{x}^{(N)}(t)-l(t)|^{2})^{%
\frac{1}{2}}dt=O(\frac{1}{\sqrt{N}}),
\end{aligned}
\end{equation*}%
where the last equality is due to Lemma \ref{averageerror} and $\mathbb{E}\underset{0\leq t\leq T}{\sup}(|\bar{z}_i(t)|^2+|l(t)|^2)\leq K$.
Similarly, we can prove that the difference of terminal term is also order of $\frac{1}{\sqrt{N}}$. The proof is complete.   \hfill$\square$

\end{proof}

Now, we consider the perturbation to $i$-th agent, i.e. the agent $\mathcal{A}_i$ choose an alternative control $u_i(\cdot)$, while other agents $\mathcal{A}_j$, for $j\neq i$, still take the decentralized strategy $\bar{u}_j(\cdot)$. Then the perturbed centralized state of $\mathcal{A}_i$ is given by
\begin{equation}\label{perturbed centralized state}
\left\{
\begin{aligned}
dy_{i}(t)=&~[A(t)y_{i}(t)+B(t)u_{i}(t)+F(t)y^{(N)}(t)+b(t)]dt \\
&+[C(t)y_{i}(t)+D(t)u_{i}(t)+H(t)y^{(N)}(t)+\sigma (t)]dW_{i}(t) \\
&+[\widetilde{C}(t)y_{i}(t)+\widetilde{D}(t)u_{i}(t)+\widetilde{H}(t)y^{(N)}(t)+\widetilde{\sigma }(t)]d\widetilde{W}_{i}(t),
\\
y_{i}(0)=&~x,%
\end{aligned}%
\right.
\end{equation}
and the perturbed centralized state of  $\mathcal{A}_j$ is given by
\begin{equation}\label{perturbed centralized state yj}
\left\{
\begin{aligned}
dy_{j}(t)=&~[A(t)y_{j}(t)+B(t)\bar{u}_{j}(t)+F(t)y^{(N)}(t)+b(t)]dt \\
&+[C(t)y_{j}(t)+D(t)\bar{u}_{j}(t)+H(t)y^{(N)}(t)+\sigma (t)]dW_{j}(t) \\
&+[\widetilde{C}(t)y_{j}(t)+\widetilde{D}(t)\bar{u}_{j}(t)+\widetilde{H}(t)y^{(N)}(t)+\widetilde{\sigma }(t)]d\widetilde{W}_{j}(t),
\\
y_{j}(0)=&~x,%
\end{aligned}%
\right.
\end{equation}
where $y^{(N)}(t)=\frac{1}{N}\sum_{i=1}^N y_{i}(t)$. Moreover, the corresponding decentralized states with perturbation satisfy
\begin{equation}
\left\{
\begin{aligned}\label{pyi}
d\bar{y}_{i}(t)=&~[A(t)\bar{y}_{i}(t)+B(t)u_{i}(t)+F(t)l(t)+b(t)]dt \\
&+[C(t)\bar{y}_{i}(t)+D(t)u_{i}(t)+H(t)l(t)+\sigma (t)]dW_{i}(t) \\
&+[\widetilde{C}(t)\bar{y}_{i}(t)+\widetilde{D}(t)u_{i}(t)+\widetilde{H}(t)l(t)+\widetilde{\sigma }(t)]d\widetilde{W}_{i}(t), \\
\bar{y}_{i}(0)=&~x,%
\end{aligned}%
\right.
\end{equation}
and
\begin{equation}
\left\{
\begin{aligned}\label{pyj}
d\bar{y}_{j}(t)=&~[A(t)\bar{y}_{j}(t)+B(t)\bar{u}_{j}(t)+F(t)l(t)+b(t)]dt \\
&+[C(t)\bar{y}_{j}(t)+D(t)\bar{u}_{j}(t)+H(t)l(t)+\sigma (t)]dW_{j}(t) \\
&+[\widetilde{C}(t)\bar{y}_{j}(t)+\widetilde{D}(t)\bar{u}_{j}(t)+\widetilde{H}(t)l(t)+\widetilde{\sigma }(t)]d\widetilde{W}_{j}(t), \\
\bar{y}_{j}(0)=&~x.%
\end{aligned}%
\right.
\end{equation}%
To show that $\bar{u}(\cdot)=(\bar{u}_1(\cdot),\ldots,\bar{u}_N(\cdot))$ is an $\varepsilon$-Nash equilibrium, we need to prove
\begin{equation*}
\mathcal{J}_{i}(\bar{u}_{i}(\cdot ),\bar{u}_{-i}(\cdot ))-\varepsilon \leq
\underset{u_{i}\left( \cdot \right) \in \mathcal{U}_{ad}^{c}}{\inf }\mathcal{J%
}_{i}(u_{i}(\cdot ),\bar{u}_{-i}(\cdot )),~\text{for any}~u_{i}\left( \cdot \right) \in \mathcal{U}_{ad}^{c}.
\end{equation*}%
Therefore,  it only needs to consider the alternative control $u_i(\cdot)\in \mathcal{U}_{ad}^{c}$ s.t.
$\mathcal{J}_{i}(\bar{u}_{i}(\cdot ),\bar{u}_{-i}(\cdot ))\geq\mathcal{J%
}_{i}(u_{i}(\cdot ),\bar{u}_{-i}(\cdot ))$. Thus,
\begin{equation*}
\mathbb{E}\int_{0}^{T}\langle R(t)u_{i}(t),u_{i}(t)\rangle dt\leq \mathcal{J}%
_{i}(u_{i}(\cdot ),\bar{u}_{-i}(\cdot ))\leq \mathcal{J}_{i}(\bar{u}_{i}(\cdot ),%
\bar{u}_{-i}(\cdot ))\leq J_{i}(\bar{u}_{i}(\cdot ))+O(\frac{1}{\sqrt{N}}),
\end{equation*}%
which implies $\mathbb{E}\int_{0}^{T}|u_{i}(t)|^{2}dt\leq K$. Then, we have the following estimates for the perturbed state and cost functional.
\begin{lemma}\label{yestimate}
Let \textup{(H1)} and \textup{(H2)} hold, then the following estimate holds
\begin{equation*}
\mathbb{E}\underset{0\leq t\leq T}{\sup }|y^{(N)}(t)-l(t)|^{2}+
\underset{1\leq i\leq N}{\sup }\mathbb{E}\underset{0\leq t\leq T}{\sup }%
|y_{i}(t)-\bar{y}_{i}(t)|^{2}=O(\frac{1}{N}).
\end{equation*}
\end{lemma}

\begin{proof}
By recalling \eqref{lequation}, \eqref{perturbed centralized state} and \eqref{perturbed centralized state yj}, we have
\begin{equation*}
\begin{aligned}
&\mathbb{E}\underset{0\leq t\leq T}{\sup }|y^{(N)}(t)-l(t)|^{2}\\
\leq& K\mathbb{E}\int_0^T(|y^{(N)}(t)-l(t)|^{2}+\frac{1}{N^2}|u_i(t)|^2)ds+K\mathbb{E}\int_0^T|\frac{1}{N}\overset{N}{\underset{j=1,j\neq i}{\sum }}\bar{u}_{j}(t)-\mathbb{E}[\bar{u}_i(t)]|^2dt\\
&+\frac{K}{N^2}\mathbb{E}\int_0^T|u_i(t)|^2dt+\frac{K}{N^2}\mathbb{E}\overset{N}{\underset{j=1}{\sum}}\int_0^T|C(t)y_j(t)+H(t)(y^{(N)}(t)-l(t))+H(t)l(t)+\sigma(t)|^2dt\\
&+\frac{K}{N^2}\mathbb{E}\overset{N}{\underset{j=1}{\sum}}\int_0^T|\widetilde{C}(t)y_j(t)+\widetilde{H}(t)(y^{(N)}(t)-l(t))+\widetilde{H}(t)l(t)+\widetilde{\sigma}(t)|^2dt+\frac{K}{N^2}\mathbb{E}\overset{N}{\underset{j=1,j\neq i}{\sum }}\int_0^T|\bar{u}_j(t)|^2dt.
\end{aligned}
\end{equation*}
Since $\{\bar{u}_i(\cdot)\}$ are i.i.d, it follows that $\mathbb{E}[\bar{u}_i(\cdot)]=\mathbb{E}[\bar{u}_j(\cdot)]$, for $1\leq i,j \leq N$ and $j\neq i$. Denote $\mu(t)=\mathbb{E}[\bar{u}_i(t)]$, we have
\begin{equation*}
\begin{aligned}
\int_0^T\mathbb{E}|\frac{1}{N}\overset{N}{\underset{j=1,j\neq i}{\sum }}\bar{u}_{j}(t)-\mu(t)|^2dt
\leq&\frac{2(N-1)^2}{N^2}\int_0^T\mathbb{E}|\frac{1}{N-1}\overset{N}{\underset{j=1,j\neq i}{\sum }}\bar{u}_{j}(t)-\mu(t)|^2dt+\frac{2}{N^2}\int_0^T\mathbb{E}|\mu(t)|^2dt\\
=&\frac{2(N-1)}{N^2}\int_0^T\mathbb{E}|\bar{u}_{j}(t)-\mu(t)|^2dt+\frac{2}{N^2}\int_0^T\mathbb{E}|\mu(t)|^2dt=O(\frac{1}{N}).
\end{aligned}
\end{equation*}
Similar to \eqref{secondest}, by using the fact that $\mathbb{E}\underset{0\leq t\leq T}{\sup}|y_i(t)|^2\leq K$, and recalling $\mathbb{E}\int_{0}^{T}|u_{i}(t)|^{2}dt\leq K$,  we have
\begin{equation*}
\begin{aligned}
&\frac{K}{N^2}\mathbb{E}\int_0^T|u_i(t)|^2dt+\frac{K}{N^2}\mathbb{E}\overset{N}{\underset{j=1}{\sum}}\int_0^T|C(t)y_j(t)+H(t)(y^{(N)}(t)-l(t))+H(t)l(t)+\sigma(t)|^2dt\\
&+\frac{K}{N^2}\mathbb{E}\overset{N}{\underset{j=1}{\sum}}\int_0^T|\widetilde{C}(t)y_j(t)+\widetilde{H}(t)(y^{(N)}(t)-l(t))+\widetilde{H}(t)l(t)+\widetilde{\sigma}(t)|^2dt\leq \frac{K}{N}\Big(1+\mathbb{E}\int_0^T|y^{(N)}(t)-l(t)|^{2}dt\Big).
\end{aligned}
\end{equation*}
Moreover, by i.i.d property of $\bar{u}_i(\cdot)$, we get
$\frac{K}{N^2}\mathbb{E}\overset{N}{\underset{j=1,j\neq i}{\sum }}\int_0^T|\bar{u}_j(t)|^2dt=O(\frac{1}{N})$.
Synthesizing above estimates, we have
\begin{equation*}
\mathbb{E}\underset{0\leq t\leq T}{\sup }|y^{(N)}(t)-l(t)|^{2}\leq K\mathbb{E%
}\int_{0}^{T}|y^{(N)}(t)-l(t)|^{2}dt+O(\frac{1}{N}).
\end{equation*}%
Finally, by recalling \eqref{perturbed centralized state}-\eqref{pyj}, with the help of standard SDE estimates, we can complete the proof. \hfill$\square$
\end{proof}
\smallskip

By using Lemma \ref{yestimate}, similar to the proof of Lemma \ref{errorcost1}, we have the following result,  whose proof is omitted.
\begin{lemma}\label{errorcost2}
Let \textup{(H1)} and \textup{(H2)} hold, we have
\begin{equation*}
|\mathcal{J}_{i}(u_{i}(\cdot ),\bar{u}_{-i}(\cdot ))-J_{i}(u_{i}(\cdot ))|=O(%
\frac{1}{\sqrt{N}})
, \text{ for }1\leq i \leq N.
\end{equation*}
\end{lemma}
Based on above Lemmas, we can give the following main result of this section.
\begin{theorem}\label{generalnash}
Let \textup{(H1)} and \textup{(H2)} hold. Assume that $4\lambda ^{\ast }<-2|F|-6|C|^{2}-6|\widetilde{C}|^{2}-5|H|^{2}-5|\widetilde{H}|^{2}$, there exists a constant $\theta_1>0$ independent of $T$, which may depend on $\lambda^{\ast}$, $|C|$, $|\widetilde{C}|$, $|F|$, $|H|$, $|\widetilde{H}|$, $|Q|$, $|G|$, when $|B|$, $|D|$, $|\widetilde{D}|$ and $|R^{-1}|\in[0,\theta_1)$, then the strategy profile $(\bar{u}_1(\cdot),\ldots,\bar{u}_N(\cdot))$ with $\bar{u}_i(\cdot)$ given by \eqref{gcontrol} is an $\varepsilon$-Nash equilibrium of Problem (LP), where $(\bar{z}_i(\cdot),\bar{p}_i(\cdot),\bar{k}_i(\cdot),\bar{\tilde{k}}_i(\cdot))$ is the unique solution to the FBSDE \eqref{CC}.
\end{theorem}

\begin{proof}
From Lemmas \ref{errorcost1} and \ref{errorcost2}, we obtain that for $1\leq i \leq N$,
$
\mathcal{J}_{i}(\bar{u}_{i}(\cdot ),\bar{u}_{-i}(\cdot ))=J_{i}(\bar{u}%
_{i}(\cdot ))+O(\frac{1}{\sqrt{N}})
\leq J_{i}(u_{i}(\cdot ))+O(\frac{1}{\sqrt{N}})
=\mathcal{J}_{i}(u_{i}(\cdot ),\bar{u}_{-i}(\cdot ))+O(\frac{1}{\sqrt{N}}),
$
which yields that $(\bar{u}_1(\cdot),\ldots,\bar{u}_N(\cdot))$ is an $\varepsilon$-Nash equilibrium.   \hfill$\square$
\end{proof}

\subsection{Control Unconstrained Case: Riccati Approach}\label{subsec:2}
In this subsection, we consider the control unconstrained case i.e. $\Gamma=\mathbb{R}^m$. We will use Riccati approach to represent the decentralized strategies as the feedback of filtered state.
Moreover, we introduce the following assumption.

\textup{(H3)} It holds that \[F=\delta I ~\text{and}~ H=\widetilde{H}=\widetilde{C}=0,~\text{where}~\delta~\text{is a constant.}\]

For simplicity, we denote $\hat{f}_i(t)=\mathbb{E}[f_i(t)|\mathcal{G}_t^i]$ as the filtering of $f_i(t)$ w.r.t. $\mathcal{G}_t^i$, for $1\leq i \leq N$. With above setting, the decentralized strategies give by \eqref{gcontrol} will be reduced to, for $1\leq i \leq N$,
\begin{equation}\label{open-loop}
\begin{aligned}
\bar{u}_{i}(t)=&R(t)^{-1}(B^{\top }(t)\mathbb{E}[\bar{p}%
_{i}(t)|\mathcal{G}_{t}^{i}]+D^{\top }(t)\mathbb{E}[\bar{k}_{i}(t)|\mathcal{G%
}_{t}^{i}]+\widetilde{D}^{\top }(t)\mathbb{E}[\bar{\widetilde{k}}_{i}(t)|\mathcal{G%
}_{t}^{i}]),\\
=&R(t)^{-1}(B^{\top }(t)\hat{\bar{p}}
_{i}(t)+D^{\top }(t)\hat{\bar{k}}_{i}(t)+\widetilde{D}^{\top }(t)\hat{\bar{\widetilde{k}}}_{i}(t)),
\end{aligned}
\end{equation}
where $(\bar{z}_i(\cdot),\bar{p}_i(\cdot),\bar{k}_i(\cdot),\bar{\widetilde{k}}_i(\cdot))$ solves the following Hamiltonian type CC system which is a MF-FBSDE
\begin{equation}
\left\{
\begin{aligned}\label{sCC}
d\bar{z}_{i}(t)=&\{A(t)\bar{z}_{i}(t)+B(t)R^{-1}(t)(B^{\top }(t)\mathbb{E}[\bar{p}_{i}(t)|\mathcal{G}%
_{t}^{i}]+D^{\top }(t)\mathbb{E}[\bar{k}_{i}(t)|\mathcal{G}_{t}^{i}] \\
&+\widetilde{D}^{\top }(t)\mathbb{E}[\bar{\widetilde{k}}_{i}(t)|\mathcal{G}%
_{t}^{i}])+\delta \mathbb{E}[\bar{z}_{i}(t)]+b(t)\}dt\\
&+\{C(t)\bar{z}_{i}(t)+D(t)R^{-1}(t)(B^{\top }(t)\mathbb{E}[\bar{p}_{i}(t)|%
\mathcal{G}_{t}^{i}]+D^{\top }(t)\mathbb{E}[\bar{k}_{i}(t)|\mathcal{G}%
_{t}^{i}]\\
&+\widetilde{D}^{\top }(t)\mathbb{E}[\bar{\widetilde{k}}_{i}(t)|\mathcal{G}%
_{t}^{i}])+\sigma (t)\}dW_{i}(t)\\
&+\{\widetilde{D}(t)R^{-1}(t)(B^{\top }(t)\mathbb{E}[\bar{p%
}_{i}(t)|\mathcal{G}_{t}^{i}]+D^{\top }(t)\mathbb{E}[\bar{k}_{i} (t)|%
\mathcal{G}_{t}^{i}]\\
&+\widetilde{D}^{\top }(t)\mathbb{E}[\bar{\widetilde{k}}_{i}(t)|\mathcal{G}%
_{t}^{i}])+\widetilde{\sigma }(t)\}d\widetilde{W}%
_{i}(t), \\
d\bar{p}_{i}(t)=&-\{A^{\top }(t)\bar{p}_{i}(t)+C^{\top }(t)\bar{k}_{i} \left(
t\right)
-Q(t)(\bar{z}_{i}\left( t\right)-\mathbb{E}[\bar{z}_{i}(t)]\}dt\\
&+\bar{k}_{i}\left( t\right) dW_{i}(t)+\bar{\widetilde{k}}_{i}
\left( t\right) d\widetilde{W}_{i}(t), \\
\bar{z}_{i}(0)=&~x,~\bar{p}_{i}(T)=-G(\bar{z}_{i}\left( T\right) -\mathbb{E}[\bar{z}_{i}(T)]).
\end{aligned}
\right.
\end{equation}%
Moreover, in the framework of this subsection, the above decentralized strategies can be further represented as the feedback of filtered state by Riccati approach as given in the following theorem.
\begin{theorem}\label{specialu}
Let \textup{(H1)}-\textup{(H3)} hold. Suppose $\Gamma=\mathbb{R}^m$, then the decentralized strategies can be represented as
\begin{equation}
\begin{aligned}\label{scontrol}
\bar{u}_i(t)=&-\widetilde{R}(t)^{-1}\widetilde{P}^{\top}(t)\hat{\bar{z}}_i(t)-\widetilde{R}(t)^{-1}B^{\top}(t)\Lambda(t)l(t)\\
&-\widetilde{R}(t)^{-1}(B^{\top}(t)\Phi(t)+D^{\top}(t)P(t)\sigma(t)+\widetilde{D}^{\top}(t)P(t)\widetilde{\sigma}(t)), ~~~1\leq i \leq N,
\end{aligned}
\end{equation}
with
\begin{equation}\label{tildeR}
\widetilde{R}(t)=R(t)+D^{\top}(t)P(t)D(t)+\widetilde{D}^{\top}(t)P(t)\widetilde{D}(t),
\end{equation}
\begin{equation}\label{tildeP}
\widetilde{P}(t)=P(t)B(t)+C^{\top}(t)P(t)D(t),
\end{equation}
where $P(\cdot)$ and $\Lambda(\cdot)$ solve the following Riccati equations, respectively
\begin{equation}
\left\{
\begin{aligned}\label{P}
&\dot{P}(t)+P(t)A(t)+A^{\top}(t)P(t)+C^{\top}(t)P(t)C(t)+Q(t)-\widetilde{P}(t)\widetilde{R}(t)^{-1}\widetilde{P}^{\top}(t)=0,\\
&P(T)=~G,
\end{aligned}
\right.
\end{equation}
\begin{equation}
\left\{
\begin{aligned}\label{Lambda}
&\dot{\Lambda}(t)+\Lambda(t)(A(t)-B(t)\widetilde{R}(t)^{-1}\widetilde{P}^{\top}(t))+(A(t)-B(t)\widetilde{R}(t)^{-1}\widetilde{P}^{\top}(t))^{\top}\Lambda(t)\\
&\qquad+(P(t)+\Lambda(t))\delta -\Lambda(t)B(t)\widetilde{R}(t)^{-1}B^{\top}(t)\Lambda(t)-Q(t)=0,\\
&\Lambda(T)=-G,
\end{aligned}
\right.
\end{equation}
$\Phi(\cdot)$ solves the following standard ordinary differential equation (ODE)
\begin{equation}
\left\{
\begin{aligned}\label{Phi}
&\dot{\Phi}(t)+(A^{\top}(t)-\widetilde{P}(t)\widetilde{R}(t)^{-1}B^{\top}(t)-\Lambda(t)B(t)\widetilde{R}(t)^{-1}B^{\top}(t))\Phi(t)+(C^{\top}(t)\\
&\qquad-\widetilde{P}(t)\widetilde{R}(t)^{-1}D^{\top}(t)-\Lambda(t)B(t)\widetilde{R}(t)^{-1}D^{\top}(t))P(t)\sigma(t)-(\widetilde{P}(t)\widetilde{R}(t)^{-1}\widetilde{D}^{\top}(t)\\
&\qquad+\Lambda(t)B(t)\widetilde{R}(t)^{-1}\widetilde{D}^{\top}(t))P(t)\widetilde{\sigma}(t)+(P(t)+\Lambda(t))b(t)=0,\\
&\Phi(T)=~0,
\end{aligned}
\right.
\end{equation}
$l(\cdot)$ representing the limit value of the state average solves
\begin{equation}
\left\{
\begin{aligned}\label{l}
dl(t)=&\{[A(t)+\delta -B(t)\widetilde{R}(t)^{-1}(\widetilde{P}^{\top}(t)+B^{\top}(t)\Lambda(t))]l(t)+b(t)\\
&-B(t)\widetilde{R}(t)^{-1}(B^{\top}(t)\Phi(t)+D^{\top}(t)P(t)\sigma(t)+\widetilde{D}^{\top}(t)P(t)\widetilde{\sigma}(t))\}dt,\\
l(0)=&~x,
\end{aligned}
\right.
\end{equation}
and the optimal filtering $\hat{\bar{z}}_i(\cdot)$ solves the following SDE
\begin{equation}
\left\{
\begin{aligned}\label{z}
&d\hat{\bar{z}}_i(t)=\{(A(t)-B(t)\widetilde{R}(t)^{-1}\widetilde{P}^{\top}(t))\hat{\bar{z}}_i(t)+(\delta-B(t)\widetilde{R}(t)^{-1}B^{\top}(t)\Lambda(t))l(t)\\
&\qquad-B(t)\widetilde{R}(t)^{-1}(B^{\top}(t)\Phi(t)+D^{\top}(t)P(t)\sigma(t)+\widetilde{D}^{\top}(t)P(t)\widetilde{\sigma}(t))+b(t)\}dt\\
&\qquad+\{(C(t)-D(t)\widetilde{R}(t)^{-1}\widetilde{P}^{\top}(t))\hat{\bar{z}}_i(t)-D(t)\widetilde{R}(t)^{-1}B^{\top}(t)\Lambda(t)l(t)\\
&\qquad-D(t)\widetilde{R}(t)^{-1}(B^{\top}(t)\Phi(t)+D^{\top}(t)P(t)\sigma(t)+\widetilde{D}^{\top}(t)P(t)\widetilde{\sigma}(t))+\sigma(t)\}dW_i(t),\\
&\hat{\bar{z}}_i(0)=~x.
\end{aligned}
\right.
\end{equation}
\end{theorem}

\begin{proof}
Due to the coupling structure of MF-FBSDE \eqref{sCC}, we speculate that
\begin{equation}\label{pdecompose}
\bar{p}_i(t)=-P(t)\bar{z}_i(t)-\Lambda(t)\mathbb{E}[\bar{z}_i(t)]-\Phi(t),
\end{equation}
with $P(T)=G$, $\Lambda(T)=-G$ and $\Phi(T)=0$.  Here, $P(\cdot)$, $\Lambda(\cdot)$ and $\Phi(\cdot)$ satisfy deterministic equations, which will be specified later. Let us first show \eqref{scontrol}.
By applying It\^o's formula to $\bar{p}_i(\cdot)$, we can derive
\begin{equation}
\begin{aligned}\label{barp}
d\bar{p}_i(t)=&\{-(\dot{P}(t)+P(t)A(t))\bar{z}_i(t)-[\dot{\Lambda}(t)+\Lambda(t)(A(t)+\delta)+P(t)\delta] \mathbb{E}[\bar{z}_i(t)]\\
&-P(t)B(t)\bar{u}_i(t)-\Lambda(t)B(t)\mathbb{E}[\bar{u}_i(t)]-(P(t)+\Lambda(t))b(t)-\dot{\Phi}(t)\}dt\\
&-P(t)[C(t)\bar{z}_{i}(t) +D(t)\bar{u}_i(t)+\sigma (t)]dW_i(t)-P(t)[\widetilde{D}(t)\bar{u}_i(t)+\widetilde{\sigma }(t)]d\widetilde{W}_i(t).
\end{aligned}
\end{equation}
Comparing with the diffusion term in the second equation of \eqref{sCC}, we get
\begin{equation}
\begin{aligned}
\bar{k}_i(t)=&-P(t)[C(t)\bar{z}_{i}(t) +D(t)\bar{u}_i(t)+\sigma (t)],\\
\bar{\widetilde{k}}_i(t)=&-P(t)[\widetilde{D}(t)\bar{u}_i(t)+\widetilde{\sigma }(t)].\label{krelation}
\end{aligned}
\end{equation}
Taking the conditional expectation with respect to $\mathcal{G}_t^i$ on both side of \eqref{pdecompose} and \eqref{krelation}, and substituting them into \eqref{open-loop}, we have
\begin{equation}
\begin{aligned}\label{u}
\bar{u}_i(t)=&-\widetilde{R}(t)^{-1}\widetilde{P}^{\top}(t)\hat{\bar{z}}_i(t)-\widetilde{R}(t)^{-1}B^{\top}(t)\Lambda(t)\mathbb{E}[\bar{z}_i(t)]\\
&-\widetilde{R}(t)^{-1}(B^{\top}(t)\Phi(t)+D^{\top}(t)P(t)\sigma(t)+\widetilde{D}^{\top}(t)P(t)\widetilde{\sigma}(t)), ~~~1\leq i \leq N,
\end{aligned}
\end{equation}
then  \eqref{scontrol} holds by recalling that $l(\cdot)=\mathbb{E}[\bar{z}_i(\cdot)]$ (see \eqref{limit}). Moreover, we have
\begin{equation}
\begin{aligned}\label{Eu}
\mathbb{E}[\bar{u}_i(t)]=&-\widetilde{R}(t)^{-1}(\widetilde{P}^{\top}(t)+B^{\top}(t)\Lambda(t))\mathbb{E}[\bar{z}_i(t)]\\
&-\widetilde{R}(t)^{-1}(B^{\top}(t)\Phi(t)+D^{\top}(t)P(t)\sigma(t)+\widetilde{D}^{\top}(t)P(t)\widetilde{\sigma}(t)), ~~~1\leq i \leq N.
\end{aligned}
\end{equation}

Now, let us deduce the equations for $P(\cdot)$, $\Lambda(\cdot)$, $\Phi(\cdot)$ and $l(\cdot)$.
From the  drift term of \eqref{barp} and second equation in \eqref{sCC}, by noting \eqref{pdecompose} and \eqref{krelation}, one can obtain that
\begin{equation}
\begin{aligned}\label{drift}
&(\dot{P}(t)+P(t)A(t)+A^{\top}(t)P(t)+C^{\top}(t)P(t)C(t)+Q(t))\bar{z}_i(t)\\
&+[\dot{\Lambda}(t)+\Lambda(t)(A(t)+\delta)+P(t)\delta+A^{\top}(t)\Lambda(t)-Q(t)]\mathbb{E}[\bar{z}_i(t)]\\
&+(P(t)B(t)+C^{\top}(t)P(t)D(t))\bar{u}_i(t)+\Lambda(t)B(t)\mathbb{E}[\bar{u}_i(t)]\\
&+\dot{\Phi}(t)+P(t)b(t)+\Lambda(t)b(t)+A^{\top}(t)\Phi(t)+C^{\top}(t)P(t)\sigma(t)=0.
\end{aligned}
\end{equation}
By taking conditional expectation on \eqref{drift} and by virtue of \eqref{u}-\eqref{Eu}, we have
\begin{equation}
\begin{aligned}
&(\dot{P}(t)+P(t)A(t)+A^{\top}(t)P(t)+C^{\top}(t)P(t)C(t)+Q(t)-\widetilde{P}(t)\widetilde{R}(t)^{-1}\widetilde{P}^{\top}(t))\hat{\bar{z}}_i(t)\\
&+[\dot{\Lambda}(t)+\Lambda(t)(A(t)-B(t)\widetilde{R}(t)^{-1}\widetilde{P}^{\top}(t))+(A(t)-B(t)\widetilde{R}(t)^{-1}\widetilde{P}^{\top}(t))^{\top}\Lambda(t)\\
&+(P(t)+\Lambda(t))\delta-\Lambda(t)B(t)\widetilde{R}(t)^{-1}B^{\top}(t)\Lambda(t)-Q(t)]\mathbb{E}[\bar{z}_i(t)]\\
&+\dot{\Phi}(t)+(A^{\top}(t)-\widetilde{P}(t)\widetilde{R}(t)^{-1}B^{\top}(t)-\Lambda(t)B(t)\widetilde{R}(t)^{-1}B^{\top}(t))\Phi(t)\\
&+(C^{\top}(t)-\widetilde{P}(t)\widetilde{R}(t)^{-1}D^{\top}(t)-\Lambda(t)B(t)\widetilde{R}(t)^{-1}D^{\top}(t))P(t)\sigma(t)\\
&-(\widetilde{P}(t)\widetilde{R}(t)^{-1}\widetilde{D}^{\top}(t)+\Lambda(t)B(t)\widetilde{R}(t)^{-1}\widetilde{D}^{\top}(t))P(t)\widetilde{\sigma}(t)+(P(t)+\Lambda(t))b(t)=0,\\
\end{aligned}
\end{equation}
which suggests that $P(\cdot)$, $\Lambda(\cdot)$ and $\Phi(\cdot)$ solve \eqref{P}, \eqref{Lambda} and \eqref{Phi}, respectively. In addition, by taking expectation on both side of the first equation in \eqref{sCC} and  by noting \eqref{open-loop}, we have
\begin{equation}\label{Ez}
d\mathbb{E}[\bar{z}_i(t)]=\{(A(t)+\delta)\mathbb{E}[\bar{z}_i(t)]+B(t)\mathbb{E}[\bar{u}_i(t)]+b(t)\}dt,
\end{equation}
and by substituting \eqref{Eu} into \eqref{Ez} and by recalling $l(\cdot)=\mathbb{E}[\bar{z}_i(\cdot)]$, it is easy to show that $l(\cdot)$ solves \eqref{l}. Moreover, from \eqref{sCC} and \eqref{scontrol}, the optimal filtering $\hat{\bar{z}}_i(\cdot)$ can be expressed as in \eqref{z}. \hfill$\square$     \end{proof}

To summarize, we obtain that $(P(\cdot),\Lambda(\cdot),\Phi(\cdot),l(\cdot))$ solves the following Riccati type CC system
\begin{equation}
\left\{
\begin{aligned}\label{RCC}
&\dot{P}(t)+P(t)A(t)+A^{\top}(t)P(t)+C^{\top}(t)P(t)C(t)+Q(t)-\widetilde{P}(t)\widetilde{R}(t)^{-1}\widetilde{P}^{\top}(t)=0,\\
&\dot{\Lambda}(t)+\Lambda(t)(A(t)-B(t)\widetilde{R}(t)^{-1}\widetilde{P}^{\top}(t))+(A(t)-B(t)\widetilde{R}(t)^{-1}\widetilde{P}^{\top}(t))^{\top}\Lambda(t)\\
&\qquad+(P(t)+\Lambda(t))\delta -\Lambda(t)B(t)\widetilde{R}(t)^{-1}B^{\top}(t)\Lambda(t)-Q(t)=0,\\
&\dot{\Phi}(t)+(A^{\top}(t)-\widetilde{P}(t)\widetilde{R}(t)^{-1}B^{\top}(t)-\Lambda(t)B(t)\widetilde{R}(t)^{-1}B^{\top}(t))\Phi(t)+(C^{\top}(t)\\
&\qquad-\widetilde{P}(t)\widetilde{R}(t)^{-1}D^{\top}(t)-\Lambda(t)B(t)\widetilde{R}(t)^{-1}D^{\top}(t))P(t)\sigma(t)-(\widetilde{P}(t)\widetilde{R}(t)^{-1}\widetilde{D}^{\top}(t)\\
&\qquad+\Lambda(t)B(t)\widetilde{R}(t)^{-1}\widetilde{D}^{\top}(t))P(t)\widetilde{\sigma}(t)+(P(t)+\Lambda(t))b(t)=0,\\
&\dot{l}(t)-[A(t)+\delta -B(t)\widetilde{R}(t)^{-1}(\widetilde{P}^{\top}(t)+B^{\top}(t)\Lambda(t))]l(t)-b(t)\\
&\qquad+B(t)\widetilde{R}(t)^{-1}(B^{\top}(t)\Phi(t)+D^{\top}(t)P(t)\sigma(t)+\widetilde{D}^{\top}(t)P(t)\widetilde{\sigma}(t))=0,\\
&P(T)=~G,\Lambda(T)=-G,\Phi(T)=0,l(0)=~x,
\end{aligned}
\right.
\end{equation}
where we recall that $\widetilde{R}(t)=R(t)+D^{\top}(t)P(t)D(t)+\widetilde{D}^{\top}(t)P(t)\widetilde{D}(t)$ and
$\widetilde{P}(t)=P(t)B(t)+C^{\top}(t)P(t)D(t)$.
\smallskip

By applying Theorem \ref{generalnash}, we have
\begin{theorem}
Let \textup{(H1)}-\textup{(H3)} hold, the strategy profile $\bar{u}(\cdot)=(\bar{u}_1(\cdot),\ldots,\bar{u}_N(\cdot))$, where $\bar{u}_i(\cdot)$ is given by \eqref{scontrol} and $(P(\cdot),\Lambda(\cdot),\Phi(\cdot),l(\cdot),\hat{\bar{z}}_i(\cdot))$ solves systems \eqref{P}-\eqref{z}, is an $\varepsilon$-Nash equilibrium of Problem (LP) with $\Gamma=\mathbb{R}^m$.
\end{theorem}
\begin{remark}
\textup{(i)} When $\Gamma=\mathbb{R}^m$, if we further assume that $\widetilde{D}=\widetilde{\sigma}=0$ and only one Brownian motion $W_i$ is involved, then it reduces to the control unconstrained large-population problem with full information, which serves as a special case of our problem. We emphasize that in this situation $\widetilde{R}(\cdot)=R(\cdot)+D^{\top}(\cdot)P(\cdot)D(\cdot)$ and thus the Riccati equation for $P(\cdot)$ in the full information case is standard, whose solvability can be guaranteed automatically by \cite{Yong1999}.  By contrast, we will see that the introduction of the partial information makes it very difficult to solve the corresponding Riccati equations, in particular for $P(\cdot)$. In fact, the $\widetilde{D}(\cdot)u_i(\cdot)$ term will result in
$\widetilde{R}(\cdot)=R(\cdot)+D^{\top}(\cdot)P(\cdot)D(\cdot)+\widetilde{D}^{\top}(t)P(\cdot)\widetilde{D}(\cdot)$, where the additional term $\widetilde{D}^{\top}(\cdot)P(\cdot)\widetilde{D}(\cdot)$ cannot be combined with the original term $D^{\top}(\cdot)P(\cdot)D(\cdot)$ into a quadratic form (see \eqref{r1}), thus the equation for $P(\cdot)$ is not a standard Riccati equation and then the existing results cannot be implemented.   Therefore, we will focus on the wellposedness of system \eqref{RCC} in subsection \ref{subsec:3}.

\textup{(ii)} If we suppose that $\bar{p}_i(\cdot)$ has the following decomposition $($which is different to \eqref{pdecompose}$)$
\begin{equation}
\bar{p}_i(t)=-P(t)\bar{z}_i(t)-\varphi(t), \label{prelation}
\end{equation}
with $P(T)=-G,\varphi(T)=G\widetilde{l}(T)$. Then by taking similar procedure in above proof, it can be verified that $P(\cdot)$ still solves \eqref{P}, and $\varphi(\cdot)$ solves the following ODE
\begin{equation}
\left\{
\begin{aligned}\label{vari}
&\dot{\varphi}(t)+(A(t)-B(t)\widetilde{R}(t)^{-1}\widetilde{P}^{\top}(t))^{\top}\varphi(t)+(C(t)-D(t)\widetilde{R}(t)^{-1}\widetilde{P}^{\top}(t))^{\top}P(t)\sigma(t)\\
&\qquad-\widetilde{P}(t)\widetilde{R}(t)^{-1}\widetilde{D}^{\top}(t)P(t)\widetilde{\sigma}(t)+P(t)(\delta \widetilde{l}(t)+b(t))-Q(t)\widetilde{l}(t)=0,\\
&\varphi(T)=~G\widetilde{l}(T).
\end{aligned}
\right.
\end{equation}
and $\widetilde{l}(\cdot)$ solves
\begin{equation}
\left\{
\begin{aligned}\label{newl}
d\widetilde{l}(t)=&\{(A(t)+\delta-B(t)\widetilde{R}(t)^{-1}\widetilde{P}^{\top}(t))\widetilde{l}(t)+b(t)-B(t)\widetilde{R}(t)^{-1}(B^{\top}(t)\varphi(t)\\
&+D^{\top}(t)P(t)\sigma(t)+\widetilde{D}^{\top}(t)P(t)\widetilde{\sigma}(t))\}dt,\\
\widetilde{l}(0)=&~x.
\end{aligned}
\right.
\end{equation}
In this case, for any $1\leq i \leq N$ the decentralized strategies can be represented as
\begin{equation}
\begin{aligned}
\bar{u}_i(t)=-\widetilde{R}(t)^{-1}\widetilde{P}^{\top}(t)\hat{\bar{z}}_i(t)
-\widetilde{R}(t)^{-1}(B^{\top}(t)\varphi(t)
+D^{\top}(t)P(t)\sigma(t)+\widetilde{D}^{\top}(t)P(t)\widetilde{\sigma}(t)).\end{aligned}
\end{equation}
According to above analysis, it's not surprising to discover that the system of \eqref{vari} and \eqref{newl} is a kind of coupled forward backward ODEs, whose well-posedness is not easy to check. For this reason, we introduce a new decomposition as \eqref{pdecompose} which allows us to solve $P(\cdot),\Lambda(\cdot),\Phi(\cdot)$ and $l(\cdot)$ one by one. We emphasize that by comparing \eqref{pdecompose} with \eqref{prelation}, one can check that $\varphi(\cdot)=\Lambda(\cdot)\widetilde{l}(\cdot)+\Phi(\cdot)$. Moreover, if the system \eqref{RCC} is uniquely solvable, we can also solve uniquely systems \eqref{vari} and \eqref{newl}.  The existence is obvious and we only mention the uniqueness here. In fact, let $(\varphi^\prime(\cdot),\widetilde{l}^\prime(\cdot))$ be an another solution to \eqref{vari} and \eqref{newl}. After simple calculation, we can verify that $\Phi(\cdot)=\varphi(\cdot)-\Lambda(\cdot)\widetilde{l}(\cdot)$ and $\Phi^\prime(\cdot)=\varphi^\prime(\cdot)-\Lambda(\cdot)\widetilde{l}^\prime(\cdot)$ both satisfy the third equation of system \eqref{RCC}.  Due to the uniqueness of solution to \eqref{RCC}, we have $\Phi(\cdot)=\Phi^\prime(\cdot)$. By substituting the relationship $\varphi(\cdot)=\Lambda(\cdot)\widetilde{l}(\cdot)+\Phi(\cdot)$ and $\varphi^\prime(\cdot)=\Lambda(\cdot)\widetilde{l}^\prime(\cdot)+\Phi(\cdot)$ into \eqref{newl}, we have that $\widetilde{l}(\cdot)$ and $\widetilde{l}^\prime(\cdot)$ satisfy the same ODE, which implies $\widetilde {l }(\cdot)=\widetilde {l}^\prime(\cdot)$ from the classical ODE theory and then further $\varphi(\cdot)=\varphi^\prime(\cdot)$.  Thus, we will focus on the well-posedness of system \eqref{RCC} in next subsection.
\end{remark}
\begin{remark}\label{remark for partial information structure}
Huang and Wang \cite{Huang2016} also considered  a class of unconstrained LQ large-population problems with partial information via decoupling methods. However, the diffusion coefficient of the dynamic of individual agent in \cite{Huang2016} is in a simple manner. Indeed,  their diffusion coefficient depends neither on  control nor state. In our work, we study a  general partial information stochastic large-population problem. Moveover, the well-posedness of a new Riccati type CC system \eqref{RCC} is also provided  in next subsection, which looks interesting itself.
\end{remark}

\subsection{Well-posedness of Riccati Type CC System}\label{subsec:3}
In this subsection, we are going to study the well-posedness of general Riccati type CC system \eqref{RCC} consisting of four equations (see \eqref{P}-\eqref{l}).

We emphasize that due to our general partial information structure, especially that the diffusion term of the state contains the term $\widetilde{D}(\cdot)u_i(\cdot)$, the equations for $P(\cdot)$ and $\Lambda(\cdot)$ are no longer standard Riccati equations and it will arise essential difficulties to get the well-posedness of system \eqref{RCC}. Indeed, firstly, it is obvious that we have
\begin{equation}\label{r1}
D^{\top}(t)P(t)D(t)+\widetilde{D}^{\top}(t)P(t)\widetilde{D}(t)\neq (D(t)+\widetilde{D}(t))^{\top}P(t)(D(t)+\widetilde{D}(t)),
\end{equation}
which means that equation \eqref{P} for $P(\cdot)$ is not a standard Riccati equation. Secondly, since usually the inequality $\delta P(\cdot)-Q(\cdot)\geq 0$ fails, the well-posedness of equation \eqref{Lambda} for $\Lambda(\cdot)$ is not obvious at all.

To study the well-posedness of system \eqref{RCC}, we first recall the following useful results, which are classical in algebra.  It is worth noting that the second assertion in the following lemma serves as a corollary of the Weyl's inequality.  For saving space, we omit the proof here.
\begin{lemma}\label{alegbra1}
Let $\mathbb{A},\mathbb{B}\in \mathcal{S}^n$, then the following results hold:\\
\textup{(i)} if $\mathbb{A}\geq\mathbb{B}>0$, then $\mathbb{B}^{-1}\geq\mathbb{A}^{-1}>0$.\\
\textup{(ii)} if $\mathbb{A}\geq \mathbb{B}$, then $\lambda_k(\mathbb{A})\geq\lambda_k(\mathbb{B})$, $k=1,\ldots,n$, where $\lambda_k(\mathbb{A})$, $(resp.~\lambda_k(\mathbb{B}))$ is $k$-th eigenvalue of $\mathbb{A}$ $(resp.~\mathbb{B})$, i.e. $\lambda_1(\mathbb{A})\ge\lambda_2(\mathbb{A})\ge\ldots\ge\lambda_n(\mathbb{A})$ and $\lambda_1(\mathbb{B})\ge\lambda_2(\mathbb{B})\ge\ldots\ge\lambda_n(\mathbb{B})$.
\end{lemma}

%\begin{proof}\textup{(i)} Since matrices $\mathbb{A}$, $\mathbb{B}$ are positive definite, then there exist two non-singular matrices $\mathbb{C}$ and $\mathbb{D}$, such that $\mathbb{A}=\mathbb{C}\mathbb{C}^\top$, $\mathbb{B}=\mathbb{D}\mathbb{D}^\top$. Hence, $\mathbb{A}\geq\mathbb{B}$ is equivalent to $\mathbb{D}(\mathbb{D}^{-1}\mathbb{C}\mathbb{C}^\top\mathbb{D}^{-\top}-I)\mathbb{D}^\top\geq0$, i.e. $\mathbb{D}^{-1}\mathbb{C}\mathbb{C}^\top\mathbb{D}^{-\top}-I\geq0$.

%Noting that $\mathbb{C}$ and $\mathbb{D}$ are both non-singular, we deduce that $\mathbb{D}^{-1}\mathbb{C}\mathbb{C}^\top\mathbb{D}^{-\top}$ is similar to $\mathbb{C}^\top\mathbb{D}^{-\top}\mathbb{D}^{-1}\mathbb{C}$, which implies $\mathbb{C}^\top\mathbb{D}^{-\top}\mathbb{D}^{-1}\mathbb{C}-I\geq0$, i.e. $\mathbb{D}^{-\top}\mathbb{D}^{-1}-\mathbb{C}^{-\top}\mathbb{C}^{-1}\geq0$, then desired result is obvious.\\
%\textup{(ii)} Because of $\mathbb{A}\geq\mathbb{B}$, then $\mathbb{M}=\mathbb{A}-\mathbb{B}\geq0$. Applying Weyl's inequality to $\mathbb{A}=\mathbb{B}+\mathbb{M}$, we can get
%\begin{equation*}\lambda_k(\mathbb{A})=\lambda_k(\mathbb{B}+\mathbb{M})\geq \lambda_k(\mathbb{B})+\lambda_n(\mathbb{M})\geq \lambda_k(\mathbb{B}), \quad k=1,\ldots,n,\end{equation*}where $\lambda_n(\mathbb{M})\geq 0$.    \hfill$\square$\end{proof}

Let us first focus on the well-posedness of equation \eqref{P}. As mentioned above, equation \eqref{P} is not a standard Riccati equation due to the additional term related to $\widetilde{D}$. In the following lemma,   we can show the uniqueness of a solution for equation \eqref{P} by Lemma \ref{alegbra1} and assumptions \textup{(H1)}-\textup{(H3)}. Moreover, we will use modified iterative method and mathematical induction (see the modified term $\widehat{Q}(\cdot)$ and $\Psi(\cdot)$ in the following proof) to obtain the existence of a solution for equation \eqref{P}.

\begin{lemma}\label{Plemma}
Let \textup{(H1)}-\textup{(H3)} hold, then Riccati equation \eqref{P} admits a unique solution $P(\cdot)\in C([0,T];\mathcal{S}_+^n)$.
\end{lemma}

\begin{proof}
Firstly, we prove \eqref{P} admits at most one solution $P(\cdot)\in C([0,T];\mathcal{S}_+^n)$. Suppose that $P_1(\cdot)$ and $P_2(\cdot)$ are two solutions of \eqref{P}, we denote $\widehat{P}(\cdot)=P_1(\cdot)-P_2(\cdot)$, then we have

\begin{equation*}
\left\{
\begin{aligned}
&\dot{\widehat{P}}(t)+\widehat{P}(t)A(t)+A^{\top}(t)\widehat{P}(t)+C^{\top}(t)\widehat{P}(t)C(t)\\
&\qquad-(\widehat{P}(t)B(t)+C^{\top}(t)\widehat{P}(t)D(t))R_1(t)^{-1}(P_1(t)B(t)+C^{\top}(t)P_1(t)D(t))^{\top}\\
&\qquad-(P_2(t)B(t)+C^{\top}(t)P_2(t)D(t))R_2(t)^{-1}(\widehat{P}(t)B(t)+C^{\top}(t)\widehat{P}(t)D(t))^\top\\
&\qquad+(P_2(t)B(t)+C^{\top}(t)P_2(t)D(t))R_2(t)^{-1}(D^{\top}(t)\widehat{P}(t)D(t)\\
&\qquad+\widetilde{D}^{\top}(t)\widehat{P}(t)\widetilde{D}(t))R_1(t)^{-1}(P_1(t)B(t)+C^{\top}(t)P_1(t)D(t))^{\top}=0,\\
&\widehat{P}(T)=0,
\end{aligned}
\right.
\end{equation*}
where $R_i(t)=R(t)+D^{\top}(t)P_i(t)D(t)+\widetilde{D}^{\top}(t)P_i(t)\widetilde{D}(t)$, $i=1,2$.
From Lemma \ref{alegbra1}, we have $R_i(t)^{-1}\leq R(t)^{-1}$, $\forall t\in [0,T]$, $i=1,2$, and
\begin{equation*}
|R_i(t)^{-1}|=\sqrt{\underset{k=1}{\overset{m}{\sum}}\lambda_k^2(R_i(t)^{-1})}
\leq \sqrt{\underset{k=1}{\overset{m}{\sum}}\lambda_k^2(R(t)^{-1})}=|R(t)^{-1}|<\infty.
\end{equation*}
where the last inequality is due to $R\gg0$ (which implies $R^{-1}\in L^{\infty }(0,T;\mathcal{S}^{m})$). Consequently, $|R_1(t)^{-1}|$ and $|R_2(t)^{-1}|$ are uniformly bounded. From Gronwall's inequality, we have $\widehat{P}(t)=0$, which yields the uniqueness of $P(\cdot)$.

Secondly, let us focus on the existence of a solution to equation \eqref{P}. Motivated by \cite{Yong1999}, we set
\begin{equation*}
\left\{
\begin{aligned}
&\widehat{A}(t)=A(t)-B(t)\Psi(t),\widehat{C}(t)=C(t)-D(t)\Psi(t),\\
&\widehat{Q}(t)=Q(t)+\Psi^{\top}(t)(R(t)+\widetilde{D}^{\top}(t)P(t)\widetilde{D}(t))\Psi(t),\\
&\Psi(t)=(R(t)\!+\!D^{\top}(t)P(t)D(t)\!+\!\widetilde{D}^{\top}(t)P(t)\widetilde{D}(t))^{-1}
(P(t)B(t)\!+\!C^\top(t)P(t)D(t))^\top.\\
\end{aligned}
\right.
\end{equation*}
It is easy to verify that equation \eqref{P} is equivalent to the following equation
\begin{equation}
\left\{
\begin{aligned}\label{Pe}
&\dot{P}(t)+P(t)\widehat{A}(t)+\widehat{A}^\top(t)P(t)+\widehat{C}^\top(t)P(t)\widehat{C}(t)+\widehat{Q}(t)=0,\\
&P(T)=G.
\end{aligned}
\right.
\end{equation}
Now, we will use modified iterative method and mathematical induction to prove the existence of a solution to \eqref{Pe}, thus \eqref{P} also has a solution. To do this, we set
\begin{equation}
\left\{
\begin{aligned}\label{P0}
&\dot{P}_0(t)+P_0(t)A(t)+A^\top(t)P_0(t)+C^\top(t)P_0(t)C(t)+Q(t)=0,\\
&P_0(T)=G,
\end{aligned}
\right.
\end{equation}
which admits a unique $P_0(\cdot)\in C([0,T];\mathcal{S}_+^n)$ by Lemma 7.3 of Chapter 6 in \cite{Yong1999}.
For $i\ge0$, we define
\begin{equation}
\left\{
\begin{aligned}\label{iterative}
&\Psi_i(t)=(R(t)+D^{\top}(t)P_i(t)D(t)+\widetilde{D}^{\top}(t)P_i(t)\widetilde{D}(t))^{-1}
(P_i(t)B(t)\\&\qquad\qquad+C^\top(t)P_i(t)D(t))^\top,\\
&\widehat{A}_i(t)=A(t)-B(t)\Psi_i(t),\quad \widehat{C}_i(t)=C(t)-D(t)\Psi_i(t),\\
&\widehat{Q}_i(t)=Q(t)+\Psi_i^{\top}(t)(R(t)+\widetilde{D}^{\top}(t)P_i(t)\widetilde{D}(t))\Psi_i(t),\\
\end{aligned}
\right.
\end{equation}
and let $P_{i+1}(t)$ be defined by the following equation
\begin{equation}
\left\{
\begin{aligned}\label{Pie}
&\dot{P}_{i+1}(t)+P_{i+1}(t)\widehat{A}_i(t)+\widehat{A}_i(t)P_{i+1}(t)+\widehat{C}_i^\top(t)P_{i+1}(t)\widehat{C}_i(t)+\widehat{Q}_i(t)=0,\\
&P_{i+1}(T)=G.
\end{aligned}
\right.
\end{equation}
Noticing that $R\gg0$, $Q\ge0$ and $G\geq0$, by using Lemma 7.3 of Chapter 6 in \cite{Yong1999} and mathematical induction, one can check that for $i\ge 0$, $P_i(\cdot)$ is well defined and moreover $P_i(\cdot)\in C([0,T];\mathcal{S}_+^n)$.
We claim that $P_i(\cdot)$, for $i\ge0$, is a decreasing sequence in $C([0,T];\mathcal{S}_+^n)$. For simplicity, we set $\Psi_{-1}(t)=0$ and denote $\Delta_i(t)=P_i(t)-P_{i+1}(t)$ and $\Upsilon_i(t)=\Psi_{i-1}(t)-\Psi_i(t)$. Indeed, when $i=0$, by \eqref{P0}-\eqref{Pie}, we have
\begin{equation}
\begin{aligned}
&-[\dot{\Delta}_0(t)+\Delta_0(t)\widehat{A}_0(t)+\widehat{A}_0^\top(t)\Delta_0(t)+\widehat{C}_0^\top(t)\Delta_0(t)\widehat{C}_0(t)]\\
=&P_0(t)(A(t)-\widehat{A}_0(t))+(A(t)-\widehat{A}_0(t))^\top P_0(t)+C^\top(t)P_0(t)C(t)\\
&-\widehat{C}_0^\top(t)P_0(t)\widehat{C}_0(t)+Q(t)-\widehat{Q}_0(t)\\
=&\Upsilon^\top_0(t)(R(t)+D^\top(t)P_0(t)D(t)+\widetilde{D}^\top(t)P_0(t)\widetilde{D}(t))\Upsilon_0(t)\\
&-[P_0(t)B(t)+\widehat{C}_0^\top(t)P_0(t)D (t)+\Upsilon^\top_0(t)(R(t)+\widetilde{D}^\top(t)P_0(t)\widetilde{D}(t))]\Upsilon_0(t)\\
&-\Upsilon^\top_0(t)[B^\top(t)P_0(t)+D^\top(t)P_0(t)\widehat{C}_0(t)+(R(t)+\widetilde{D}^\top(t)P_0(t)\widetilde{D}(t))\Upsilon_0(t)]\\
=&\Upsilon^\top_0(t)(R(t)+D^\top(t)P_0(t)D(t)+\widetilde{D}^\top(t)P_0(t)\widetilde{D}(t))\Upsilon_0(t)\geq 0.
\end{aligned}
\end{equation}
Using $\Delta_0(T)=0$ and Lemma 7.3 of \cite{Yong1999}, we get $P_0(t)\geq P_1(t)$, for all $t\in [0,T]$. For $i\ge1$, if assume that $P_{i-1}(t)\geq P_i(t)$, $t\in[0,T]$, it is sufficient to prove $P_{i}(t)\geq P_{i+1}(t)$, $t\in[0,T]$.
By using \eqref{Pie}, we have that $\Delta_i(t)$ satisfies
\begin{equation}
\begin{aligned}\label{deltak}
-\dot{\Delta}_i(t)=&\Delta_i(t)\widehat{A}_i(t)+\widehat{A}_i^\top(t)\Delta_i(t)+\widehat{C}_i^\top(t)\Delta_i(t)\widehat{C}_i(t)
+P_i(t)(\widehat{A}_{i-1}(t)-\widehat{A}_i(t))\\
&+(\widehat{A}_{i-1}(t)-\widehat{A}_i(t))^\top P_i(t)+\widehat{C}_{i-1}^\top(t)P_i(t)\widehat{C}_{i-1}(t)\\
&-\widehat{C}_i^\top(t)P_i(t)\widehat{C}_i(t)+\widehat{Q}_{i-1}(t)-\widehat{Q}_i(t).
\end{aligned}
\end{equation}
According to \eqref{iterative}, we have
\begin{equation*}
\begin{aligned}
&\widehat{A}_{i-1}(t)-\widehat{A}_i(t)=-B(t)\Upsilon_i(t),~~\widehat{C}_{i-1}(t)-\widehat{C}_{i}(t)=-D(t)\Upsilon_i(t),\\
&\widehat{C}_{i-1}^\top(t)P_i(t)\widehat{C}_{i-1}(t)-\widehat{C}_i^\top(t)P_i(t)\widehat{C}_i(t)=\Upsilon_i^\top(t)D^\top(t)P_i(t)D(t)\Upsilon_i(t)\\
&-\widehat{C}_i^\top(t)P_i(t)D(t)\Upsilon_i(t)-\Upsilon_i^\top(t)D^\top(t)P_i(t)\widehat{C}_i(t),\\
&\widehat{Q}_{i-1}(t)-\widehat{Q}_i(t)=\Upsilon_i^\top(t)(R(t)+\widetilde{D}^\top(t)P_i(t)\widetilde{D}(t))\Upsilon_i(t)\\
&+\Psi_i^\top(t)(R(t)+\widetilde{D}^\top(t)P_i(t)\widetilde{D}(t))\Upsilon_i(t)+\Upsilon_i^\top(t)(R(t)+\widetilde{D}^\top(t)P_i(t)\widetilde{D}(t))\Psi_i(t)\\
&+\Psi_{i-1}^\top(t)\widetilde{D}^\top(t)(P_{i-1}(t)-P_i(t))\widetilde{D}(t)\Psi_{i-1}(t).
\end{aligned}
\end{equation*}
From \eqref{deltak} and above estimates as well as $P_{i-1}(t)\geq P_i(t)$, we obtain
\begin{equation}
\begin{aligned}\label{di}
&-[\dot{\Delta}_i(t)+\Delta_i(t)\widehat{A}_i(t)+\widehat{A}_i^\top(t)\Delta_i(t)+\widehat{C}_i^\top(t)\Delta_i(t)\widehat{C}_i(t)]
=-P_i(t)B(t)\Upsilon_i(t)\\
&-\Upsilon_i^\top(t)B^\top(t)P_i(t)+\Upsilon_i^\top(t)D^\top(t)P_i(t)D(t)\Upsilon_i(t)-\widehat{C}_i^\top(t)P_i(t)D(t)\Upsilon_i(t)\\
&-\Upsilon_i^\top(t)D^\top(t)P_i(t)\widehat{C}_i(t)+\Upsilon_i^\top(t)(R(t)+\widetilde{D}^\top(t)P_i(t)\widetilde{D}(t))\Upsilon_i(t)\\
&+\Psi_i^\top(t)(R(t)+\widetilde{D}^\top(t)P_i(t)\widetilde{D}(t))\Upsilon_i(t)+\Upsilon_i^\top(t)(R(t)+\widetilde{D}^\top(t)P_i(t)\widetilde{D}(t))\Psi_i(t)\\
&+\Psi_{i-1}^\top(t)\widetilde{D}^\top(t)(P_{i-1}(t)-P_i(t))\widetilde{D}(t)\Psi_{i-1}(t)\\
=&\Upsilon_i^\top(t)(R(t)+D^\top(t)P_i(t)D(t)+\widetilde{D}^\top(t)P_i(t)\widetilde{D}(t))\Upsilon_i(t)\\
&+\Psi_{i-1}^\top(t)\widetilde{D}^\top(t)(P_{i-1}(t)-P_i(t))\widetilde{D}(t)\Psi_{i-1}(t)\geq 0.
\end{aligned}
\end{equation}
Using $\Delta_i(T)=0$ and Lemma 7.3 of \cite{Yong1999} again, we have  $P_{i}(t)\geq P_{i+1}(t)$, $t\in[0,T]$. Therefore, $\{P_i(\cdot)\}$ is a decreasing sequence in $C([0,T];\mathcal{S}_+^n)$, and thus has a limit denoted by $P(\cdot)$.  To show that $P(\cdot)$ solves \eqref{Pe} (and hence \eqref{P}), it remains to prove that $\{P_i(\cdot)\}$ is a Cauchy sequence in $C([0,T];\mathcal{S}_+^n)$  and $\{\dot{P}_i(\cdot)\}$ is a Cauchy sequence in $C([0,T];\mathcal{S}^n)$, which also provides that $P(\cdot)\in C([0,T];\mathcal{S}_+^n)$ and $\dot{P}(\cdot)\in C([0,T];\mathcal{S}^n)$.

In fact, let $\widetilde{R}_{i}(t)=R(t)+D^{\top}(t)P_i(t)D(t)+\widetilde{D}^{\top}(t)P_i(t)\widetilde{D}(t)$, one can get
\begin{equation}
\begin{aligned}\label{Upsilon}
&\Upsilon_i(t)=\Psi_{i-1}(t)-\Psi_i(t)\\
=&\widetilde{R}_{i-1}(t)^{-1}(B^\top(t)\Delta_{i-1}(t)+D^\top(t)\Delta_{i-1}(t)C(t))-\widetilde{R}_{i-1}(t)^{-1}(D^\top(t)\Delta_{i-1}(t)D(t)\\
&+\widetilde{D}^\top(t)
\Delta_{i-1}(t)\widetilde{D}(t))\widetilde{R}_{i}(t)^{-1}(B^\top(t)P_{i}(t)+D^\top(t)P_{i}(t)C(t)).
\end{aligned}
\end{equation}
Noting $\Delta_i(T)=0$ and integrating on both side of \eqref{di}, we get
\begin{equation}
\begin{aligned}\label{Delta}
\Delta_i(t)=&\int_t^T [\Delta_i(s)\widehat{A}_i(s)+\widehat{A}_i^\top(s)\Delta_i(s)+\widehat{C}_i^\top(s)\Delta_i(s)\widehat{C}_i(s)\\
&+\Upsilon_i^\top(s)\widetilde{R}_{i}(s)\Upsilon_i(s)-\Psi_{i-1}^\top(s)\widetilde{D}^\top(s)\Delta_{i-1}(s)\widetilde{D}(s)\Psi_{i-1}(s)]ds.
\end{aligned}
\end{equation}
Substituting \eqref{Upsilon} into \eqref{Delta},  and by noticing assumptions \textup{(H1)}-\textup{(H3)} and the uniformly boundedness of $|P_i(\cdot)|$, $|\widetilde{R}_{i}(\cdot)|$, $|\widetilde{R}_{i}(\cdot)^{-1}|$, we obtain
\begin{equation*}
|\Delta_i(t)|\leq K\int_t^T [|\Delta_{i-1}(s)|+|\Delta_i(s)|]ds.
\end{equation*}
Using Gronwall's inequality, we get $|\Delta_i(t)|\leq K\int_t^T |\Delta_{i-1}(s)|ds$. By iteration, we deduce
\begin{equation}\label{Deltaiestimate}
|\Delta_i(t)|\leq \frac{K^{i}}{(i-1)!}(T-t)^{i-1}|v_1(0)|, \text { \quad where } v_1(0)=\int_0^T |\Delta_{0}(s)|ds.
\end{equation}
For any $m>i\ge1$, we have
\begin{equation*}
\begin{aligned}
|P_i(t)-P_m(t)|
\leq|\Delta_i(t)|+|\Delta_{i+1}(t)|+\ldots|\Delta_{m-1}(t)|\leq \underset{j=i}{\overset{m}{\sum}} \frac{K^{j}}{(j-1)!}(T-t)^{j-1}|v_1(0)|,
\end{aligned}
\end{equation*}
and thus $\underset{0\leq t \leq T}{\sup}|P_i(t)-P_m(t)|\leq \underset{j=i}{\overset{m}{\sum}} \frac{K^{j}}{(j-1)!}T^{j-1}|v_1(0)|$. Hence, $\{P_i(\cdot)\}$ is a Cauchy sequence in $C([0,T];\mathcal{S}_+^n)$.  Moreover, from \eqref{Delta},  we have
\begin{equation}
\begin{aligned}
-\dot{\Delta}_i(t)=& \Delta_i(t)\widehat{A}_i(t)+\widehat{A}_i^\top(t)\Delta_i(t)+\widehat{C}_i^\top(t)\Delta_i(t)
\widehat{C}_i(t)\\
&+\Upsilon_i^\top(t)\widetilde{R}_{i}(t)\Upsilon_i(t)-\Psi_{i-1}^\top(t)\widetilde{D}^\top(t)
\Delta_{i-1}(t)\widetilde{D}(t)\Psi_{i-1}(t),
\end{aligned}
\end{equation}
and by using \eqref{Upsilon}, assumptions \textup{(H1)}-\textup{(H3)} and uniformly boundedness of $|P_i(\cdot)|$, $|\widetilde{R}_{i}(\cdot)|$, $|\widetilde{R}_{i}(\cdot)^{-1}|$, we have
$|\dot{\Delta}_i(t)|\leq K(|\Delta_i(t)|+|\Delta_{i-1}(t)|).
$
Then from \eqref{Deltaiestimate} and similar arguments as above, we can obtain
$\{\dot{P}_i(\cdot)\}$ is a Cauchy sequence in $C([0,T];\mathcal{S}^n)$. The proof is complete.    \hfill$\square$
\end{proof}

Now, we are going to give the well-posedness of system \eqref{RCC}. As discussed at the beginning of this subsection, equation \eqref{Lambda} for $\Lambda(\cdot)$ is not a standard Riccati equation too, since $\delta P(\cdot)-Q(\cdot)\geq 0$ fails usually. Inspired by Yong \cite{Yong2013}, we transform the solvability of $\Lambda(\cdot)$ to the solvability of another Riccati equation, whose well-posedness can be guaranteed through some algebraic inequalities. In fact, we have the following theorem, which is the main result of this subsection.
\begin{theorem}\label{RCCtheorem}
Let assumptions \textup{(H1)}-\textup{(H3)} hold, then Riccati type CC system \eqref{RCC} admits a unique solution $(P(\cdot),\Lambda(\cdot),\Phi(\cdot),l(\cdot))\in $ $C([0,T];\mathcal{S}_+^n\times \mathcal{S}^n\times \mathbb{R}^n\times\mathbb{R}^n)$.
\end{theorem}

\begin{proof}
From Lemma \ref{Plemma}, we know that equation \eqref{P} admits a unique solution in $C([0,T];\mathcal{S}_+^n)$. Now let us study the well-posedness of $\Lambda(\cdot)$. Motivated by \cite{Yong2013}, we set $\Pi(\cdot):=P(\cdot)+\Lambda(\cdot)$, then from  system  \eqref{RCC}, we know that $\Pi(\cdot)$ solves the following equation
\begin{equation}
\left\{
\begin{aligned}\label{Pi}
&\dot{\Pi}(t)+\Pi(t)(A(t)-B(t)\widetilde{R}(t)^{-1}D^{\top}(t)P(t)C(t))+(A(t)\\
&\qquad-B(t)\widetilde{R}(t)^{-1}D^{\top}(t)P(t)C(t))^{\top}\Pi(t)+\delta\Pi(t)+C^{\top}(t)(P(t)\\
&\qquad-P(t)D(t)\widetilde{R}(t)^{-1}D^{\top}(t)P(t))C(t)-\Pi(t)B(t)\widetilde{R}(t)^{-1}B^{\top}(t)\Pi(t)=0,\\
&\Pi(T)=0.
\end{aligned}
\right.
\end{equation}
By recalling that
$\widetilde{R}(t)=R(t)+D^{\top}(t)P(t)D(t)+\widetilde{D}^{\top}(t)P(t)\widetilde{D}(t)$, one can get
\begin{equation*}
\begin{aligned}
&P(t)-P(t)D(t)\widetilde{R}(t)^{-1}D^{\top}(t)P(t)\\
=&P(t)-P(t)D(t)(R(t)+D^{\top}(t)P(t)D(t)+\widetilde{D}^{\top}(t)P(t)\widetilde{D}(t))^{-1}D^{\top}(t)P(t)\\
=&P(t)^{\frac{1}{2}}[I-P(t)^{\frac{1}{2}}D(t)R(t)^{-\frac{1}{2}}(I+R(t)^{-\frac{1}{2}}D^{\top}(t)P(t)^{\frac{1}{2}}P(t)^{\frac{1}{2}}D(t)R(t)^{-\frac{1}{2}}\\
&+R(t)^{-\frac{1}{2}}\widetilde{D}^{\top}(t)P(t)^{\frac{1}{2}}P(t)^{\frac{1}{2}}\widetilde{D}(t)R(t)^{-\frac{1}{2}})^{-1}R(t)^{-\frac{1}{2}}D^{\top}(t)P(t)^{\frac{1}{2}}]P(t)^{\frac{1}{2}}\\
=&P(t)^{\frac{1}{2}}[I-\Sigma(t)(I+\Sigma^{\top}(t)\Sigma(t)+\widetilde{\Sigma}^{\top}(t)\widetilde{\Sigma}(t))^{-1}\Sigma^{\top}(t)]P(t)^{\frac{1}{2}},
\end{aligned}
\end{equation*}
where
$
\Sigma(t)=P(t)^{\frac{1}{2}}D(t)R(t)^{-\frac{1}{2}}$ and $\widetilde{\Sigma}(t)=P(t)^{\frac{1}{2}}\widetilde{D}(t)R(t)^{-\frac{1}{2}}.$
From Lemma \ref{alegbra1}, we have
\[
(I+\Sigma^{\top}(t)\Sigma(t)+\widetilde{\Sigma}^{\top}(t)\widetilde{\Sigma}(t))^{-1}\leq (I+\Sigma^{\top}(t)\Sigma(t))^{-1},
\]
and further
\begin{equation}
\begin{aligned}\label{inequality}
I-\Sigma(t)(I+\Sigma^{\top}(t)\Sigma(t)+\widetilde{\Sigma}^{\top}(t)\widetilde{\Sigma}(t))^{-1}
\Sigma^{\top}(t)
\geq I-\Sigma(t)(I+\Sigma^{\top}(t)\Sigma(t))^{-1}\Sigma^{\top}(t).
\end{aligned}
\end{equation}
Noting that $I-\Sigma(t)(I+\Sigma^{\top}(t)\Sigma(t))^{-1}\Sigma^{\top}(t)=(I+\Sigma(t)\Sigma^{\top}(t))^{-1}$,
we have
\begin{equation*}
\begin{aligned}
&P(t)-P(t)D(t)\widetilde{R}(t)^{-1}D^{\top}(t)P(t)\geq P(t)^{\frac{1}{2}}(I+\Sigma(t)\Sigma^{\top}(t))^{-1}P(t)^{\frac{1}{2}}\\
=&P(t)^{\frac{1}{2}}(I+P(t)^{\frac{1}{2}}D(t)R(t)^{-1}D^{\top}(t)P(t)^{\frac{1}{2}})^{-1}P(t)^{\frac{1}{2}}\geq 0.
\end{aligned}
\end{equation*}
Therefore, the following inequality holds
\begin{equation}\label{positive}
C^{\top}(t)(P(t)-P(t)D(t)\widetilde{R}(t)^{-1}D^{\top}(t)P(t))C(t)\geq 0
\end{equation}
and noticing that $R(t)+D^{\top}(t)P(t)D(t)+\widetilde{D}^{\top}(t)P(t)\widetilde{D}(t)>0$, we obtain that equation \eqref{Pi} admits a unique solution $\Pi(\cdot)\in C([0,T];\mathcal{S}_+^n)$. By recalling  $\Pi(\cdot)=P(\cdot)+\Lambda(\cdot)$ and Lemma \ref{Plemma}, we get the well-posedness of equation \eqref{Lambda}.

Once $P(\cdot)$ and $\Lambda(\cdot)$ are uniquely solved, the well-posedness of $\Phi(\cdot)$ holds by noting that the equation for $\Phi(\cdot)$ is just an ODE. Then similarly, the well-posedness of $l(\cdot)$ holds also.     \hfill$\square$
\end{proof}

\begin{remark}\label{Riccatiemph}
In this subsection, the well-posedness of a new Riccati type CC system \eqref{RCC}  is given, which generalizes the results of \cite{Huang2016,Yong2013}. To overcome the difficulties arisen from our general partial information structure, we introduce modified iterative method and equivalent transform method motivated by \cite{Yong1999,Yong2013}. We mention that our Riccati equations are quite different to the ones in \cite{Huang2016,Yong2013} due to the additional term  $\widetilde{D}^{\top}(\cdot)P(\cdot)\widetilde{D}(\cdot)$. However, it is interesting that we can use some algebraic inequalities as well as modified iterative method and equivalent transform method to obtain the well-posedness of our new Riccati type CC system.
\end{remark}

\section{Applications}\label{sec:app}
In this section, we apply our theoretical results to solve the Example \ref{example}. Let $\mathbb{R}_+$ be the set of all positive real number. For any real number $x$, we denote $x^+=\max\{x,0 \}$. Let admissible control set be $\mathcal{U}_{ad}^{c}=\{u_{i}(\cdot )~|~u_{i}(\cdot )\in L_{\mathcal{G}_{t}}^{2}(0,T;\Gamma)\}$, where $\Gamma\subseteq\mathbb{R}$. Then the general inter-bank borrowing and lending problem can be formulated as follows.

\textbf{Problem (IBL)} For $1\leq i \leq N$, to find strategy profile $\bar{u}(\cdot)=(\bar{u}_1(\cdot),\ldots,\bar{u}_N(\cdot))$, where $\bar{u}_i(\cdot)\in\mathcal{U}_{ad}^{c}$, such that
$
\mathcal{J}_i(\bar{u}_i(\cdot),\bar{u}_{-i}(\cdot))=\underset{%
u_{i}\left( \cdot \right) \in \mathcal{U}_{ad}^{c}}{\inf }\mathcal{J}_i(u_i(\cdot),\bar{u}_{-i}(\cdot))$, subjects to \eqref{exstate} and \eqref{excost}.

\begin{remark}
Noting that the cost functional in \cite{Carmona2015} has the cross term between $u_i(\cdot)$ and $(x_i(\cdot)-x^{(N)}(\cdot))$. We claim that our following conclusions can be extended similarly to such case with no essential difficulties.
\end{remark}
When $\Gamma=\mathbb{R}_+$, the decentralized strategies are given by
\begin{equation}\label{exampleu}
\bar{u}_{i}(t)=\Big\{\frac{B\mathbb{E}[\bar{p}%
_{i}(t)|\mathcal{G}_{t}^{i}]+D\mathbb{E}[\bar{k}_{i}(t)|\mathcal{G%
}_{t}^{i}]+\widetilde{D}\mathbb{E}[\bar{\widetilde{k}}_{i}(t)|%
\mathcal{G}_{t}^{i}]}{r}\Big\}^+,
\end{equation}%
where $(\bar{z}_{i}(\cdot),\bar{p}_{i}(\cdot),\bar{k}_{i}(\cdot),\bar{\widetilde{k}}_{i}(\cdot))$ solves the following  MF-FBSDE
\begin{equation}
\left\{
\begin{aligned}\label{exaHCC}
d\bar{z}_{i}(t)=&\{(A-a)\bar{z}_{i}(t)+B
[r^{-1}(B\mathbb{E}[\bar{p}_{i}(t)|\mathcal{G}%
_{t}^{i}]+D\mathbb{E}[\bar{k}_{i}(t)|\mathcal{G}_{t}^{i}] \\
&+\widetilde{D}\mathbb{E}[\bar{\widetilde{k}}_{i}(t)|\mathcal{G}%
_{t}^{i}])]^++a\mathbb{E}[\bar{z}_{i}(t)]+b\}dt\\
&+\{C\bar{z}_{i}(t) +D[r^{-1}(B\mathbb{E}[\bar{p}_{i}(t)|%
\mathcal{G}_{t}^{i}]+D\mathbb{E}[\bar{k}_{i}(t)|\mathcal{G}%
_{t}^{i}]\\
&+\widetilde{D}\mathbb{E}[\bar{\widetilde{k}}_{i}(t)|\mathcal{G}%
_{t}^{i}])]^++\sigma \}dW_{i}(t)\\
&+\{\widetilde{D}[r^{-1}(B\mathbb{E}[\bar{p%
}_{i}(t)|\mathcal{G}_{t}^{i}]+D\mathbb{E}[\bar{k}_{i} (t)|%
\mathcal{G}_{t}^{i}]\\
&+\widetilde{D}\mathbb{E}[\bar{\widetilde{k}}_{i}(t)|\mathcal{G}%
_{t}^{i}])]^++\widetilde{\sigma }\}d\widetilde{W}%
_{i}(t), \\
d\bar{p}_{i}(t)=&-[(A-a)\bar{p}_{i}(t)+C\bar{k}_{i} \left(
t\right)
-\epsilon(\bar{z}_{i}\left( t\right) -\mathbb{E}[\bar{z}_{i}(t)])]dt\\
&+\bar{k}_{i}\left( t\right) dW_{i}(t)+\bar{\widetilde{k}}_{i}
\left( t\right) d\widetilde{W}_{i}(t), \\
\bar{z}_{i}(0)=&x,~\bar{p}_{i}(T)=-c(\bar{z}_{i}\left( T\right) -\mathbb{E}[\bar{z}_{i}(T)]).
\end{aligned}
\right.
\end{equation}%
From Theorems \ref{Hwellposedness} and \ref{generalnash}, the following results hold.
\begin{theorem}
Assume that $2A<a-3C^2$, there exists a constant $\theta_1>0$ independent of $T$, which may depend on $A$, $a$, $C$, $\epsilon$ and $c$, when $B$, $D$, $\widetilde{D}$ and $r^{-1}\in[0,\theta_1)$, then there exists a unique adapted solution $(\bar{z}_i(\cdot),\bar{p}_i(\cdot),\bar{k}_i(\cdot),\bar{\widetilde{k}}_i(\cdot))\in L_{\mathcal{F
}_{t}^{W,\widetilde{W}}}^{2}(0,T;\mathbb{R}\times \mathbb{R}\times \mathbb{R}\times \mathbb{R})$ to MF-FBSDE \eqref{exaHCC}. Moreover, the strategy profile  $\bar{u}(\cdot)=(\bar{u}_1(\cdot),\ldots,\bar{u}_N(\cdot))$, where $\bar{u}_i(\cdot)$ is given by \eqref{exampleu}, is an $\varepsilon$-Nash equilibrium of Problem (IBL).
\end{theorem}

If $\Gamma=\mathbb{R}$, we have $\mathcal{U}_{ad}^{c}=\{u_{i}(\cdot )~|~u_{i}(\cdot )\in L_{\mathcal{G}_{t}}^{2}(0,T;\mathbb{R})\}$. By applying Theorem \ref{specialu}, we can represent the decentralized strategies as the following feedback of filtered state
\begin{equation}\label{exafeed}
\bar{u}_i(t)=-\frac{(P(t)B+CP(t)D)\hat{\bar{z}}_i(t)+B\Lambda(t) l(t)+B\Phi(t)+DP(t)\sigma+DP(t)\widetilde{\sigma}}{r+D^2P(t)+\widetilde{D}^2P(t)},
\end{equation}
where $P(\cdot),\Lambda(\cdot),\Phi(\cdot),l(\cdot)$ solve, respectively, the following equations
\begin{equation}
\label{exaP}
\dot{P}(t)+2(A-a)P(t)+C^2P(t)+\epsilon-\frac{(P(t)B+CP(t)D)^2}{r+D^2P(t)+\widetilde{D}^2P(t)}
=0,\qquad
P(T)=c,
\end{equation}
and
\begin{equation}\label{exaL}
\begin{aligned}
&\dot{\Lambda}(t)+2\left(A-a-\frac{B(P(t)B+CP(t)D)}{r+D^2P(t)+\widetilde{D}^2P(t)}\right)\Lambda(t)\\
&\qquad+(P(t)+\Lambda(t))a -\frac{\Lambda(t)^2B^2}{r+D^2P(t)+\widetilde{D}^2P(t)}-\epsilon=0,\qquad\Lambda(T)=-c,
\end{aligned}
\end{equation}
and
\begin{equation}\label{exaPhi}
\begin{aligned}
&\dot{\Phi}(t)+\left(A-a-\frac{B(P(t)B+CP(t)D)}{r+D^2P(t)+\widetilde{D}^2P(t)}-\frac{\Lambda(t)B^2}{r+D^2P(t)+\widetilde{D}^2P(t)}\right)\Phi(t)\\
&+\Bigg(C-\frac{D(P(t)B+CP(t)D)}{r+D^2P(t)+\widetilde{D}^2P(t)}-\frac{\Lambda(t)BD}{r+D^2P(t)+\widetilde{D}^2P(t)}\Bigg)P(t)\sigma\\
&-\Bigg(\frac{\widetilde{D}(P(t)B+CP(t)D)}{r+D^2P(t)+\widetilde{D}^2P(t)}+\frac{\Lambda(t)B\widetilde{D}}{r+D^2P(t)+\widetilde{D}^2P(t)}\Bigg)P(t)\widetilde{\sigma}\\
&+(P(t)+\Lambda(t))b=0, \qquad
\Phi(T)=0,
\end{aligned}
\end{equation}
and
\begin{equation}
\begin{aligned}\label{exal}
dl(t)=&\Bigg\{\left[A-\frac{B(P(t)B+CP(t)D+B\Lambda(t))}{r+D^2P(t)+\widetilde{D}^2P(t)}\right]l(t)+b\\
&-\frac{B(B\Phi(t)+DP(t)\sigma+\widetilde{D}P(t)\widetilde{\sigma})}{r+D^2P(t)+\widetilde{D}^2P(t)}
\Bigg\}dt,\qquad\qquad\qquad
l(0)=x.
\end{aligned}
\end{equation}
The optimal filtering $\hat{\bar{z}}_i(\cdot)$ is determined by
\begin{equation}
\left\{
\begin{aligned}\label{exaz}
d\hat{\bar{z}}_i(t)=&~\Bigg\{\left(A-a-\frac{B(P(t)B+CP(t)D)}{r+D^2P(t)+\widetilde{D}^2P(t)}\right)\hat{\bar{z}}_i(t)-\frac{B^2\Lambda(t)}{r+D^2P(t)+\widetilde{D}^2P(t)}l(t)\\
&-\frac{B(B\Phi(t)+DP(t)\sigma+\widetilde{D}P(t)\widetilde{\sigma})}{r+D^2P(t)+\widetilde{D}^2P(t)}+b+al(t)\Bigg\}dt\\
&+\Bigg\{\left(C-\frac{D(P(t)B+CP(t)D)}{r+D^2P(t)+\widetilde{D}^2P(t)}\right)\hat{\bar{z}}_i(t)-\frac{DB\Lambda(t)}{r+D^2P(t)+\widetilde{D}^2P(t)}l(t)\\
&-\frac{D(B\Phi(t)+DP(t)\sigma+\widetilde{D}P(t)\widetilde{\sigma})}{r+D^2P(t)+\widetilde{D}^2P(t)}+\sigma\Bigg\}dW_i(t),\\
\hat{\bar{z}}_i(0)=&~x.
\end{aligned}
\right.
\end{equation}
Moreover, we have
\begin{theorem}
The strategy profile $\bar{u}(\cdot)=(\bar{u}_1(\cdot),\ldots,\bar{u}_N(\cdot))$, where $\bar{u}_i(\cdot)$ is given by \eqref{exafeed} and $(P(\cdot),\Lambda(\cdot),\Phi(\cdot),l(\cdot),\hat{\bar{z}}_i(\cdot))$ solves systems \eqref{exaP}-\eqref{exaz}, is an $\varepsilon$-Nash equilibrium of Problem (IBL) with $\Gamma=\mathbb{R}$.
\end{theorem}

Finally, for comparing our results with the results of  \cite{Carmona2015}, and also for further clarifications about the financial implication, we consider some special cases and present corresponding numerical results.

To do this, let $C=D=\widetilde{D}=0$, then we can derive a similar result as \cite{Carmona2015}. Indeed, in this setting, we have $P(t)=-\Lambda(t)$, $\Phi(t)\equiv0$. Moreover, equation \eqref{exaP} has the following explicit solution
\begin{equation*}\label{Psolution}
P(t)=\frac{r}{B^2}\frac{A-a+\theta+[-(A-a)+\theta]\frac{B^2c/r-(A-a+\theta)}{B^2c/r-(A-a)+\theta}e^{2\theta (t-T)}}{1-\frac{B^2c/r-(A-a+\theta)}{B^2c/r-(A-a)+\theta}e^{2\theta (t-T)}},
\end{equation*}
where $\theta=\sqrt{(A-a)^2+\frac{B^2\epsilon}{r}}$. At this moment, the decentralized strategies of Problem (IBL) read as
\begin{equation*}
\bar{u}_i(t)=\frac{1}{B}\frac{A-a+\theta+[-(A-a)+\theta]\frac{B^2c/r-(A-a+\theta)}{B^2c/r-(A-a)+\theta}e^{2\theta (t-T)}}{1-\frac{B^2c/r-(A-a+\theta)}{B^2c/r-(A-a)+\theta}e^{2\theta (t-T)}}(\hat{\bar{z}}_i(t)-l(t)),
\end{equation*}
where $l(t)=xe^{-At}-\frac{b}{A}(1-e^{At})$ and
\begin{equation*}
\hat{\bar{z}}_i(t)=\Theta(t)x+\Theta(t)\int_0^t\frac{(a-\frac{B^2P(s)}{r})l(s)+b  }{\Theta(s)}ds+\Theta(t)\int_0^t\frac{\sigma}{\Theta(s)} dW_i(s),
\end{equation*}
with
$
\Theta(t)=e^{\int_0^t(A-a-\frac{B^2P(s)}{r})ds}$.
Moreover, if $A=0$ and $r=B=1$, our results coincide with the results of \cite{Carmona2015} (see (5.8) in \cite{Carmona2015}).

In most instances, however, it is intractable to find an explicit solution to Problem (IBL). To better illustrate our results, we give some numerical simulations below. Suppose that $N=20$, $T=1$. Set the initial data as $x=1$, $A=3.2$, $a=1.5$, $B=2.8$, $b=2$, $C=0.6$, $\sigma=0.8$, $D=0$, $\widetilde{D}=2$, $\tilde{\sigma}=0.3$, $\epsilon=3.3$, $r=2.5$, $c=5$. Let $\Pi=P+\Lambda$, then Fig \ref{fig1} gives the numerical solutions of $P$, $\Lambda$, $\Pi$, $\Phi$ and $l$. In our numerical simulations, we will illustrate the influence of the partial information structure. For $D\neq0$ case, we can do similar numerical simulations and we prefer to omit it here.

\begin{figure}[H]
  \centering
  \includegraphics[width=5.4in,height=2.4in]{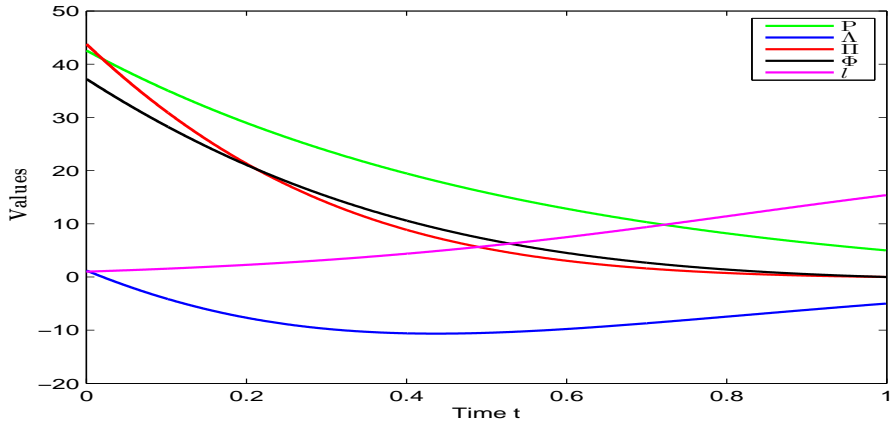}
  \caption{The numerical solutions of $P$, $\Lambda$, $\Pi$, $\Phi$ and $l$.}
  \label{fig1}
\end{figure}

Fig \ref{fig2} and Fig  \ref{fig3} show the optimal filtering of decentralized states and the decentralized control strategies of $20$ banks, which characterize the corresponding dynamics of log-monetary reverse and control rate of borrowing from or lending to a central bank.  Fig \ref{fig6} and Fig \ref{fig7} draw the optimal filtering of decentralized states and decentralized control strategies of $20$ banks when $\widetilde{D}$ becomes larger. By comparing Fig \ref{fig2} and Fig \ref{fig6} (resp. Fig \ref{fig3} and Fig \ref{fig7}), we find that the optimal filtering of decentralized states (resp.  the decentralized strategies) go up (resp. go down) when $\widetilde{D}$ is larger. In fact, the change of $\widetilde{D}$ will result the unknown information fluctuates greatly (eg. the central bank situation). Consequently, each bank communicates frequently with others for the unknown risk diversification, and reduces the cash flow between the central bank.

\begin{figure}[H]
\centering
\begin{minipage}[c]{0.48\textwidth}
\centering
\includegraphics[width=1\hsize=1]{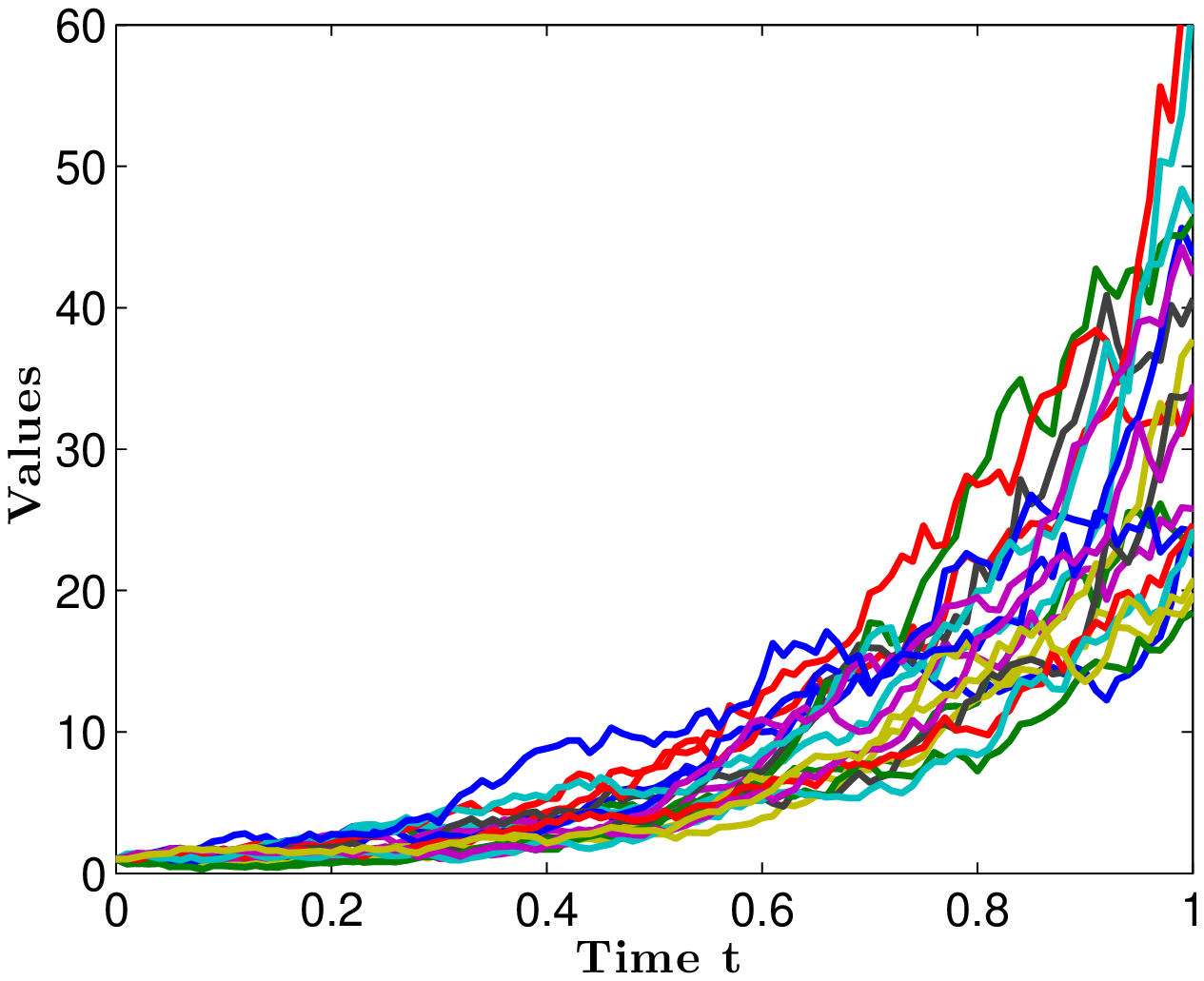}
\end{minipage}
\hspace{0.01\textwidth}
\begin{minipage}[c]{0.48\textwidth}
\centering
\includegraphics[width=1\hsize=1]{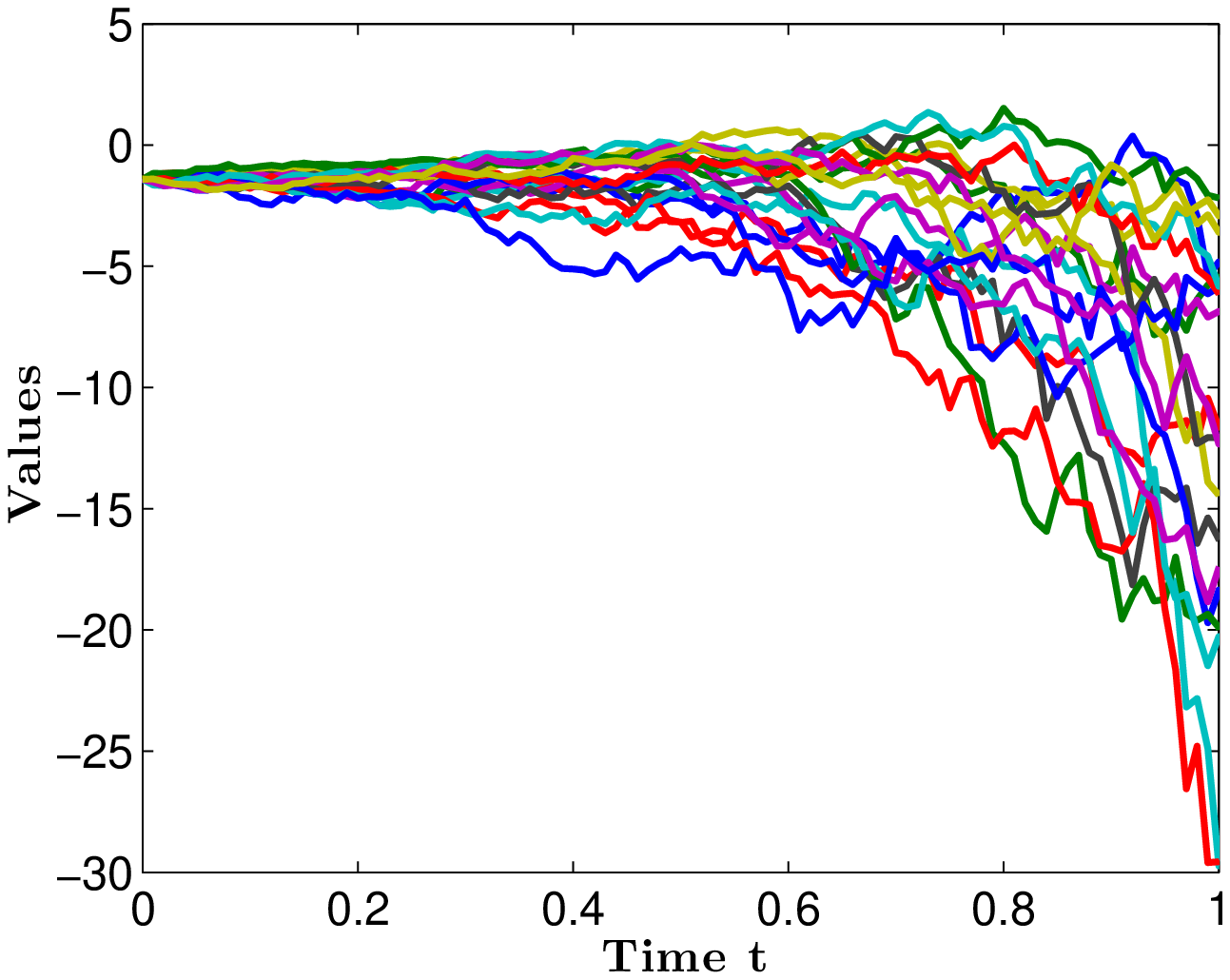}
\end{minipage}\\[3mm]
\begin{minipage}[t]{0.48\textwidth}
\centering
\caption{Optimal filtering of decentralized states when $C=0.6,\widetilde{D}=2$.}
\label{fig2}
\end{minipage}
\hspace{0.01\textwidth}
\begin{minipage}[t]{0.48\textwidth}
\centering
\caption{Decentralized control strategies when $C=0.6,\widetilde{D}=2$.}
\label{fig3}
\end{minipage}
\end{figure}

\begin{figure}[H]
\centering
\begin{minipage}[c]{0.48\textwidth}
\centering
\includegraphics[width=1 \hsize=1]{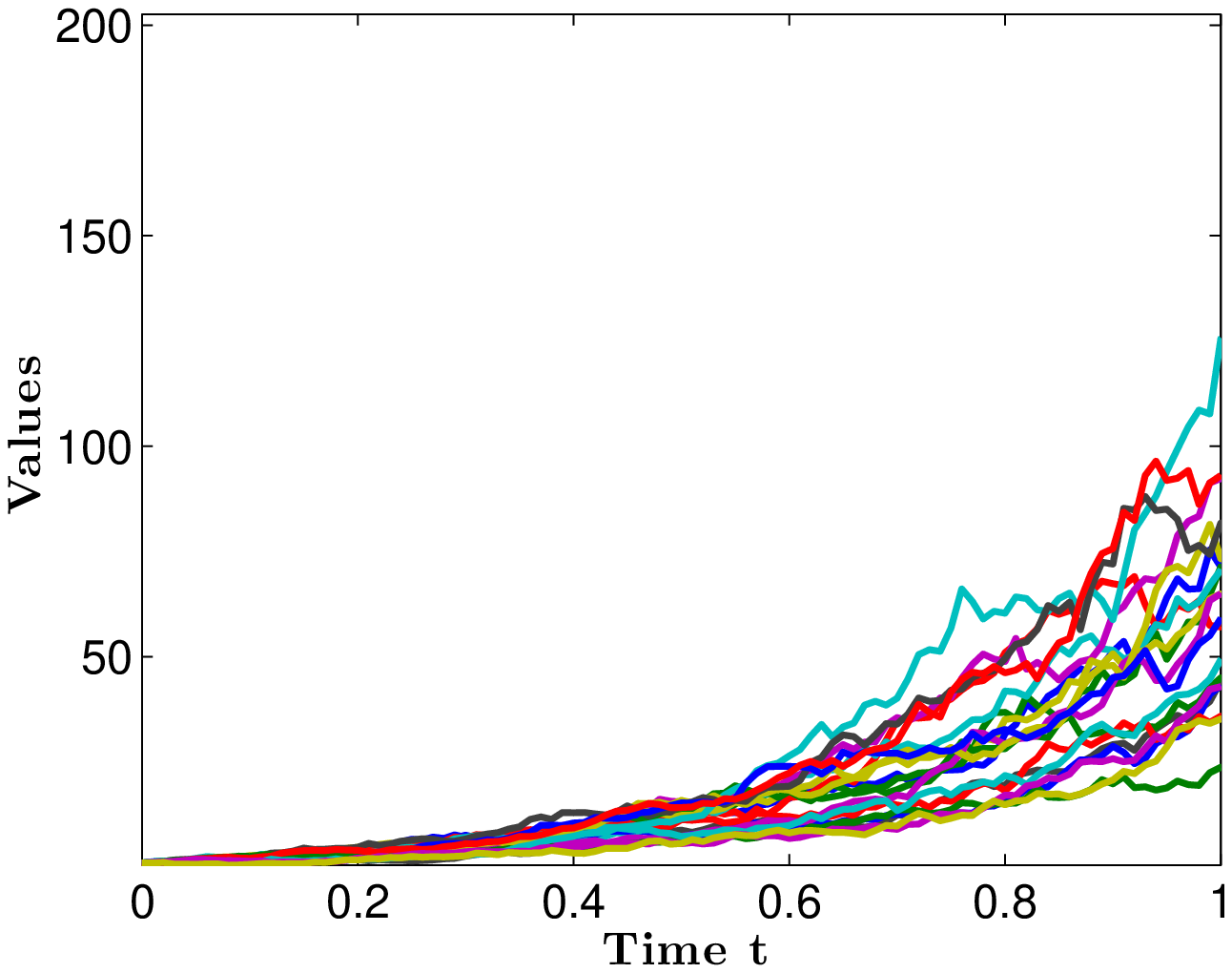}
\end{minipage}
\hspace{0.01\textwidth}
\begin{minipage}[c]{0.48\textwidth}
\centering
\includegraphics[width=1 \hsize=1]{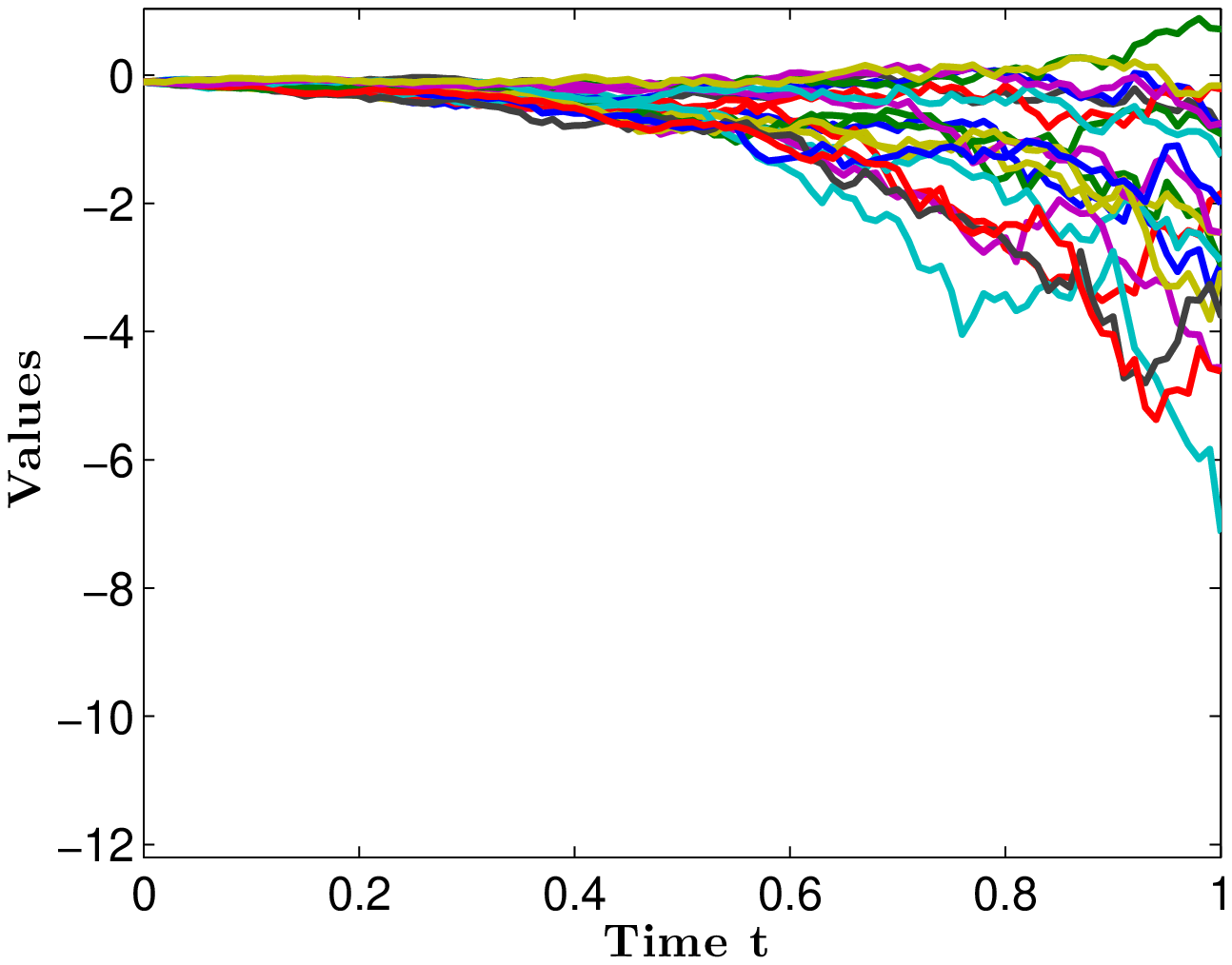}
\end{minipage}\\[3mm]
\begin{minipage}[t]{0.48\textwidth}
\centering
\caption{Optimal filtering of decentralized states when $C=0.6,\widetilde{D}=6$.}
\label{fig6}
\end{minipage}
\hspace{0.01\textwidth}
\begin{minipage}[t]{0.48\textwidth}
\centering
\caption{Decentralized control strategies when $C=0.6,\widetilde{D}=6$.}
\label{fig7}
\end{minipage}
\end{figure}

\section{Conclusions}\label{sec:con}
In this paper, a general stochastic large-population problem with partial information has been considered, where the diffusion of the dynamics of each agent can depend both on the state and the control. In control constrained case, by using Hamiltonian approach, we have obtained the decentralized strategies through a mixed nonlinear MF-FBSDE with projection operator, whose well-posedness has also been studied. Moreover, the corresponding $\varepsilon$-Nash equilibrium property has also been verified. In control unconstrained case, by using Riccati approach, the decentralized strategies can be further represented as the feedback of filtered state through a new Riccati type CC system, which  are quite different to the classical ones due to the additional term generated from partial information structure. We have used some algebraic inequalities as well as modified iterative method and equivalent transform method to obtain the well-posedness of our Riccati type CC system. As an application,  a general inter-bank borrowing and lending problem has been studied.

\newpage

\appendix
\section{Appendix} \label{appendix}
\setcounter{equation}{0}
\renewcommand{\theequation}{A.\arabic{equation}}

This appendix is devoted to proving Theorem \ref{wellposeness}.  For any fixed $(Y(\cdot ),Z(\cdot ),\widetilde{Z}(\cdot ))\in L_{\mathcal{F}_{t}^{W,\widetilde{W}}}^{2}(0,T;\mathbb{R}^{m}\times \mathbb{R}^{m\times d}\times \mathbb{R}^{m\times d})$, we define a mapping
\[\mathcal{V}_{1}:L_{\mathcal{F}_{t}^{W,\widetilde{W}}}^{2}(0,T;\mathbb{R}^{m}\times \mathbb{R}^{m\times d}\times \mathbb{R}^{m\times d})\longrightarrow L_{\mathcal{F%
}_{t}^{W,\widetilde{W}}}^{2}(0,T;\mathbb{R}^{n})\]
such that $\mathcal{V}_1(Y(\cdot ),Z(\cdot ),\widetilde{Z}(\cdot ))=X(\cdot )$, where $X$ solves the following SDE
\begin{equation}
\begin{aligned}
X(t)=&x+\int_{0}^{t}b(s,\Theta(s))ds+\int_{0}^{t}\sigma
(s,\Theta(s))dW(s)+\int_{0}^{t}\widetilde{\sigma}
(s,\Theta(s))d\widetilde{W}(s),\label{forward}
\end{aligned}
\end{equation}
with $\Theta(s)=(X(s),\mathbb{E}[X(s)],Y(s),\mathbb{E}[Y(s)|\mathcal{F}^W_s],Z(s),\mathbb{E}[Z(s)|\mathcal{F}^W_s],
\widetilde{Z}(s),\mathbb{E}[\widetilde{Z}(s)|\mathcal{F}^W_s]).$
Under assumptions \textup{(A1)}-\textup{(A3)} and \textup{(A5)}, it's easy to verify that the mapping $\mathcal{V}_1$ is well defined. Moreover, we have $\mathbb{E}\sup_{0\leq t\leq T}|X(t)|^2<\infty$.
\begin{proposition}\label{X}
Let \textup{(A1)}-\textup{(A3)} and \textup{(A5)} hold. Assume that $X_i(\cdot)$ is the solution of \eqref{forward} w.r.t. $(Y_i(\cdot),Z_i(\cdot),\widetilde{Z}_i(\cdot))$, $i=1,2$, then there exist positive constants $C_1$, $C_2$ and $C_3$, s.t. for all $\lambda\in\mathbb{R}$, it follows that

\begin{equation}
\begin{aligned}
&e^{-\lambda t}\mathbb{E}|X_{1}(t)-X_{2}(t)|^{2}+\bar{\lambda}%
_{1}\int_{0}^{t}e^{-\lambda s}\mathbb{E}|X_{1}(s)-X_{2}(s)|^{2}ds \\
\leq &(\rho _2C_1+\rho _{3}C_{1}+w_3^2+w_4^2+\kappa_3^2+\kappa_4^2)\int_{0}^{t}e^{-\lambda s}\mathbb{E}%
|Y_{1}(s)-Y_{2}(s)|^{2}ds\\
&+(\rho _4C_2+\rho _5C_{2}+w_5^2+w_6^2+\kappa_5^2+\kappa_6^2)\int_{0}^{t}e^{-\lambda s}\mathbb{E}|Z_{1}(s)-Z_{2}(s)|^{2}ds\\
&+(\rho _6C_3+\rho _7C_{3}+w_7^2+w_8^2+\kappa_7^2+\kappa_8^2)\int_{0}^{t}e^{-\lambda s}\mathbb{E}|\widetilde{Z}_{1}(s)-\widetilde{Z}_{2}(s)|^{2}ds,\label{E1}
\end{aligned}
\end{equation}
where $\bar{\lambda}_{1}=\lambda -2\lambda _{1}-(\rho _{2}+\rho_3)C_{1}^{-1}-(\rho_{4}+\rho_5)C_{2}^{-1}-(\rho_{6}+\rho_7)C_{3}^{-1}-2\rho _{1}-w_1^2-w_2^2-\kappa_1^2-\kappa_2^2$.
Moreover,
\begin{equation}
\begin{aligned}
&e^{-\lambda t}\mathbb{E}|X_{1}(t)-X_{2}(t)|^{2}\\
\leq &(\rho _2C_1+\rho _{3}C_{1}+w_3^2+w_4^2+\kappa_3^2+\kappa_4^2)\int_{0}^{t}e^{-\bar{\lambda}_1(t-s)}e^ {-\lambda s}\mathbb{E}%
|Y_{1}(s)-Y_{2}(s)|^{2}ds\\
&+(\rho _4C_2+\rho _5C_{2}+w_5^2+w_6^2+\kappa_5^2+\kappa_6^2)\int_{0}^{t}e^{-\bar{\lambda}_1(t-s)}e^ {-\lambda s}\mathbb{E}|Z_{1}(s)-Z_{2}(s)|^{2}ds\\
&+(\rho _6C_3+\rho _7C_{3}+w_7^2+w_8^2+\kappa_7^2+\kappa_8^2)\int_{0}^{t}e^{-\bar{\lambda}_1(t-s)}e^ {-\lambda s}\mathbb{E}|\widetilde{Z}_{1}(s)-\widetilde{Z}_{2}(s)|^{2}ds.\label{E2}
\end{aligned}
\end{equation}%
\end{proposition}

\begin{proof}
For simplicity, for $\varphi=X,Y,Z,\widetilde{Z}$  and $\phi=b,\sigma,\widetilde{\sigma}$, we denote $\widehat{\varphi}(s)=\varphi_1(s)-\varphi_2(s)$, and
\begin{equation*}
\begin{aligned}
\widehat{\phi}(s)=&\phi(s,X_1,\mathbb{E}[X_1],Y_1,\mathbb{E}[Y_1|\mathcal{F}^W_s],Z_1,\mathbb{E}[Z_1|\mathcal{F}^W_s],
\widetilde{Z}_1,\mathbb{E}[\widetilde{Z}_1|\mathcal{F}^W_s])\\
&-\phi(s,X_2,\mathbb{E}[X_2],Y_2,\mathbb{E}[Y_2|\mathcal{F}^W_s],Z_2,\mathbb{E}[Z_2|\mathcal{F}^W_s],
\widetilde{Z}_2,\mathbb{E}[\widetilde{Z}_2|\mathcal{F}^W_s]).
\end{aligned}
\end{equation*}%
Applying It\^o's formula to $e^{-\lambda t}|\widehat{X}(t)|^{2}$ and then taking expectation, we have
\begin{equation}
\begin{aligned}e^{-\lambda t}\mathbb{E}|\widehat{X}(t)|^{2}
 =&-\lambda
\int_{0}^{t}e^{-\lambda s}\mathbb{E}|\widehat{X}(s)|^{2}ds+2\mathbb{E}%
\int_{0}^{t}e^{-\lambda s}\langle \widehat{X}(s), \widehat{b}(s)\rangle ds \\
&+\mathbb{E}\int_{0}^{t}e^{-\lambda s}|\widehat{\sigma}(s)|^{2}ds+\mathbb{E}\int_{0}^{t}e^ {-\lambda s}|\widehat{\widetilde{\sigma}}(s)|^{2}ds.\label{e1}
\end{aligned}
\end{equation}%
On one hand, from \textup(A1), \textup(A2) and $|\mathbb{E}[V(s)|\mathcal{F}_s^W]|^2\leq \mathbb{E}[|V(s)|^2|\mathcal{F}_s^W]$, for $V(s)=\widehat{Y}(s),\widehat{Z}(s),\widehat{\widetilde{Z}}(s)$, it follows that

\begin{equation*}
\begin{aligned}
&2\langle \widehat{X}(s),\widehat{b}(s)\rangle\\
\leq&(2\lambda_1+\rho_2C_1^{-1}+\rho_3C_1^{-1}+\rho_4C_2^{-1}+\rho_5C_2^{-1}+\rho_6C_3^{-1}+\rho_7C_3^{-1})|\widehat{X}(s)|^2\\
&+\rho_2C_1|\widehat{Y}(s)|^2+\rho_3C_1\mathbb{E}[|\widehat{Y}(s)|^2|\mathcal{F}_s^W]+\rho_4C_2|\widehat{Z}(s)|^2
+\rho_5C_2\mathbb{E}[|\widehat{Z}(s)|^2|\mathcal{F}_s^W]\\
&+\rho_6C_3|\widehat{\widetilde{Z}}(s)|^2+\rho_7C_3\mathbb{E}[|\widehat{\widetilde{Z}}(s)|^2|\mathcal{F}_s^W]
+2\rho_1|\widehat{X}(s)||\mathbb{E}[\widehat{X}(s)]|.
\end{aligned}
\end{equation*}%
On the other hand, by Lipschitz condition of $\sigma$, and $|\mathbb{E}[\widehat{X}(s)]|^2\leq \mathbb{E}[|\widehat{X}(s)|^2]$, we get
\begin{equation*}
\begin{aligned}
|\widehat{\sigma}(t)|^2\leq&w_1^2|\widehat{X}(s)|^2+w_2^2|\mathbb{E}[\widehat{X}(s)]|^2+w_3^2|\widehat{Y}(s)|^2
+w_4^2|\mathbb{E}[\widehat{Y}(s)|\mathcal{F}_s^W]|^2\\
&+w_5^2|\widehat{Z}(s)|^2+w_6^2|\mathbb{E}[\widehat{Z}(s)|\mathcal{F}_s^W]|^2
+w_7^2|\widehat{\widetilde{Z}}(s)|^2+w_8^2|\mathbb{E}[\widehat{\widetilde{Z}}(s)|\mathcal{F}_s^W]|^2\\
\leq&+w_1^2|\widehat{X}(s)|^2+w_2^2\mathbb{E}|\widehat{X}(s)|^2+w_3^2|\widehat{Y}(s)|^2
+w_4^2\mathbb{E}[|\widehat{Y}(s)|^2|\mathcal{F}_s^W]\\
&+w_5^2|\widehat{Z}(s)|^2+w_6^2\mathbb{E}[|\widehat{Z}(s)|^2|\mathcal{F}_s^W]
+w_7^2|\widehat{\widetilde{Z}}(s)|^2+w_8^2\mathbb{E}[|\widehat{\widetilde{Z}}(s)|^2|\mathcal{F}_s^W].
\end{aligned}
\end{equation*}
Similarly, we can derive
\begin{equation*}
\begin{aligned}
|\widehat{\widetilde{\sigma}}(t)|^2
\leq&\kappa_1^2|\widehat{X}(s)|^2+\kappa_2^2\mathbb{E}|\widehat{X}(s)|^2+\kappa_3^2|\widehat{Y}(s)|^2
+\kappa_4^2\mathbb{E}[|\widehat{Y}(s)|^2|\mathcal{F}_s^W]\\
&+\kappa_5^2|\widehat{Z}(s)|^2+\kappa_6^2\mathbb{E}[|\widehat{Z}(s)|^2|\mathcal{F}_s^W]
+\kappa_7^2|\widehat{\widetilde{Z}}(s)|^2+\kappa_8^2\mathbb{E}[|\widehat{\widetilde{Z}}(s)|^2|\mathcal{F}_s^W].
\end{aligned}
\end{equation*}
By noticing $\mathbb{E}[\mathbb{E}[|V(s)|^2|\mathcal{F}_s^W]]=\mathbb{E}|V(s)|^2$, for $V(s)=\widehat{Y}(s),\widehat{Z}(s),\widehat{\widetilde{Z}}(s)$, and by combining above estimates and \eqref{e1}, we can obtain \eqref{E1}. Moreover, by applying It\^o's formula to $e^{-\bar{\lambda}_{1}(t-s)}e^{-\lambda s}|\widehat{X}(s)|^{2},s\in\lbrack 0,t]$, we get
\begin{equation*}
\begin{aligned}
&e^{-\lambda t}\mathbb{E}|\widehat{X}(t)|^{2}\\
=&-(\lambda -\bar{\lambda%
}_{1})\int_{0}^{t}e^{-\bar{\lambda}_{1}(t-s)}e^{-\lambda s}\mathbb{E}%
|\widehat{X}(s)|^{2}ds +2\mathbb{E}\int_{0}^{t}e^{-\bar{\lambda}_{1}(t-s)}e^{-\lambda s}\langle\widehat{X}(s),\widehat{b}(s)\rangle ds\\
&+\mathbb{E}%
\int_{0}^{t}e^{-\bar{\lambda}_{1}(t-s)}e^{-\lambda s}|\widehat{\sigma}
(s)|^{2}ds+\mathbb{E}%
\int_{0}^{t}e^{-\bar{\lambda}_{1}(t-s)}e^{-\lambda s}|\widehat{\widetilde{\sigma}}
(s)|^{2}ds,\\
\end{aligned}
\end{equation*}%
by repeating the previous produce, we derive \eqref{E2}.    \hfill$\square$
\end{proof}
\begin{remark}\label{remarkA1}
Denote  $||u(\cdot )||_{\lambda }^{2}:=\mathbb{E}\int_{0}^{T}e^{-\lambda s}|u(s)|^{2}ds$,
by integrating \eqref{E2} and using $\frac{1-e^{-\bar{\lambda}_{1}(T-s)}}{\bar{\lambda}_{1}}\leq\frac{1-e^{-\bar{%
\lambda}_{1}T}}{\bar{\lambda}_{1}}$, it follows
\begin{equation*}
\begin{aligned}
||X_{1}(\cdot )-X_{2}(\cdot )||_{\lambda }^{2}\leq &\frac{1-e^{-\bar{%
\lambda}_{1}T}}{\bar{\lambda}_{1}}[(\rho _2C_1+\rho _{3}C_{1}+w_3^2+w_4^2+\kappa_3^2+\kappa_4^2)||Y_{1}(\cdot )-Y_{2}(\cdot )||_{\lambda }^{2}\\
&+(\rho _4C_2+\rho _5C_{2}+w_5^2+w_6^2+\kappa_5^2+\kappa_6^2)||Z_{1}(\cdot )-Z_{2}(\cdot )||_{\lambda }^{2}\\
&+(\rho _6C_3+\rho _7C_{3}+w_7^2+w_8^2+\kappa_7^2+\kappa_8^2)||\widetilde{Z}_{1}(\cdot )-\widetilde{Z}_{2}(\cdot )||_{\lambda }^{2}].
\end{aligned}
\end{equation*}%
If let $t=T$ in \eqref{E2}, we have
\begin{equation*}
\begin{aligned}
&e^{-\lambda T}\mathbb{E}|X_{1}(T)-X_{2}(T)|^{2}\\
\leq &(1\vee e^{-\bar{%
\lambda}_{1}T})[(\rho _2C_1+\rho _{3}C_{1}+w_3^2+w_4^2+\kappa_3^2+\kappa_4^2)||Y_{1}(\cdot )-Y_{2}(\cdot )||_{\lambda }^{2}\\
&+(\rho _4C_2+\rho _5C_{2}+w_5^2+w_6^2+\kappa_5^2+\kappa_6^2)||Z_{1}(\cdot )-Z_{2}(\cdot )||_{\lambda }^{2}\\
&+(\rho _6C_3+\rho _7C_{3}+w_7^2+w_8^2+\kappa_7^2+\kappa_8^2)||\widetilde{Z}_{1}(\cdot )-\widetilde{Z}_{2}(\cdot )||_{\lambda }^{2}].
\end{aligned}
\end{equation*}%
Furthermore, if $\bar{\lambda}_1>0$ , above inequality is deduced to
\begin{equation*}
\begin{aligned}
e^{-\lambda T}\mathbb{E}|X_{1}(T)-X_{2}(T)|^{2}
\leq &[(\rho _2C_1+\rho _{3}C_{1}+w_3^2+w_4^2+\kappa_3^2+\kappa_4^2)||Y_{1}(\cdot )-Y_{2}(\cdot )||_{\lambda }^{2}\\
&+(\rho _4C_2+\rho _5C_{2}+w_5^2+w_6^2+\kappa_5^2+\kappa_6^2)||Z_{1}(\cdot )-Z_{2}(\cdot )||_{\lambda }^{2}\\
&+(\rho _6C_3+\rho _7C_{3}+w_7^2+w_8^2+\kappa_7^2+\kappa_8^2)||\widetilde{Z}_{1}(\cdot )-\widetilde{Z}_{2}(\cdot )||_{\lambda }^{2}].
\end{aligned}
\end{equation*}%
\end{remark}

For any fixed $X(\cdot)\in L_{\mathcal{F}_{t}^{W,\widetilde{W}}}^{2}(0,T;\mathbb{R}^{n})$, we define the mapping
\[\mathcal{V}_{2}:L_{\mathcal{F}_{t}^{W,\widetilde{W}}}^{2}(0,T;\mathbb{R}^{n})\longrightarrow L_{\mathcal{F}_{t}^{W,\widetilde{W}}}^{2}(0,T;\mathbb{R}^{m}\times \mathbb{R}^{m\times d}\times \mathbb{R}^{m\times d})\]
such that $\mathcal{V}_{2}(X(\cdot ))=(Y(\cdot ),Z(\cdot ),\widetilde{Z}(\cdot ))$, where $(Y(\cdot ),Z(\cdot ),\widetilde{Z}(\cdot ))$ solves the following backward SDE
\begin{equation}
Y(t)=~g(X(T),\mathbb{E}[X(T)])+\int_{t}^{T}f(s,\Theta(s))ds
-\int_{t}^{T}Z(s)dW(s)-\int_{t}^{T}\widetilde{Z}(s)d\widetilde{W}(s).\label{backward}
\end{equation}%
Under our assumptions, we can prove the well-posedness of \eqref{backward} by using similar arguments as in Darling and Pardoux \cite{Darling1997}. Therefore, the mapping $\mathcal{V}_{2}$ is well defined. Moreover, we have $\mathbb{E}\sup_{0\leq t\leq T}|Y(t)|^2<\infty$.
\begin{proposition}\label{Y}
Let \textup{(A1)},\textup{(A2)}, \textup{(A4)} and \textup{(A5)} hold. Assume that $(Y_i(\cdot),Z_i(\cdot),\widetilde{Z}_i(\cdot))$ is the solution of \eqref{backward} w.r.t. $X_i(\cdot)$, $i=1,2$, then there exist positive constants $K_1$, $K_2$ and $K_3$, s.t. for all $\lambda\in\mathbb{R}$, it follows that
\begin{equation}
\begin{aligned}
&e^{-\lambda t}\mathbb{E}|Y_{1}(t)-Y_{2}(t)|^{2}+\bar{\lambda}%
_{2}\int_{t}^{T}e^{-\lambda s}\mathbb{E}|Y_{1}(s)-Y_{2}(s)|^{2}ds\\
&\qquad+(1-\mu_4K_2-\mu_5K_2)\mathbb{E}\int_{t}^{T}e^{-\lambda s}|Z_{1}(s)-Z_{2}(s)|^{2}ds\\
&\qquad+(1-\mu_6K_3-\mu_7K_3)\mathbb{E}\int_{t}^{T}e^{-\lambda s}|\widetilde{Z}_{1}(s)-\widetilde{Z}_{2}(s)|^{2}ds \\
\leq &(\rho _{8}^{2}+\rho _{9}^{2})e^{-\lambda T}\mathbb{E}%
|X_{1}(T)-X_{2}(T)|^{2}+(\mu_1+\mu_2)K_1\int_{t}^{T}e^
{-\lambda s}\mathbb{E}|X_{1}(s)-X_{2}(s)|^{2}ds,\label{E3}
\end{aligned}
\end{equation}%
where $\bar{\lambda}_{2}=-\lambda -2\lambda _{2}-(\mu_1+\mu_2)K_1^{-1}-(\mu_4+\mu_5)K_2^{-1}-(\mu_6+\mu_7)K_3^{-1}-2\mu_3$.
Moreover,
\begin{equation}
\begin{aligned}
&e^{-\lambda t}\mathbb{E}|Y_{1}(t)-Y_{2}(t)|^{2}+(1-\mu_4K_2-\mu_5K_2)\mathbb{E}\int_{t}^{T}e^{-\bar{\lambda}_2(s-t)}e^{-\lambda s}|Z_{1}(s)-Z_{2}(s)|^{2}ds\\
&\qquad+(1-\mu_6K_3-\mu_7K_3)\mathbb{E}\int_{t}^{T}e^{-\bar{\lambda}_2(s-t)}e^{-\lambda s}|\widetilde{Z}_{1}(s)-\widetilde{Z}_{2}(s)|^{2}ds \\
\leq& (\rho _{8}^{2}+\rho _{9}^{2})e^{-\bar{\lambda}_2(T-t)}e^{-\lambda T}\mathbb{E}%
|X_{1}(T)-X_{2}(T)|^{2}\\
&+(\mu_1+\mu_2)K_1\int_{t}^{T}e^{-\bar{\lambda}_2(s-t)}e^{-\lambda s}\mathbb{E}|X_{1}(s)-X_{2}(s)|^{2}ds.\label{E4}
\end{aligned}
\end{equation}%
\end{proposition}

\begin{proof}
We denote $\widehat{\varphi}(s)=\varphi_1(s)-\varphi_2(s)$, for $\varphi=X,Y,Z,\widetilde{Z}$, and
\begin{equation*}
\begin{aligned}
\widehat{f}(s)=&f(s,X_1,\mathbb{E}[X_1],Y_1,\mathbb{E}[Y_1|\mathcal{F}^W_s],Z_1,\mathbb{E}[Z_1|\mathcal{F}^W_s],
\widetilde{Z}_1,\mathbb{E}[\widetilde{Z}_1|\mathcal{F}^W_s])\\
&-f(s,X_2,\mathbb{E}[X_2],Y_2,\mathbb{E}[Y_2|\mathcal{F}^W_s],Z_2,\mathbb{E}[Z_2|\mathcal{F}^W_s],
\widetilde{Z}_2,\mathbb{E}[\widetilde{Z}_2|\mathcal{F}^W_s]).\\
\end{aligned}
\end{equation*}%
Applying It\^o's formula to $e^{-\lambda t}|\widehat{Y}(t)|^{2}$, and then taking expectation, we have
\begin{equation}
\begin{aligned}
&e^{-\lambda t}\mathbb{E}|\widehat{Y}(t)|^{2}-\lambda
\int_{t}^{T}e^{-\lambda s}\mathbb{E}|\widehat{Y}(s)|^{2}ds+\mathbb{E}%
\int_{t}^{T}e^{-\lambda s}|\widehat{Z}(s)|^{2}ds+\mathbb{E}%
\int_{t}^{T}e^{-\lambda s}|\widehat{\widetilde{Z}}(s)|^{2}ds \\
=&e^{-\lambda T}\mathbb{E}|Y_{1}(T)-Y_{2}(T)|^{2}+2\mathbb{E}%
\int_{0}^{t}e^{-\lambda s}\langle \widehat{Y}(s),\widehat{f}(s)\rangle ds.\label{e2}
\end{aligned}
\end{equation}%
Using assumption \textup(A1) and Lipschitz condition of $f,g$, we obtain
\begin{equation*}
\begin{aligned}
&2\langle\widehat{Y}(s),\widehat{f}(s)\rangle \\\leq&(2\lambda_2+\mu_1K_1^{-1}+\mu_2K_1^{-1}+\mu_4K_2^{-1}+\mu_5K_2^{-1}+\mu_6K_3^{-1}+\mu_7K_3^{-1})|\widehat{Y}(s)|^2\\
&+\mu_1K_1|\widehat{X}(s)|^2+\mu_2K_1\mathbb{E}|\widehat{X}(s)|^2+\mu_4K_2|\widehat{Z}(s)|^2+\mu_5K_2\mathbb{E}[|\widehat{Z}(s)|^2|\mathcal{F}^W_s]\\
&+\mu_6K_3|\widehat{\widetilde{Z}}(s)|^2+\mu_7K_3\mathbb{E}[|\widehat{\widetilde{Z}}(s)|^2|\mathcal{F}^W_s]+2\mu_3|\widehat{Y}(s)||\mathbb{E}[\widehat{Y}(s)|\mathcal{F}^W_s]|,
\end{aligned}
\end{equation*}%
and
$
|\widehat{Y}(T)|^{2} =|g(X_{1}(T),\mathbb{E}[X_{1}(T)])-g(X_{2}(T),%
\mathbb{E}[X_{2}(T)])|^{2}
\leq \rho _{8}^{2}|\widehat{X}(T|^{2}+\rho _{9}^{2}\mathbb{E}%
|\widehat{X}(T)|^{2}.
$
Then from \eqref{e2} and above analysis, we can get \eqref{E3}.
Similarly, applying It\^o's formula to $e^{-\bar{\lambda}_{2}(s-t)}e^{-\lambda s}|\widehat{Y}(s)|^{2},s\in
\lbrack t,T]$, we have
\begin{equation}
\begin{aligned}
&e^{-\lambda t}\mathbb{E}|\widehat{Y}(t)|^{2}-(\lambda +\bar{\lambda}%
_{2})\int_{t}^{T}e^{-\bar{\lambda}_{2}(s-t)}e^{-\lambda s}\mathbb{E}%
|\widehat{Y}(s)|^{2}ds \\
&+\mathbb{E}\int_{t}^{T}e^{-\bar{\lambda}_{2}(s-t)}e^{-\lambda
s}|\widehat{Z}(s)|^{2}ds+\mathbb{E}\int_{t}^{T}e^{-\bar{\lambda}_{2}(s-t)}e^{-\lambda
s}|\widehat{\widetilde{Z}}(s)|^{2}ds \\
=&e^{-\bar{\lambda}_{2}(T-t)}e^{-\lambda T}\mathbb{E}%
|\widehat{Y}(T)|^{2} +2\mathbb{E}\int_{0}^{t}e^{-\bar{\lambda}_{2}(s-t)}e^{-\lambda
s}\langle \widehat{Y}(s),\widehat{f}(s)\rangle ds.\label{e4}
\end{aligned}
\end{equation}%
From the previous estimates and \eqref{e4}, we can prove \eqref{E4}.    \hfill$\square$
\end{proof}

\begin{remark}\label{remarkA2}
By integrating \eqref{E4} and choosing $K_2$, $K_3$ s.t. $0< K_2 < (\mu_4+\mu_5)^{-1}$ and $0< K_3 < (\mu_6+\mu_7)^{-1}$, we have $($noting that $\frac{1-e^{-\bar{%
\lambda}_{2}s}}{\bar{\lambda}_{2}}\leq \frac{1-e^{-\bar{%
\lambda}_{2}T}}{\bar{\lambda}_{2}}$, for any $s\in[0,T]$$)$
\begin{equation*}
\begin{aligned}
&||Y_{1}(\cdot )-Y_{2}(\cdot )||_{\lambda }^{2}\\\leq& \frac{1-e^{-\bar{%
\lambda}_{2}T}}{\bar{\lambda}_{2}}[(\rho _{8}^{2}+\rho
_{9}^{2})e^{-\lambda T}\mathbb{E}|X_{1}(T)-X_{2}(T)|^{2}+(\mu _{1}+\mu
_{2})K_1||X_{1}(\cdot )-X_{2}(\cdot )||_{\lambda }^{2}].
\end{aligned}
\end{equation*}%
In addition, let $t=0$ in \eqref{E4}, one can get
\begin{equation*}
\begin{aligned}
&||Z_{1}(\cdot )-Z_{2}(\cdot )||_{\lambda }^{2}+||\widetilde{Z}_{1}(\cdot )-\widetilde{Z}_{2}(\cdot )||_{\lambda }^{2}\\
\leq&\frac{(\rho _{8}^{2}+\rho _{9}^{2})e^{-\bar{\lambda}%
_{2}T}e^{-\lambda T}}{[(1-\mu_4K_2-\mu_5K_2)\wedge(1-\mu_6K_3-\mu_7K_3)](1\wedge e^{-\bar{\lambda}_{2}T})}\mathbb{E}%
|X_{1}(T)-X_{2}(T)|^{2}\\
&+\frac{(\mu _{1}+\mu
_{2})K_1(1\vee e^{-\bar{%
\lambda}_{2}T})}{[(1-\mu_4K_2-\mu_5K_2)\wedge(1-\mu_6K_3-\mu_7K_3)](1\wedge e^{-\bar{\lambda}_{2}T})}||X_{1}(\cdot )-X_{2}(\cdot )||_{\lambda }^{2}.
\end{aligned}
\end{equation*}%
Moreover, if $\bar{\lambda}_2>0$ and let $t=0$ in \eqref{E3}, we have
\begin{equation*}
\begin{aligned}
&||Z_{1}(\cdot )-Z_{2}(\cdot )||_{\lambda }^{2}+||\widetilde{Z}_{1}(\cdot )-\widetilde{Z}_{2}(\cdot )||_{\lambda }^{2}\\
\leq &\frac{(\rho _{8}^{2}+\rho _{9}^{2})e^{-\lambda
T}\mathbb{E}|X_{1}(T)-X_{2}(T)|^{2}+(\mu _{1}+\mu
_{2})K_1||X_{1}(\cdot )-X_{2}(\cdot )||_{\lambda }^{2}}{%
(1-\mu_4K_2-\mu_5K_2)\wedge(1-\mu_6K_3-\mu_7K_3)}.
\end{aligned}
\end{equation*}%
\end{remark}

Now, we will prove Theorem \ref{wellposeness} in the following.  

\emph{Proof of Theorem \ref{wellposeness}.}
Define a mapping $\mathcal{V}=\mathcal{V}_{2}\circ \mathcal{V}_{1}$, recall that $\mathcal{V}_1$ and $\mathcal{V}_2$ are given through \eqref{forward} and \eqref{backward}.
Then $\mathcal{V}$ is a mapping from $L_{\mathcal{F}_{t}^{W,\widetilde{W}}}^{2}(0,T;\mathbb{R}^{m}\times \mathbb{R}^{m\times d}\times \mathbb{R}^{m\times d})$ to itself. Next, we will prove that $\mathcal{V}$ is a contraction mapping under the norm $||\cdot||_\lambda$.

For any $(Y(\cdot ),Z(\cdot ),\widetilde{Z}(\cdot ))$, we set
\[X_{i}(\cdot )=\mathcal{V}_{1}(Y_{i}(\cdot ),Z_{i}(\cdot ),\widetilde{Z}_{i}(\cdot )), \]
\[(\bar{Y}_{i}(\cdot ),\bar{Z}_{i}(\cdot ),\bar{\widetilde{Z}}_{i}(\cdot ))=\mathcal{V(}(Y_{i}(\cdot ),Z_{i}(\cdot ),\widetilde{Z}_{i}(\cdot ))),\quad i=1,2.\]
With the help of inequalities in Remarks \ref{remarkA1} and  \ref{remarkA2}, we can derive
\begin{equation*}
\begin{aligned}
&||\bar{Y}_{1}(t)-\bar{Y}_{2}(t)||_{\lambda }^{2}+||\bar{Z}_{1}(t)-\bar{Z}%
_{2}(t)||_{\lambda }^{2}+||\bar{\widetilde{Z}}_{1}(t)-\bar{\widetilde{Z}}%
_{2}(t)||_{\lambda }^{2} \\
\leq &\Big\{\frac{1-e^{-\bar{\lambda}_{2}T}}{\bar{%
\lambda}_{2}}+\frac{1\vee e^{-\bar{\lambda}_{2}T}}{[(1-\mu_4K_2-\mu_5K_2)\wedge(1-\mu_6K_3-\mu_7K_3)](1\wedge e^{-\bar{\lambda}_{2}T})}\Big\} \\
&\quad\times \lbrack (\rho _{8}^{2}+\rho _{9}^{2})e^{-\lambda T}\mathbb{E}%
|X_{1}(T)-X_{2}(T)|^{2}+(\mu_1+\mu_2)K_1||X_{1}(t)-X_{2}(t)||_{\lambda }^{2}] \\
\leq &\Big\{\frac{1-e^{-\bar{\lambda}_{2}T}}{\bar{%
\lambda}_{2}}+\frac{1\vee e^{-\bar{\lambda}_{2}T}}{[(1-\mu_4K_2-\mu_5K_2)\wedge(1-\mu_6K_3-\mu_7K_3)](1\wedge e^{-\bar{\lambda}_{2}T})}\Big\} \\
&\quad\times \lbrack (\rho _{8}^{2}+\rho _{9}^{2})(1\vee e^{-\bar{\lambda}%
_{1}T})+(\mu_1+\mu_2)K_1\frac{1-e^{-\bar{\lambda}_{1}T}}{%
\bar{\lambda}_{1}}] \\
&\quad\times [(\rho _2C_1+\rho _{3}C_{1}+w_3^2+w_4^2+\kappa_3^2+\kappa_4^2)||Y_{1}(\cdot )-Y_{2}(\cdot )||_{\lambda }^{2}\\
&\quad+(\rho _4C_2+\rho _5C_{2}+w_5^2+w_6^2+\kappa_5^2+\kappa_6^2)||Z_{1}(\cdot )-Z_{2}(\cdot )||_{\lambda }^{2}\\
&\quad+(\rho _6C_3+\rho _7C_{3}+w_7^2+w_8^2+\kappa_7^2+\kappa_8^2)||\widetilde{Z}_{1}(\cdot )-\widetilde{Z}_{2}(\cdot )||_{\lambda }^{2}],
\end{aligned}
\end{equation*}%
thus the first assertion can be proved.

Now let us show the second assertion of Theorem \ref{wellposeness}. Since $2(\lambda_1+\lambda_2)<-2\rho_1-2\mu_3-(\mu_4+\mu_5)^2-(\mu_6+\mu_7)^2-w_1^2-w_2^2-\kappa_1^2-\kappa_2^2$, then we can choose $\lambda\in\mathbb{R}$, $0<K_2 <(\mu_4+\mu_5)^{-1}$, $0< K_3< (\mu_6+\mu_7)^{-1}$, and sufficiently large $C_1,C_2,C_3,K_1$ such that $\bar{\lambda}_{1}>0,\bar{\lambda}_{2}>0,1-\mu_4K_2-\mu_5K_2>0,1-\mu_6K_3-\mu_7K_3>0$.
Therefore, we have
\begin{equation*}
\begin{aligned}
&||\bar{Y}_{1}(t)-\bar{Y}_{2}(t)||_{\lambda }^{2}+||\bar{Z}_{1}(t)-\bar{Z}%
_{2}(t)||_{\lambda }^{2}+||\bar{\widetilde{Z}}_{1}(t)-\bar{\widetilde{Z}}%
_{2}(t)||_{\lambda }^{2} \\
\leq &\Big\{\frac{1}{\bar{%
\lambda}_{2}}+\frac{1}{(1-\mu_4K_2-\mu_5K_2)\wedge(1-\mu_6K_3-\mu_7K_3)}\Big\} \\
&\quad\times \lbrack (\rho _{8}^{2}+\rho _{9}^{2})e^{-\lambda T}\mathbb{E}%
|X_{1}(T)-X_{2}(T)|^{2}+(\mu_1+\mu_2)K_1||X_{1}(t)-X_{2}(t)||_{\lambda }^{2}] \\
\leq &\Big\{\frac{1}{\bar{%
\lambda}_{2}}+\frac{1}{(1-\mu_4K_2-\mu_5K_2)\wedge(1-\mu_6K_3-\mu_7K_3)}\Big\} \Big[ \rho _{8}^{2}+\rho _{9}^{2}+(\mu_1+\mu_2)K_1\frac{1}{%
\bar{\lambda}_{1}}\Big] \\
&\quad\times [(\rho _2C_1+\rho _{3}C_{1}+w_3^2+w_4^2+\kappa_3^2+\kappa_4^2)||Y_{1}(\cdot )-Y_{2}(\cdot )||_{\lambda }^{2}\\
&\quad+(\rho _4C_2+\rho _5C_{2}+w_5^2+w_6^2+\kappa_5^2+\kappa_6^2)||Z_{1}(\cdot )-Z_{2}(\cdot )||_{\lambda }^{2}\\
&\quad+(\rho _6C_3+\rho _7C_{3}+w_7^2+w_8^2+\kappa_7^2+\kappa_8^2)||\widetilde{Z}_{1}(\cdot )-\widetilde{Z}_{2}(\cdot )||_{\lambda }^{2}],
\end{aligned}
\end{equation*}%
which allows us to complete the proof.       \hfill$\square$

\end{document}